\documentclass[12pt]{amsart}
\usepackage{amsmath, amsthm, amscd, amssymb, amsfonts, latexsym}
\usepackage{fullpage}
\usepackage{bm}
\usepackage[all]{xy}
\usepackage{tikz}
\usepackage{fancybox}
\numberwithin{equation}{section}

      \def\cF{\mathcal F}      \def\cL{\mathcal L}            

\def\fp{\mathfrak{p}}

\def\fp{\mathfrak{p}}

\def\Z{{\mathbb Z}} \def\R{{\mathbb R}}  \def\N{{\mathrm N}} \def\C{{\mathbb C}} \def\Q{{\mathbb Q}} \def\P{{\mathbb P}}
\def\X{{\mathbb X}} \def\E{{\mathbb E}}

\newcommand{\PP}{\mathbb P}
\newcommand{\kommentar}[1]{}

\newcommand{\legendre}[2]{\left(\frac{#1}{#2}\right)}

\DeclareMathOperator{\re}{Re}
\DeclareMathOperator{\im}{Im}

\def \e{\mathbf{e}}

\renewcommand{\mod}[1]{\,{\rm mod}\,#1}
\renewcommand{\mod}[1]{\,(\mathrm{mod}\,#1)}

\def\le{\leqslant} \def\ge{\geqslant}

\newtheorem{lem}{Lemma}[section]
\newtheorem{prop}[lem]{Proposition}
\newtheorem{thm}[lem]{Theorem}

\newtheorem{conj}[lem]{Conjecture}
\theoremstyle{definition}

\newtheorem{rem}[lem]{Remark}

\definecolor{pink}{rgb}{1,.2,.6}
\definecolor{orange}{rgb}{0.7,0.3,0}
\definecolor{blue}{rgb}{.2,.6,.75}
\definecolor{green}{rgb}{.4,.7,.4}
\definecolor{purple}{RGB}{127,0,255}

\newcommand\cube{\begin{tikzpicture}[scale=2.3]
    \coordinate (A1) at (0, 0);
    \coordinate (A2) at (0, 0.1);
    \coordinate (A3) at (0.1, 0.1);
    \coordinate (A4) at (0.1, 0);
    \coordinate (B1) at (0.03, 0.03);
    \coordinate (B2) at (0.03, 0.13);
    \coordinate (B3) at (0.13, 0.13);
    \coordinate (B4) at (0.13, 0.03);

    \draw (A1) -- (A2);
    \draw (A2) -- (A3);
    \draw (A3) -- (A4);
    \draw (A4) -- (A1);
    \draw[densely dotted] (A1) -- (B1);
    \draw[densely dotted] (B1) -- (B2);
    \draw (A2) -- (B2);
    \draw (B2) -- (B3);
    \draw (A3) -- (B3);
    \draw (A4) -- (B4);
    \draw (B4) -- (B3);
    \draw[densely dotted] (B1) -- (B4);
\end{tikzpicture}}

\newcommand\tinycube{\begin{tikzpicture}[scale=1.3]
    \coordinate (A1) at (0, 0);
    \coordinate (A2) at (0, 0.1);
    \coordinate (A3) at (0.1, 0.1);
    \coordinate (A4) at (0.1, 0);
    \coordinate (B1) at (0.03, 0.03);
    \coordinate (B2) at (0.03, 0.13);
    \coordinate (B3) at (0.13, 0.13);
    \coordinate (B4) at (0.13, 0.03);

    \draw (A1) -- (A2);
    \draw (A2) -- (A3);
    \draw (A3) -- (A4);
    \draw (A4) -- (A1);
    \draw[densely dotted] (A1) -- (B1);
    \draw[densely dotted] (B1) -- (B2);
    \draw (A2) -- (B2);
    \draw (B2) -- (B3);
    \draw (A3) -- (B3);
    \draw (A4) -- (B4);
    \draw (B4) -- (B3);
    \draw[densely dotted] (B1) -- (B4);
\end{tikzpicture}}

\begin{document}

\title{Asymmetric Distribution of Extreme Values of Cubic $L$-functions at $s=1$}
\author{Pranendu Darbar, Chantal David, Matilde Lalin, Allysa Lumley}

\address{Pranendu Darbar: Max Planck Institute for Mathematics, Vivatsgasse 7, 53111 Bonn}
\email{darbarpranendu100@gmail.com}

\address{Chantal David: 
Department of Mathematics and Statistics, Concordia University, 1455 de Maisonneuve West, Montr\'eal, QC H3G 1M8, Canada }
\email{chantal.david@concordia.ca}

\address{Matilde Lal\'in: 
D\'epartement de math\'ematiques et de statistique,
                                    Universit\'e de Montr\'eal.
                                    CP 6128, succ. Centre-ville.
                                     Montreal, QC H3C 3J7, Canada}\email{matilde.lalin@umontreal.ca}

 \address{Allysa Lumley: Department of Mathematics and Statistics, York University, N520 Ross, 4700 Keele Street,
Toronto, ON M3J 1P3, Canada} \email{alumley2@yorku.ca}              

\begin{abstract}
We investigate the distribution of values of cubic Dirichlet $L$-functions at $s=1$. Following ideas of Granville and Soundararajan  for quadratic $L$-functions, we model the distribution of $L(1,\chi)$ by the distribution of random Euler products $L(1,\mathbb{X})$ for certain family of random variables $\mathbb{X}(p)$ attached to each prime. We obtain a description of the proportion of $|L(1,\chi)|$ that are larger or that are smaller than a given bound, and yield more light into the Littlewood bounds. Unlike
the quadratic case, there is an asymmetry between lower and upper
bounds for the cubic case, and small values are less probable than large values.
 
\end{abstract}
\date{\today}
\numberwithin{equation}{section}
\maketitle

\section{Introduction}
We study in this paper the distribution of the values of Dirichlet $L$-functions attached to cubic characters. Let $\chi$ be a primitive cubic Dirichlet character over $\Q$, and let
\begin{align*}
L(s, \chi) = \sum_{n=1}^\infty \frac{\chi(n)}{n^s}.
\end{align*}
We are interested in the distribution of the special values $|L(1, \chi)|$ as $\chi$ varies over the family $\mathcal{F}_3$ of cubic primitive characters over $\Q$. 
The approach of this work is to compare the distribution of values for $|L(1,\chi)|$ with the distribution of values of a random Euler product, $|L(1,\mathbb{X})|$, where  $L(1,\mathbb{X})=\prod_{p}\left(1-\frac{\mathbb{X}(p)}{p}\right)^{-1}$ and the $\mathbb{X}(p)$ are independent random variables that take the values $0,1,\omega_3 := e^{\frac{2 \pi i}{3}}$ and $\omega_3^2$ with suitable probabilities.  
The main motivation for this work is in determining the extreme values that $|L(1,\chi)|$ can take. 

This problem has been studied thoroughly in the case of the quadratic characters, with the pioneering work of Granville and Soundararajan \cite{GS}  describing the distribution of extreme values for $L(1,\chi_d)$ where $\chi_d$ varies over quadratic characters,   
in order to gain understanding of the well-known discrepancy between the extreme values that $L(1,\chi_d)$ may exhibit (the $\Omega$-results of Chowla, described below) and the conditional bounds on these extreme values (the $O$-results of Littlewood). 

For a quadratic character $\chi_d$ of conductor $|d|$, it was shown by Littlewood \cite{Littlewood-1928} (assuming GRH, the Generalized Riemann Hypothesis) that
\begin{align}\label{LWbound1}
\left( \frac{1}{2} + o(1) \right) \frac{\zeta(2)}{e^\gamma \log_2{|d|}} \leq   L(1, \chi_d)  \leq (2 + o(1)) e^\gamma \log_2{|d|} ,
\end{align}
where here and throughout, $\gamma$ is the Euler-Mascheroni constant, and the notation $\log_j$ represents the $j$-fold iterated logarithm, so that $\log_2 |d|=\log\log |d|$.
On the other hand, under the same hypothesis, Littlewood also established that there are infinitely many fundamental discriminants $d$ for which 
\begin{align}\label{LWbound2} L(1, \chi_d) \geq (1 + o(1)) e^\gamma \log_2{|d|},\end{align}
and there are infinitely many fundamental discriminants $d$ such that 
\begin{align}\label{LWbound3}
L(1, \chi_d) \leq (1 + o(1)) \zeta(2)/(e^\gamma \log_2{|d|}).
\end{align}
In \cite{chowla48}, Chowla removed the assumption of GRH on the $\Omega$-result \eqref{LWbound2}. If we compare the bounds in \eqref{LWbound1} with those obtained in \eqref{LWbound2} and \eqref{LWbound3}, we can see there is a discrepancy in the coefficient of the main term. In \cite{GS}, the authors provide asymptotics for the probability that $L(1,\chi_d) > e^{\gamma}\tau$, and for the probability the $L(1,\chi_d) <
\frac{\zeta(2)}{e^{\gamma}\tau}$ uniformly in a wide range of $\tau$.
The uniformity of their results provides evidence that the $\Omega$-results of Chowla may represent the true nature of these extreme values, but
falls just short of determining which coefficient is the correct one.

We now turn to the extreme values of cubic characters. 
Our first result is a conditional $O$-result in the style of Littlewood for cubic characters.
\begin{prop}\label{thm:littlewood} Let $\chi$ be a primitive character of order $3$ and  conductor $q$. Assume GRH for $L(s,\chi)$.   Then we have 
$$
\left( \frac{1}{\sqrt{2}} + o(1) \right) \left(\frac{\zeta(3)}{e^{\gamma}\log_2 q}\right)^{\frac{1}{2}}
\leq |L(1,\chi)| \leq  \left( 2 + o(1) \right) e^\gamma \log_2 q.
$$
 \end{prop}

The upper bound is not new, as the original proof of Littlewood holds for any  primitive character $\chi$ of conductor $q$, but to our knowledge, the lower bound does not appear in the literature.  Proposition \ref{thm:littlewood} is a particular case of Proposition \ref{thm:littlewood-gen} which holds for characters of prime order $\ell \geq  3$.
We also obtain $\Omega$-results, namely Theorem \ref{maxmin value under GRH}, which exhibits a constant discrepancy similar to the case of quadratic characters: 1 versus 2 for the upper bound and 1 versus $\frac{1}{\sqrt{2}}$ for the lower bound. Other $\Omega$-results for the large values of characters of order $\ell \geq 2$ as \cite[Theorem 1.2]{Lam2017} (for the upper bound) exhibit the same constant discrepancy.

Our main theorems describe the distribution of extreme values for $|L(1,\chi)|$ for a family of cubic characters. Limiting distribution results for cubic characters over $\Q(\sqrt{-3})$ were obtained by Akbary and Hamieh \cite{AkbaryHamieh, AkbaryHamieh2}.  To our knowledge,  our results are the first describing the distribution of extreme values for cubic characters over $\Z$,  by considering the tail of the distribution,  and represent the first family where there is an {\it asymmetric distribution}. 
The origin of this asymmetric distribution is simple: the minimum value of $|L(1,\chi)|$ for cubic characters is larger than the corresponding minumum for quadratic characters. This can already be seen at the level of each individual factor in the Euler product, since  $\chi(p)=\omega_3^h$ for some $h=0, 1,$ or  $2$ and then we have 
\[\left| 1-\frac{\chi(p)}{p}\right| =\left( 1-\frac{2\cos\left(\frac{2\pi h}{3}\right)}{p}+\frac{1}{p^2}\right)^{1/2}.\]
While the minimum value is achieved at  $h=0$, the maximum value is given at  $h=1$ or $h=2$  and therefore 
\[\left| 1-\frac{1}{p}\right| \leq \left| 1-\frac{\chi(p)}{p}\right| \leq \left| 1+\frac{1}{p}+\frac{1}{p^2}\right|^{1/2},\]
which explains the asymmetry. 

See also \cite[Exercise 4, p.~366]{MV-book} for the unconditional lower bound $|L(1, \chi)|\gg \frac1{\sqrt{\log q}}$ for $\chi$ cubic.

As in \cite{GS}, the range of uniformity we are able to achieve leads us to believe that the  $\Omega$-results of Theorem \ref{maxmin value under GRH}  may represent the true nature of the extreme values in cubic families.

In order to precisely describe our results, we need some notation. For $X$ large, let 
 $\mathcal{F}_3(X)$ denote the subset of $\mathcal{F}_3$ of cubic characters with conductor bounded by $X$ and let
\begin{align*} 
\phi_X(\tau)&:= \P \left( |L(1, \chi)| >e^{\gamma}\tau \right) := \frac{1}{|\mathcal{F}_3(X)|}\sum_{\substack{\chi\in \mathcal{F}_3(X)\\ |L(1, \chi)|>e^{\gamma}\tau}}1 \\
\psi_{X}(\tau) &:= \P \left( |L(1, \chi)|  <  \left(\frac{\zeta(3)}{e^{\gamma}}\right)^{\frac{1}{2}}\frac{1}{\tau} \right) := \frac{1}{|\mathcal{F}_3(X)|}\sum_{\substack{\chi\in \mathcal{F}_3(X)\\ |L(1, \chi)| <  \left(\frac{\zeta(3)}{e^{\gamma}}\right)^{\frac{1}{2}}\frac{1}{\tau}
}} 1.
\end{align*}
To study these probabilities, we use the method of moments, and we show that the 
 complex moments of $|L(1,\chi)|$ agree with the expectations of random Euler products $L(1,\mathbb{X})$ 
  for a large range of values. 

We now define the independent random variables $\X(p)$ given by 
\begin{equation}\label{eq:Xp2}
\mathbb{X}(p)=\begin{cases}
                 1 & \mbox{with probability} = \frac{1}{3}, \\ 
                 \omega_3 &  \mbox{with probability} = \frac{1}{3},\\
                 \omega_3^2 &  \mbox{with probability} = \frac{1}{3},
                \end{cases}
                \end{equation}
                when $p \equiv 2 \mod 3$ or $p=3$, and by
\begin{equation}\label{eq:Xp1}
\mathbb{X}(p)=\begin{cases}
                 0 & \mbox{with probability} = \frac{2}{p+2}, \\ 
                 1 & \mbox{with probability} = \frac{p}{3(p+2)}, \\ 
                 \omega_3 &  \mbox{with probability} =  \frac{p}{3(p+2)},\\
                 \omega_3^2 &  \mbox{with probability} =  \frac{p}{3(p+2)},
                \end{cases}
                \end{equation}
                when $p\equiv 1 \mod{3}$.
                We will see in Section 2 why these random variables are naturally associated to the cubic family $\cF_3$.

Let $n=p_1^{a_1}p_2^{a_2}\cdots p_k^{a_k}$ be the prime power factorization of $n$. We extend the definition of $\mathbb{X}$ by multiplicativity
$$\mathbb{X}(n)=\mathbb{X}(p_1)^{a_1}\mathbb{X}(p_2)^{a_2}\cdots \mathbb{X}(p_k)^{a_k},$$
and we define
\begin{align*} 
L(1,\mathbb{X})=\sum_{n=1}^{\infty}\frac{\mathbb{X}(n)}{n}=\prod_{p \text{ prime}}\left(1-\frac{\mathbb{X}(p)}{p}\right)^{-1}, 
 \end{align*}
where both the series and the product are almost surely convergent by Lemma \ref{expectation}.
Before describing how well this model approximates the distribution of $|L(1,\chi)|$, it is useful to understand the behaviour of the distribution of $|L(1,\mathbb{X})|$.  To this end, we define,  for $\tau>0$, 
\begin{align}
\Phi(\tau) &:= \P \left( |L(1, \X)| >  e^\gamma \tau \right), \label{def:Phi}\\
\Psi(\tau) &:=  \P \left( |L(1, \X)| < \left(\frac{\zeta(3)}{ e^\gamma}\right)^{\frac{1}{2}} \frac{1}{\tau} \right).\label{def:Psi}
\end{align}
We obtain the following asymptotic behaviour for $\Phi(\tau)$ and $\Psi(\tau)$, which are each decaying doubly exponentially, although with different rates:
\begin{thm} \label{thm-RV}
	For large $\tau$, we have
	\begin{align*}
	\Phi(\tau)=\exp\left(-\frac{2e^{ \tau-C_{\max}}}{\tau}\left(1+O(\tau^{-1})\right)\right)
	\end{align*}
	and
	\[
	\Psi(\tau)=\exp\bigg(-\frac{e^{\tau^2-C_{\min}}}{\tau^2} \left(1+O(\tau^{-1})\right)\bigg),
	\]
	where $C_{\max}\approx 0.98727\dots $ and $C_{\min}\approx 1.40459\dots$ are defined in \eqref{constantC0} and \eqref{minimum value} respectively.
\end{thm}

We then show that the distribution of $L(1, \chi)$ over $\chi\in \mathcal{F}_3(X)$ is well approximated by the distribution for the random Euler product $L(1, \X)$, in a large range of $\tau$.

\begin{thm}\label{theorem-relating-to-RV} 
	Let $X$ be large.
	Then, uniformly in the range $1 \leq  \tau \leq  \log_2 X- \log_3 X-\log_4 X-2$, we have
	\begin{align*}
	\phi_X(\tau) &=\exp\left(-\frac{2e^{ \tau-C_{\max}}}{\tau}\left(1+O(\tau^{-1/2})\right)\right), \\
	\end{align*}
	and uniformly in the range $1 \leq  \tau \leq  \sqrt{\log_2 X- \log_3 X-\log_4 X-2}$, 
	\begin{align*}
	\psi_X(\tau) &=\exp\left(-\frac{e^{ \tau^2-C_{\min}}}{\tau^2}\left(1+O(\tau^{-1/2})\right)\right), 
	\end{align*}
	where $C_{\max}$ and $C_{\min}$ are as in Theorem \ref{thm-RV}.
\end{thm}

If we compare Theorem \ref{theorem-relating-to-RV} with the other distribution and extreme values results found in the literature, for families of quadratic characters \cite{GS, DL, lumley-1line}, for $|\zeta(1+it)|$ \cite{GS06}, for all characters modulo $q$ \cite{LLS}, or for $L$-functions of automorphic forms on $\mathrm{GL}_n$ \cite{Lumley-edge}, we see that the probability distributions for cubic characters are different,
as $\phi_X(\tau)$ and $\psi_X(\tau)$ decay with different rates. Since $\psi_X(\tau)$ decays faster, the small values are less probable, which is reflected in the $O$-results and $\Omega$-results for small values of cubic characters, in Proposition \ref{thm:littlewood} and Theorem \ref{maxmin value under GRH}. The constants $C_{\max}$ and $C_{\min}$ are also different than the constant $C_1$ which appears consistently in the articles on quadratic families mentioned above.

We remark that the ranges of uniformity for $\phi_X(\tau)$ and  $\psi_X(\tau)$ in Theorem \ref{theorem-relating-to-RV} are comparable to \cite[Theorem 1]{GS}, taking into account the increased rate of decay of Theorem \ref{thm-RV} in the case of $\Psi(\tau)$.

Theorem  \ref{theorem-relating-to-RV} can be improved under GRH.
	\begin{thm}\label{main distribution result under GRH} Assume GRH for Hecke $L$-functions over $\Q(\omega_3)$.
	Let $X$ be large and let $e^{10}\leq A \leq (\log_2 X)^C$ be a real number. Then, uniformly in the range $\tau \leq  \log_2 X+\log_3 X -\log_2 A-37+C_{\max}+o(1)$, we have
    \begin{align*}
    \phi_X(\tau)=\exp\left(-\frac{2e^{\tau-C_{\max}}}{ \tau}\left(1+O(\tau^{-1/2}+A^{-1})\right)\right)
    \end{align*}
    and, uniformly in the range $\tau \leq  \sqrt{\log_2 X+\log_3 X -\log_2 A-37+C_{\min}+o(1)}$, we have
    \[
    \psi_X(\tau)=\exp\left(-\frac{e^{\tau^2-C_{\min}}}{ \tau^2}\left(1+O(\tau^{-1/2}+A^{-1})\right)\right),
    \]
    where $C_{\max}$ and $C_{\min}$ are defined in \eqref{constantC0} and \eqref{minimum value} respectively.
	\end{thm}

The above result gives an asymptotic formula if $A=A(X)$ is chosen as an arbitrary function of $X$ such that $A(X)\rightarrow \infty$ as $X \rightarrow \infty$.

We now move to the $\Omega$-results. 

\begin{thm}[Detecting maximum and minimum values]\label{maxmin value under GRH}
	 Assume GRH for Hecke $L$-functions over $\Q(\omega_3)$.  For  $X$ large, there are $\gg X^{\frac{1}{2}}$ cubic characters with prime conductor bounded by $X$ such that
	 \begin{equation} \label{eq:maximum value under GRH}
	 |L(1, \chi)|\geq e^{\gamma}\left(\log_2 X+\log_3 X-\log(2\log 3)+o(1)\right).   
	 \end{equation}
	 For $X$ large, there are $\gg X^{\frac{1}{2}}$ cubic characters  with prime conductor bounded by $X$ such that
	\begin{equation}	\label {eq:minimum value under GRH}
	|L(1, \chi)|\leq \frac{\zeta(3)^{\frac12}}{(e^{\gamma}\log_2 X+\log_3 X-\log(2\log 3)+o(1))^{\frac12}}.
	\end{equation}
\end{thm}
Other $\Omega$-results for the large values of characters of order $3$ can be found in \cite[Theorem 2.1]{Lam2017}, where the
proportion of characters is $X^{1-\epsilon}$, but the lower bound is smaller, containing only the first term $\log_2{X}$. We are able to get an additional $\log_3{X}$ because we are using the distribution. We are then led to the following conjecture  about the precise size of the maximal and minimal order for $L$-functions associated with cubic characters, which is the equivalent of a conjecture of \cite{GS} for quadratic characters, refining the  conjectures of \cite{MV}  \cite[Conjecture 2]{GS}. A similar conjecture for the large values of $|\zeta( 1 + it) |$ can be found in \cite{GS06}. For characters of order $\ell \geq 3$, a conjecture for the leading term  $\log_2X$ for the large values can be found in
\cite{Lam2017}.
 
\begin{conj}\label{precise-min-max} Let $\cF_3(X)$ be the family of primitive cubic characters over $\Q$ with conductor bounded by $X$. Then, 
\[\max_{ \chi \in \mathcal{F}_3(X)}|L(1,\chi)|=e^{\gamma}(\log_2X+\log_3X+C_{\max}-\log 2+o(1)),
\]
and 
\[\min_{\chi\in\mathcal{F}_3(X)}|L(1,\chi)|=\left(\frac{\zeta(3)}{e^{\gamma}(\log_2 X+\log_3 X+C_{\min}+o(1))}\right)^{\frac{1}{2}}, 
 \]
where $C_{\max}$ and $C_{\min}$ are as in Theorem \ref{thm-RV}.
\end{conj}
We reach this conjecture by assuming that the range for the upper bound (respectively lower bound) in Theorem \ref{main distribution result under GRH} can be sufficiently extended so that one can replace the value of $\tau$ in the expression for $\phi_X(\tau)$ (resp.~$\psi_X(\tau)$) by  $\tau_{\max}=\log_2X+\log_3X+C_{\max}-\log 2+\delta$   (resp.~$\tau_{\min}=\sqrt{\log_2 X+\log_3 X+C_{\min}+\delta}$) for some constant $\delta\geq 0$. This results in
    \[  \phi_X(\tau_{\max})\leq X^{-e^\delta} \left( 1 + o(1) \right) , \qquad \psi_X(\tau_{\min})\leq X^{-e^\delta}  \left( 1 + o(1) \right) .\]
Since $X^{-e^\delta}<\frac{1}{X}$ for $\delta>0$ and the total number of characters is of order $X$, this suggests that we do not cross the barrier of $\delta>0$ when considering $\max_{ \chi \in \mathcal{F}_3(X)}|L(1,\chi)|$ and $\min_{\chi\in\mathcal{F}_3(X)}|L(1,\chi)|$. 

Furthermore, Theorem \ref{maxmin value under GRH} comes (under GRH) very close to exhibiting $X^{\frac{1}{2}}$ characters such that $|L(1,\chi)|$ is close to the extreme values predicted by Conjecture \ref{precise-min-max}, since  $-\log{(2 \log 3)} = -0.78719 \dots$ and $C_{\max}-\log 2 =0.29412\dots$. The same observation (with the appropriate modifications) holds for the small values in  Theorem \ref{maxmin value under GRH}.

Notice that our results in Theorem \ref{main distribution result under GRH} prove (under GRH) some conjectures that are cubic analogues to those of Montgomery and Vaughan \cite{MV} (more precisely,  Theorem \ref{main distribution result under GRH} implies an analogue of Conjecture 1 and the upper bound on Conjecture 2 in the notation of \cite{GS}). Namely, using the maximal allowable value $\tau$ in Theorem \ref{main distribution result under GRH} implies that 
the proportion of characters $\chi \in \cF_3(X)$ such that
\begin{equation*} 
|L(1,\chi)| > e^{\gamma}\left( \log_2X+\log_3X - \log_2 A - 37 + C_{\max} + o(1) \right) 
\end{equation*}
 is 
$$
\exp \left( - \frac{2 e^{\log_2X+\log_3X - \log_2 A - 37 + o(1)}}{\log_2X} \right).$$
If we replace by  $\tau =\log_2 X$ (which is in the range of Theorem \ref{main distribution result under GRH}), we get that the proportion of  characters $\chi \in \cF_3(X)$  such that $|L(1,\chi)| > e^{\gamma} \log_2 X$ is 
both  $>\exp \left( - C\frac{ \log{X} }{\log_2 X} \right)$ and $<\exp \left( - c\frac{ \log{X} }{\log_2 X} \right)$ for some appropriate constants $0<c<C<\infty$, thus leading (under GRH) to a cubic analogue of Conjecture 1 in \cite{GS}.

Similarly, if we take $\tau=\log_2X+\log_3X - \log_2 A - 37 + C_{\max}$ with $A$ to be a very large constant (that is, sufficiently large to absorb the constant in the error term, but still a constant) then we  conclude that the proportion of  characters $\chi \in \cF_3(X)$  such that $|L(1,\chi)| > e^{\gamma} (\log_2 X+\log_3 X)$ is bounded above by $X^\Theta$ with some $-1<\Theta<0$, and this leads to the upper bound of the corresponding cubic  analogue to Conjecture 2 in \cite{GS}.

Less
attention has been paid to the  small values in the literature, maybe because the arguments are completely symmetric to the large values for quadratic characters. However, as we have observed in Proposition \ref{thm:littlewood} and Theorems  \ref{thm-RV}, \ref{theorem-relating-to-RV}, \ref{main distribution result under GRH}, and \ref{maxmin value under GRH}, we see an asymmetric behaviour for the family of cubic characters. 
Of course, all these results are related: the minimum value of $|L(1,\chi)|$ for cubic characters is much larger than what can be found for the quadratic case, which fits perfectly the faster decay for the occurrence of small values of Theorem \ref{theorem-relating-to-RV}.

There have been several papers working toward proving statements similar to  Conjecture \ref{precise-min-max}
for different families of $L$-functions,  including the work of Bondarenko and Seip \cite{BS2017} (for the Riemann zeta function on the critical line), Aistleitner et al. \cite{AMM2019} (for the Riemann zeta function on the $1$-line) and Aistleitner et al. \cite{AMMP2019} (for $L(1,\chi)$ with $\chi$ taken over the family of primitive characters modulo $q$), making use of a resonator method.
Again, these results just fall short of confirming the conjectures. 

\subsection{Generalizations to any prime order $\ell$}

Since the distribution of extreme values, and in particular of small values, is very different for the family of cubic characters than other families studied in the literature (as quadratic characters or all character of conductor $q$), it is interesting to speculate on what would happen for characters of prime order $\ell > 3$. One can see how most of the results of the present paper could be considered in that context, first by generalizing the sieve of Lemma 
\ref{Lem41GS} to start the computations, which is not trivial, and which would now take place over $\Q(\omega_\ell)$, creating obvious complications (for example, this is not in general a principal ideal domain).  Granville and Lamzouri \cite{GranvilleLamzouri} developed a model which can be applied to number theoretic questions about large values of different families of $L$-functions that can be modelled by almost independent random variables. This model could potentially be applied to describe the distribution of large values of $|L(1, \chi)|$ where $\chi$ is a character of order $\ell$.  We hope to address some of these questions in the future. As a starting point, some of the results of this paper can be easily generalized from cubic to characters of order $\ell$, as the $O$-results of Littlewood
and the $\Omega$-results. 
We will prove these statements in Section \ref{sec:omega}.

\begin{prop}\label{thm:littlewood-gen}  Let $\ell$ be an odd prime and $\chi$ be a character of order $\ell$ and  conductor $q$. Assume GRH for $L(s,\chi)$. Then we have 
 \begin{equation*}
 \frac{C_\ell}{\left(2e^{\gamma}\log_2 q\right)^{\cos\left(\frac{\pi}{\ell}\right)}}(1+o(1)) 
\leq |L(1,\chi)| \leq  2e^\gamma \log_2 q (1+o(1)),\end{equation*}
 where 
 \begin{align} \label{def-Cell}
 C_\ell=\prod_{p}\left[\left(1-\frac{1}{p}\right)^{-\cos\left(\frac{\pi}{\ell}\right)}\left(1+\frac{2\cos\left(\frac{\pi}{\ell}\right)}{p}+\frac{1}{p^2}\right)^{-1/2}\right].\end{align}
\end{prop}

\begin{rem}
 For $\ell=3$, we have $C_3=\zeta(3)^{\frac{1}{2}}$, and the bounds
  \[\left(\frac{\zeta(3)}{2e^{\gamma}\log_2 q}\right)^{\frac{1}{2}} \;(1+o(1)) 
\leq |L(1,\chi)| \leq  2e^\gamma \log_2 q (1+o(1)).\]
In addition, $\lim_{\ell \rightarrow \infty} C_\ell=\zeta(2)$, approaching the original bound of Littlewood. 
\end{rem}

\begin{rem} The fact that large values of Dirichlet characters $L$-functions at $s=1$ behave differently if the character is quadratic or of odd order was also noticed in 
\cite{GS-JAMS}, where the authors show that the P\'olya-Vinogradov bound can be improved for odd order characters. In \cite[Theorem 1]{GS-JAMS}, the authors prove that if
$\chi$ is a primitive character modulo $q$ of odd order $g$, then 
\begin{align*}
\max_{x} \left|\sum_{n\leq x}\chi(n) \right| \ll_g\sqrt{q}(\log q)^{1-\delta_g+o(1)}  \text{ with } \delta_g=1-\frac{g}{\pi}\sin\frac{\pi}g.
\end{align*}
The exponent approaches $1$ at a rate of $O(1/g^2)$ as $g$ grows (which is the P\'olya--Vinogradov bound for quadratic characters as $g$ grows),
at the same rate that $C_\ell$ is approaching the lower bound for quadratic characters. 
\end{rem}

 \begin{thm}[Detecting minimum value]\label{minimum value under GRH-gen} Let $\ell$ be a prime.  Assume GRH for Hecke $L$-functions over $\Q(\omega_3)$.
	 For  all large $X$, there are $\gg X^{\frac{1}{2}}$  characters  of order $\ell$ with prime conductor bounded by $X$ such that
	\begin{equation*}	
	|L(1, \chi)|\leq \frac{C_\ell}{(e^{\gamma}\log_2 X+\log_3 X-\log(2\log \ell )+o(1))^{\cos \left(\frac{\pi}{\ell}\right)}},
	\end{equation*}
	where $C_\ell$ is given by \eqref{def-Cell}.
\end{thm}
The $\Omega$-results are obtained  by following a pretentious approach of Granville and Soundararajan, forcing the characters $\chi$ to minimize the value of $\re(\chi)$. More precisely, we count the characters $\chi$ for which $\chi(p)=e^{\frac{(\ell-1)\pi i}{\ell}}$ or $e^{\frac{(\ell+1)\pi i}{\ell}}$ for the primes $p$ essentially below $\log X \log_2X$. Thus, the somewhat surprising appearance of the exponent $\cos \left(\frac{\pi}{\ell}\right)$ is explained by this 
 strategy of minimizing $\re(\chi(p))$.

As in the cubic case, there is a discrepancy between the constants of Proposition \ref{thm:littlewood-gen} (which is $\left(1/2\right)^{\cos(\pi/\ell)}$) and Theorem \ref{minimum value under GRH-gen} (which is 1). Interestingly, 
 $\left(1/2\right)^{\cos(\pi/\ell)} \rightarrow 1/2$,  which is the discrepancy found for quadratic $L$-functions.
 We could speculate that also in this case, the true nature of the extreme values is given by the minimum values of Theorem \ref{minimum value under GRH-gen}  and not the $O$-results,
 and
 \[
	\min_{\chi\in\mathcal{F}_{\ell}(X)}|L(1,\chi)|=\frac{C_{\ell}}{\left(e^{\gamma}(\log_2 X+\log_3 X+C_{\text{min}, \ell}+o(1))\right)^{\cos\left(\frac{\pi}{\ell}\right)}}, 
		\] 
	where  $C_\ell$ is given by \eqref{def-Cell} and $C_{\text{min}, \ell}$ is a constant that can be explicitly determined from the corresponding random model.
 Computing the distribution would shed more light on this question.
We expect the distribution function in the lower bound $\psi_X$ for general $\ell$ to take the shape
\[ 
    \psi_X(\tau)=\exp\bigg(-\frac{e^{\tau^{a(\ell)}-C_{\text{min},\ell}}}{ \tau^{a(\ell)}}\left(1+o(1)\right)\bigg).
    \]
where $a(\ell) = \left( {\cos\left(\frac{\pi}{\ell}\right)} \right)^{-1}$.

The average over the family of Section \ref{Section2}, which is needed to determine the random variables $\X(p)$, can also be done easily for the general family $\cF_\ell$ of characters of order $\ell$, which we will do.

\subsection{Organization of the paper}

This article is organized as follows.  In Section \ref{section-average}, we define the family of cubic characters, present the relevant background, and compute the average character value over the family. This leads to the definition of the random variables \eqref{eq:Xp2} and \eqref{eq:Xp1}. In Section \ref{section-largemoments}, we prove Theorem \ref {Asympformula}, estimating the complex moments in our family with the moments of the random variables. In Section \ref{working with the RV}, we prove Theorems \ref{thm-RV} and  \ref{theorem-relating-to-RV}. We consider $\Omega$-results in Section \ref{sec:omega}, where we prove Eq.~\eqref{eq:maximum value under GRH} in Theorem \ref{maxmin value under GRH},
Propositions \ref{thm:littlewood-gen}, and Theorem  \ref{minimum value under GRH-gen} (Eq.~\eqref{eq:minimum value under GRH} in Theorem \ref{maxmin value under GRH} is a particular case of Theorem \ref{minimum value under GRH-gen}).
Finally,  we prove Theorem \ref{main distribution result under GRH} in Section \ref{sec:1.4}.

\section*{Acknowledgements} The authors are grateful to Andrew Granville, Youness Lamzouri and K. Soundararajan for helpful discussions and to the referee for several suggestions and corrections that greatly improved the exposition of the paper. This work is supported by the Natural Sciences and Engineering Research Council of Canada [DG-155635-2019 to CD, DG-355412-2022 to ML, PDF-532937-2019 to AL], the Fonds de recherche du Qu\'ebec - Nature et technologies [Projet de recherche en \'equipe 300951 to CD and ML], the Research Council of Norway [Grant 275113 through the Alain Bensoussan Fellowship Programme of the European Research Consortium for Informatics and Mathematics to PD].

\section{Average over the character family} \label{Section2}

\label{section-average}
Let $\ell$ be a fixed odd prime, and let  $\cF_\ell$ be the set of primitive characters of order $\ell$ over $\Q$ with conductor coprime to $\ell$, i.e.,
\begin{align}  \label{defF} \cF_\ell = \left\{ \chi = \chi_{p_1}^{e_1} \;  \chi_{p_2}^{e_2} \cdots  \chi_{p_s}^{e_s} \;:\; p_i \;\text{distinct primes}, \;p_i \equiv 1 \mod \ell, \; e_i \in \{1, \dots, \ell-1 \}   \right\}, \end{align}
where $\chi_p$ is the $\ell$th residue symbol modulo $p$ defined by
\begin{align*}
\chi_p(\alpha) \equiv \alpha^{(p-1)/\ell} \mod p.
\end{align*}
Since there is no canonical choice of a primitive $\ell$th root of unity modulo $p$ in $\Q$, there is no canonical choice of $\chi_p$ versus any  $\chi_p^h$ with $1< h \leq \ell-1$. 

The conductor of each character $\chi_{p_1}^{e_1} \;  \chi_{p_2}^{e_2}  \cdots  \chi_{p_s}^{e_s}$ in $\cF_\ell$ is $p_1 \cdots p_s$, where the $p_i$ are distinct primes congruent to $1 \mod \ell$, and for each $p_1 \cdots p_s$  there are $(\ell-1)^s$ primitive cubic characters with this conductor. 
We denote by $\cF_\ell(X)$ the subset of $\cF_\ell$ consisting of characters of conductor $\leq X$.
Let $\omega_\ell$ denote a primitive $\ell$th root of unity in $\C$.

\begin{prop}  \label{prop-gives-RVs}
We have
\begin{align*}
\sum_{\chi \in \cF_\ell(X)} 1 &= K_\ell X + O \left( X^{\frac{\ell+2}{\ell+5}+\epsilon}\right), \end{align*}
and more generally for $m  \in \Z, m \geq 1$ and $m$ an $\ell$th power, 
\begin{align*} 
\sum_{\chi \in \cF_\ell(X)} \chi(m) &= K_\ell \prod_{\substack{p \mid m\\ p \equiv 1 \mod \ell}} \left( \frac{p}{p+\ell-1} \right) X + O \left( \ell ^{\omega(m)} X^{\frac{\ell+2}{\ell+5}+\epsilon} \right).
\end{align*}
Here
\begin{align*}
K_\ell = r_\ell F_{\ell,2}(1),
\end{align*}
where $r_\ell$ is the residue of the Dedekind zeta function of the $\ell$th cyclotomic extension $\Q(\omega_\ell)$  at $s=1$ and $F_{\ell,2}(s)$ is given by \eqref{F2}.
The power of $X$ in the error term can be improved to $O\left(\ell^{\omega(m)}X^{\frac{1}{2} + \epsilon}\right)$ assuming the Lindel\"of Hypothesis.
Therefore, when $m$ is an $\ell$th power, 
\begin{align} \label{quotient-family}
 \frac{1}{|\mathcal{F}_\ell(X)|}  \sum_{\chi \in \cF_\ell(X)} \chi(m) &=  \prod_{\substack{p \mid m\\ p \equiv 1 \mod \ell}} \left( \frac{p}{p+\ell-1} \right) + O \left( \ell^{\omega(m)} X^{-\frac{3}{\ell+5}+ \epsilon} \right),
\end{align}
and the error term can be improved to  $O \left( \ell^{\omega(m)} X^{-\frac{1}{2} + \epsilon} \right)$ assuming the Lindel\"of Hypothesis.
\end{prop}

\begin{rem} For the applications in this work, we will specialize to the case $\ell=3$. 
The main term for the first result was proven by Cohn \cite{Cohn} in the cubic case
by using the Dirichlet series
\begin{align*}
 \prod_{p \equiv 1 \mod 3} \left( 1 + 2p^{-s} \right) = \sum_{n =1}^\infty \frac{u(n)}{n^s}, 
\end{align*}
where $u(n)$ is the number of cubic characters of conductor $n$. 
\end{rem}

\begin{proof} 
Let $a_\ell(n)$ be the number of primitive characters of order $\ell$ and conductor $n$. Then,
\begin{align*}
\sum_{n=1}^{\infty} \frac{a_\ell(n)}{n^s} &= \prod_{p \equiv 1 \mod{\ell}} \left( 1 + \frac{\ell-1}{p^s} \right) = 
\prod_{p \equiv 1 \mod{\ell}} \left( 1 - \frac{1}{p^s} \right)^{-(\ell-1)} F_{\ell,1}(s) \\
&= \zeta_K(s) F_{\ell,2}(s),
\end{align*}
where $\zeta_K(s)$ is the Dedekind zeta function of $K=\Q(\omega_\ell)$. In other words, we have
\begin{align*}
\zeta_K(s) = \left( 1 - \frac{1}{\ell^s} \right)^{-1} \prod_{p \neq \ell} \left( 1 - \frac{1}{p^{e_\ell(p) s}} \right)^{- \frac{(\ell-1)}{e_\ell(p)}},
\end{align*}
where $e_\ell(p)$ is the multiplicative order of $p$ modulo $\ell$,
\begin{align*}
F_{\ell,1}(s) &= \prod_{p \equiv 1 \mod{\ell}} \left( 1 + \frac{\ell-1}{p^s} \right)  \left( 1 - \frac{1}{p^s} \right)^{\ell-1},  
\end{align*}
and 
\begin{align}\label{F2}
F_{\ell,2}(s) &= F_{\ell,1}(s) \left( 1 - \frac{1}{\ell^s} \right) \prod_{\substack{p \neq \ell\\ e_\ell(p) > 1}} \left( 1 - \frac{1}{p^{e_\ell(p) s}} \right)^{ \frac{\ell-1}{e_\ell(p)}}.
\end{align}
Notice that $F_{\ell,1}(s)$ and $F_{\ell,2}(s)$ converge absolutely for $\re(s) > \frac12 + \epsilon$.

For any $\frac12 + \epsilon \leq \sigma \leq 1$, we will  use the bound
\begin{align}
 |\zeta_K\left(\sigma+it\right)| &\ll_\epsilon t^{\frac{(\ell-1)(1-\sigma)}{3} +\epsilon} \;\;\;\;\; \text{\cite{HB88}}, \label{eq:HB}\\
|\zeta_K\left(\sigma+it\right)| &\ll_\epsilon t^{\epsilon} \;\;\;\;\; \text{Lindel\"of Hypothesis.} \label{eq:LH}
 \end{align}
We apply Perron's formula
\[\sum_{n \leq X} a_\ell(n)=\frac{1}{2\pi i} \int_{1+\epsilon-iT}^{1+\epsilon+iT} \zeta_K(s) F_{\ell,2}(s) \frac{X^s}{s} ds +O \left( \frac{X^{1+\epsilon}}{T}\right),\] 
and we move the integral to $s=\frac12 + \epsilon$.
Completing the above integral over a rectangle with vertices in $\frac{1}{2}+\epsilon\pm i T$, $1+\epsilon\pm i T$ and applying the bound \eqref{eq:HB}, we have, for  the horizontal integrals, 
\begin{align}
\ll & \frac{1}{T}\int_{\frac{1}{2}+\epsilon}^{1+\epsilon} |\zeta_K(\sigma+iT)| |F_{\ell,2}(\sigma+iT)| X^{\sigma} d\sigma\ll \frac{1}{T} \max_{\frac{1}{2}+\epsilon \leq \sigma \leq 1+\epsilon} T^{\frac{(\ell-1)(1-\sigma)}{3} +\epsilon}  X^\sigma \nonumber \\\ll&  \frac{1}{T}  \left( T^{\frac{\ell-1}{6} +\epsilon}  X^{\frac{1}{2}+\epsilon}+T^{\epsilon}  X\right). \label{eq:horizontal}
\end{align}
For the  integral over $s=\frac{1}{2}+\epsilon$, we have 
\begin{align}
\ll & \int_{-T}^{T} \left|\zeta_K\left(\frac{1}{2}+\epsilon+it\right)\right| \left|F_{\ell,2}\left(\frac{1}{2}+\epsilon+it\right)\right|  X^{\frac{1}{2}+\epsilon} \frac{dt}{\left|\frac{1}{2}+\epsilon+it\right|}\nonumber\\
\ll & T^{\frac{\ell-1}{6}+\epsilon} X^{\frac{1}{2}+\epsilon}.\label{eq:vertical}
\end{align}
Setting $T=X^{\frac{3}{\ell+5}}$, we finally obtain 
\begin{align*}
|\cF_\ell(X)| = r_\ell \, F_{\ell,2}(1) \,X + O \left( X^{\frac{\ell+2}{\ell+5}+\epsilon}\right),
\end{align*}
where $r_\ell$ is the residue of $\zeta_K(s)$ at $s=1$.

If we assume the Lindel\"of Hypothesis, we can use the bound \eqref{eq:LH} instead of \eqref{eq:HB}. Then \eqref{eq:horizontal}  and \eqref{eq:vertical} give a bound of $\ll T^{\epsilon-1} X +T^\epsilon X^{\frac{1}{2}+\epsilon}$. Taking $T=X^\frac{1}{2}$ leads to an error term of $O \left( X^{\frac{1}{2}+\epsilon}\right)$. 

Now suppose that $m$ is a $\ell$-th power. Then,
\begin{align*} 
\sum_{\chi \in \cF_\ell(X)} \chi(m) = \sum_{\substack{\chi \in \cF_\ell(X) \\ \left( \text{cond}(\chi), m \right) =1}} 1 = \sum_{\substack{n \leq X \\ (n, m)=1}} a_\ell(n), 
\end{align*}
and 
\begin{align*}
\sum_{\substack{n=1 \\ (n,m)=1}}^{\infty} \frac{a_\ell(n)}{n^s} = \prod_{\substack{p \equiv 1 \mod{\ell} \\ p \nmid m}} \left( 1 + \frac{\ell-1}{p^s} \right) = 
 \prod_{\substack{p \equiv 1 \mod{\ell} \\ p \mid m}} \left( 1 + \frac{\ell-1}{p^s} \right)^{-1} \zeta_K(s) F_{\ell,2}(s).
\end{align*}
We apply Perron's formula 
\[\sum_{n \leq X} a_\ell(n)=\frac{1}{2\pi i} \int_{1+\epsilon-iT}^{1+\epsilon+iT}  \prod_{\substack{p \equiv 1 \mod{\ell} \\ p \mid m}} \left( 1 + \frac{\ell-1}{p^s} \right)^{-1} \zeta_K(s) F_{\ell, 2}(s) \frac{X^s}{s} ds +O \left( \frac{X^{1+\epsilon}}{T}\right),\] 
and we move the integral to $s=\frac12 + \epsilon$ as before.  The bounds are the same, except that we have the extra factor 
\[\Bigg| \prod_{\substack{p \equiv 1 \mod{\ell} \\ p \mid m}} \left( 1 + \frac{\ell-1}{p^s} \right)^{-1}\Bigg|\ll \ell^{\omega(m)}.\]
Finally, we obtain,
\begin{align*}
\sum_{\chi \in \cF_\ell(X)} \chi(m)  =  \prod_{\substack{p \equiv 1 \mod{\ell} \\ p \mid m}} \left(\frac{p}{p + \ell-1}\right) r_K \, F_{\ell,2}(1) \,X + O \left(\ell^{\omega(m)} X^{\frac{\ell+2}{\ell+5}+\epsilon}\right).
\end{align*}
As before, assuming the Lindel\"of Hypothesis, the error term can be improved to $O \left(\ell^{\omega(m)} X^{\frac{1}{2}+\epsilon}\right)$. 

\end{proof}

Notice that the main term in \eqref{quotient-family} represents the expected value of the random variable $\mathbb{X}(m)$ defined for primes $p=\ell$ and $p\not \equiv 1 \mod \ell$ by   
\begin{equation*}
\mathbb{X}(p)=
                 \omega_\ell^k   \mbox{ with probability} = \textstyle{\frac{1}{\ell}}, \;\mbox{for $0 \leq k \leq \ell-1$},
                               \end{equation*}
                and for $p \equiv 1 \mod \ell$ by 
\begin{equation*}
\mathbb{X}(p)=\begin{cases}
                 0 & \mbox{with probability} = \frac{\ell-1}{p+\ell-1}\\
                   \omega_\ell^k  & \mbox{with probability} = \frac{p}{\ell(p+\ell-1)}, \;\mbox{for $0 \leq k \leq \ell-1$},

                \end{cases}
                \end{equation*}
and extended multiplicatively as  
$$\mathbb{X}(n)=\mathbb{X}(p_1)^{a_1}\mathbb{X}(p_2)^{a_2}\cdots \mathbb{X}(p_k)^{a_k},$$
for $n=p_1^{a_1}p_2^{a_2}\cdots p_k^{a_k}$.
Specializing   \eqref{quotient-family}  to $m=p^\ell$ where $p$ is a prime, we find that $\chi(p^\ell)$ is, on average over $\chi$, asymptotic to $1$ if $ p \not \equiv 1 \mod{\ell}$ and to $p/(p+\ell-1)$ otherwise. Now observe that $\chi(p^\ell)={\bf 1}_{\chi(p)\not = 0}$, so that \eqref{quotient-family}  can be interpreted as ``$\chi(p)$ is zero with probability $1-{\bf 1}_{p\equiv 1\mod{\ell}}(\ell-1)/(p+\ell-1)$''. Our choice of $\mathbb{X}(p)$ is the unique choice taking values in $\{0\}\cup \{\omega_\ell^k\, :\, k\}$ with the property that $\mathbb{X}(p)=0$ happens with the same probabilty as $\chi(p)=0$, and $\PP(\mathbb{X}(p)=\omega_\ell^k)$ is independent of $k$.
This justifies the definition of the random variables. 

We can also give a simple heuristic for the
random variables, independently of Proposition \ref{prop-gives-RVs}. The following argument is also found in  Lemma 8.1 of \cite{BDFL} for the function field case. 

The heuristic is as follows.
From the definition of  $\mathcal{F}_\ell$ given by \eqref{defF}, each primitive character of conductor $n$ (square-free and supported only on primes which are $\equiv 1 \mod \ell$) can be written uniquely as
$\chi_{n_1} \chi^2_{n_2}\cdots \chi_{n_{\ell-1}}^{\ell-1}$ for $n=n_1 n_2\cdots n_{\ell-1}$. Then, to count the characters, we can count the $(\ell-1)$-tuples $(n_1, n_2,\dots,n_{\ell-1})$ (supported on primes which are $\equiv 1 \mod \ell$) which are both square-free and pairwise coprime.

We create a model modulo $p^2$ for a fixed prime $p$. In order  to count the characters, we want to count the $(\ell-1)$-tuples $(n_1, n_2,\dots,n_{\ell-1}) \mod{p^2}$ such that $n_k$ is not equivalent to  $0 \mod{p^2}$  (this models the square-free condition), and $p$ does not simultaneously divide any pair of $n_j$ and $n_k$ (this models the coprimality condition).
Now we focus  on the additional condition that the character is nonzero. For that, we consider $\chi(p^\ell)=\chi_{n_1} (p^\ell) \chi_{n_2}^2 (p^\ell)\cdots \chi_{n_{\ell-1}}^{\ell-1}(p^\ell)$, as this character only takes the values 0 and 1. We need to count the $(\ell-1)$-tuples $(n_1, n_2,\dots,n_{\ell-1}) \mod{p^2}$ such that $n_k$ is not equivalent to  $0 \mod {p^2}$ and such that $n_k\not \equiv 0 \mod p$ (this models the fact that the character is not zero). 
Then, a model for $\P \left( \X(p) \neq 0 \right) = 1 - \P \left(\X(p)=0 \right)$  yielding the product on the left side of \eqref{quotient-family} is given by
the quotient 
\begin{align*}
&\frac{ \# \{ (n_1, \dots, n_{\ell-1}) \mod{p^2} \;:\; n_k \not\equiv 0 \mod{p^2} \forall k, \left( n_k \not\equiv 0 \mod p \forall k \right) \}}
{\# \{ (n_1, \dots, n_{\ell-1}) \mod{p^2} \;:\; n_k \not\equiv 0 \mod{p^2} \forall k, \left( n_k \equiv 0 \mod p \mbox{ for at most one value of }k \right) \} \} } .
\end{align*}

In the case where $p\not \equiv 1 \mod \ell$ or $p=\ell$, since the  $n_k$ are supported on primes that are $\equiv 1 \mod \ell$, the conditions $n_k \equiv 0 \mod p$ are never met, and the above quotient has the same numerator and denominator, 
which gives $\P \left( \X(p) = 0 \right)=0$. 

In the case where $p \equiv 1 \mod \ell$, the fact that the $n_k$ are supported only on primes that are $\equiv 1 \mod \ell$ is not a restriction as they cover all the residue classes mod $p^2$ by the Chinese Remainder Theorem and Dirichlet's Theorem for primes in arithmetic progressions. A simple counting argument then gives
\begin{align*}
&\frac{ \# \{ (n_1, \dots, n_{\ell-1}) \mod{p^2} \;:\; n_k \not\equiv 0 \mod{p^2} \forall k, \left( n_k \not\equiv 0 \mod p \forall k \right) \}}
{\# \{ (n_1, \dots, n_{\ell-1}) \mod{p^2} \;:\; n_k \not\equiv 0 \mod{p^2} \forall k, \left( n_k \equiv 0 \mod p \mbox{ for at most one value of }k \right) \} \} }\\
&= \frac{(p^2-p)^{\ell-1} }{(p^2-p)^{\ell-1} +(\ell-1)(p^2-p)^{\ell-2}(p-1) } =\frac{p^2-p}{(p^2-p) +(\ell-1)(p-1) }   = \frac{p}{p+\ell-1} ,
\end{align*}
and this gives $\P \left( \X(p) = 0 \right)=\frac{\ell-1}{p+\ell-1}$.

\section{Moments of $L(1,\chi)$} 
\label{section-largemoments} 

As we mentioned in the introduction, the proof of Theorem \ref{theorem-relating-to-RV} relies on comparing the moments of $|L(1,\chi)|^{2z}$ for $\chi \in \cF_3(X)$, to the expectations $\mathbb{E}(|L(1,\mathbb{X})|^{2z})$. We compute these moments in this section. 

In order to properly state our result, we introduce the generalized divisor function. Let $z \in \C$, and $n$ be a positive integer. 
The $z$-divisor function $d_z(n)$ is the multiplicative function defined on primes powers by 
\begin{align} \label{dzn}
d_z(p^a) = \frac{\Gamma(z+a)}{\Gamma(z) \; a!}.\end{align}
Then, for any Dirichlet series $D(s) = \sum_{n=1}^\infty \frac{a(n)}{n^s}$, where the $a(n)$ are completely multiplicative, and for $s \in \C$ such that the series is absolutely convergent, we have
$$
D(s)^z = \sum_{n=1}^\infty \frac{d_z(n) a(n)}{n^s}.
$$
We collect here the estimates that we will use for $d_z(n)$, which can also be found at the end of Section 2 of \cite{GS}. We have $|d_z(n)| \leq d_{|z|}(n)$, and for a real number $k \geq 1$, $d_k(mn) \leq d_k(m) d_k(n)$.
For any positive integers $a,b,n$, we have $d_a(n) d_b(n) \leq d_{a+b}(n)$, and for any $z \in \C, \beta \in \R$, $|d_z(n)|^\beta \leq d_{|z|^\beta}(n)$.

The goal of this section is to prove the following result. 
\begin{thm}\label{Asympformula}
Let $z$ be a complex number and $X$ be large. Then, uniformly in the region $|z|\leq\frac{\log X}{16\log_2X\log_3 X}$ we have 
\begin{align*}
\frac1{|\mathcal{F}_3(X)|}\sum_{\chi\in\mathcal{F}_3(X)}|L(1,\chi)|^{2z}&=\sum_{\substack{n,m\ge 1\\ nm^2=\tinycube}}\frac{d_z(m)d_z(n)}{mn}
\prod_{\substack{ p \equiv 1 \mod 3\\p \mid nm}} \frac{p}{p+2}
  +O\left(\exp\left(-\frac{\log X}{16\log_2 X}\right)\right)\\
  &=\E(|L(1, \X)|^{2z})+O\left(\exp\left(-\frac{\log X}{16\log_2 X}\right)\right),
\end{align*}
where the $z$-divisor function $d_z(n)$ is the multiplicative function defined by \eqref{dzn} and the $\cube$ notation indicates that the sum is restricted to the case where $nm^2$ is a perfect cube. 
\end{thm}
We remark that the range in Theorem \ref{Asympformula}  is more limited than what is provided in \cite[Theorem 3]{GS}, as  the authors are able to take $|z|\le \frac{\log X\log_3 X}{e^{12}\log_2X}$. A key input of their argument is a 
result of Graham and Ringrose \cite[Theorem 5]{GR} that is only available over $\Q$. For cubic characters, one needs to work over $\Q(\omega_3)$.

We first remark that the second identity of  Theorem \ref{Asympformula} follows directly since
\begin{align*}
\E(|L(1, \X)|^{2z}) &= \sum_{n, m=1}^\infty \frac{d_z(n) d_z(m)}{n m} \E \left( \X(n) \overline{\X}(m) \right) 
= \sum_{n, m=1}^\infty \frac{d_z(n) d_z(m)}{n m} \E \left( \X(nm^2) \right) \\
&= \sum_{\substack{n, m=1\\nm^2 = \tinycube}}^\infty \frac{d_z(n) d_z(m)}{n m} \prod_{\substack{p \equiv 1 \mod 3\\p \mid mn}}  \frac{p}{p+2},
\end{align*}
where the last line follows from Lemma \ref{expectation}.

Before proceeding to the proof of Theorem \ref{Asympformula}, we require a few auxiliary results. 
We first recall the following definition 
\begin{align} \label{count-zeros} N(\sigma,T,\chi):=\#\{\rho =\beta+i\gamma: L(\rho,\chi)=0,  \;\sigma\leq\beta<1 \text{ and } |\gamma|\leq T\}.\end{align}
Estimates for $N(\sigma, T, \chi)$ are known as zero density estimates. The Generalized Riemann Hypothesis states that for any $\sigma>\frac12$ we have $N(\sigma, T,\chi) =0$ for all characters $\chi$ modulo $q$.  In the absence of the Riemann Hypothesis one can show that $N(\sigma, T, \chi)$ is quite small when compared to the total zero count 
$$N(T,\chi) :=\#\{\rho =\beta+i\gamma: L(\rho,\chi)=0,  \;0< \beta < 1 \text{ and }  |\gamma|\leq T\}
\sim \frac{T}{\pi}\log T.$$

\begin{lem} \cite[Theorem 12.2]{Montgomery-lectures}\label{zerodensMont}
Suppose that $Q\geq 1$ and $T\ge 2$. If $\frac12\leq \sigma \leq \frac45$, we have 
\[\sum_{q\leq Q}\,\,\sideset{}{^*} \sum_{\chi \mod{q}}N(\sigma, T,\chi)\ll (Q^2T)^{\frac{3(1-\sigma)}{2-\sigma}}(\log (QT))^9,\]
and if $\frac45\leq \sigma \leq 1$ then 
\[\sum_{q\leq Q}\,\,\sideset{}{^*} \sum_{\chi \mod{q}}N(\sigma, T,\chi)\ll (Q^2T)^{\frac{2(1-\sigma)}{\sigma}}(\log (QT))^{14}.\]
\end{lem}
Note that the power of $T$ here is strictly less than 1 for all $\frac12< \sigma\leq 1$.  We will see why this becomes useful in just a few lemmas. First, let us write an expression for  $|L(1,\chi)|^{2z}$ under the assumption that  $N(1-\epsilon, 2(\log q)^{2/\epsilon},\chi)=0$ for a fixed $\epsilon$:

\begin{lem} \label{DL-momentsprop}
 Let $q$ be large and $0<\epsilon < \frac12$ be fixed. Let $y$ be a real number such that $\frac{\log q}{\log_2q}\leq \log y\leq \log q$. Furthermore, assume that $L(s,\chi)$ 
 has no zeros inside the rectangle $\{s: 1-\epsilon < \re(s)\leq 1 \, \mathrm{ and }\,  |\im(s)|\leq 2(\log q)^{2/\epsilon}\}$.  Then for any complex number $z$ such that $|z|\leq\log y/(4\log_2q\log_3q)$ we have 
\begin{align*}
|L(1,\chi)|^{2z}=\sum_{m,n\ge 1}\frac{d_z(m)d_z(n)\chi(n)\overline{\chi}(m)}{mn}e^{-mn/y} + O_{\epsilon}\left(\exp\left(-\frac{\log y}{4\log_2 q}\right)\right),
 \end{align*}
 where $d_z(n)$ is the $z$-divisor function defined by \eqref{dzn}.
\end{lem}
\begin{proof} 
The proof can be easily adapted from the proof of Proposition 3.3 in \cite{DL}. 
\end{proof}

We will need two versions of the P\'olya--Vinogradov inequality.

\begin{lem} \cite{Hinz} \label{PV-Hinz} Let $K$ be a number field of degree $d$ and let $\chi$ be a primitive nonprincipal character of the group of narrow ideal-classes modulo an ideal $\mathfrak{q}$. Then we have 
\[\sum_{N(\mathfrak{a})\leq X} \chi(\mathfrak{a}) \ll N(\mathfrak{q})^{\frac{1}{d+1}}(\log N(\mathfrak{q}))^d X^{\frac{d-1}{d+1}}.\]
\end{lem}

\begin{lem}
\label{Lem41GS}
If $n$ is a positive integer which is not a cube, then
$$\sum_{\chi \in \cF_3(X)} \chi(n) \ll X^{\frac12} \,\log{X} \,\tilde{n}^{\frac12} \,(\log{\tilde{n}})^{\frac32} ,$$
where $\tilde{n}$ denotes the radical of $n$. 
\end{lem}
\begin{rem}
This result is a weaker analogue of  \cite[Lemma 4.1]{GS}, which gives a bound of $X^{\frac12} n^\frac{1}{4} (\log{n})^{\frac12}$ in the quadratic case. This is because we have to use Lemma \ref{PV-Hinz}, since we are working with integers in $\Z[\omega_3]$,  and the degree of the extension is $2$.
\end{rem}

The following statement provides a way to express the number of elements of $\cF_3(X)$ as a single sum over characters of  $\Z[\omega_3]$, with some additional conditions.
\begin{lem}\cite[Lemma 2.1]{BY} \label{lemma-BY}
\label{BY-1to1-cubic-chars}
The primitive cubic Dirichlet characters of conductor $q$, $(q,3)=1$ are of the form $\chi_n:m\to\legendre{m}{n}_3$ for some $n\in \Z[\omega_3]$, $n\equiv 1 \mod 3$, $n$ squarefree and not divisible by any rational primes, with norm $N(n)=q$.
\end{lem}

From Lemma \ref{lemma-BY}, we have 
\[
\sum_{\chi \in \cF_3(X)} 1 =  \sideset{}{'}\sum_{\substack {n \in \Z[\omega_3] \\ N(n) \leq X}} 1,
\]
where the prime  indicates that the sum runs over the integers $n \in \Z[\omega_3]$ which are square-free, not divisible by any $p \in \Z$, and such that $n \equiv 1 \mod 3$. We use the detectors
\begin{equation*}
\sum_{\substack{d \in \Z, d \mid n \\ d \equiv 1 \mod 3}} \mu_\Z (d) = \begin{cases} 1 & \text{if $n$ is not divisible by a rational prime}, \\ 0 & \mbox{otherwise}, \end{cases}
\end{equation*}
(where $\mu_\Z(d) = \mu(|d|)$), and
\begin{equation*}
\sum_{\substack{d \in \Z[\omega_3], d \equiv 1 \mod 3\\ d^2 \mid n}} \mu_{\Z[\omega_3]} (d) = \begin{cases} 1 & \text{if $n$ is square-free}, \\ 0 & \mbox{otherwise}. \end{cases}
\end{equation*}

\begin{proof}[Proof of Lemma \ref{Lem41GS}]
Using Lemma \ref{lemma-BY}, we have
\begin{align}\label{eq:sieve}
&\sum_{\chi \in \cF_3(X)} \chi(n) 
= \sum_{\substack{d \in \Z, \;d  \equiv 1 \mod 3 \\ |d| \leq \sqrt{X}}}  \mu_\Z(d) \left( \frac{n}{d} \right)_3  
\sum_{\substack{\ell \in \Z[\omega_3], \;\ell \equiv 1 \mod 3\\ (\ell, d)=1\\ N(\ell) \leq \sqrt{X/N(d)}}} \mu_{\Z[\omega_3]} (\ell) \left( \frac{n}{\ell^2} \right)_3
\sum_{\substack{ c \in \Z[\omega_3], \;c \equiv 1 \mod 3 \\ (c,d)=1 \\ N(c) \leq X/N(d \ell^2)}} \left( \frac{n}{c} \right)_3 \\
&= \sum_{\substack{d \in \Z, \;d  \equiv 1 \mod 3 \\ e \mid d, \;e \equiv 1 \mod{3} \\|d| \leq \sqrt{X}}}  \mu_{\Z[\omega_3]} (e) \mu_\Z(d) \left( \frac{n}{d} \right)_3  \left( \frac{n}{e} \right)_3 
\sum_{\substack{\ell \in \Z[\omega_3], \;\ell \equiv 1 \mod 3\\ (\ell, d)=1 \\ N(\ell) \leq \sqrt{X/N(d)}}} \mu_{\Z[\omega_3]} (\ell) \left( \frac{n}{\ell^2} \right)_3
\sum_{\substack{ c \in \Z[\omega_3], \;c \equiv 1 \mod 3  \\ N(c) \leq X/N(d e \ell^2)}} \left( \frac{n}{c} \right)_3. \nonumber
\end{align}

The function $\psi_n: (c) \mapsto  \left( \frac{n}{c} \right)_3$ defined on ideals $(c) \subset \Z[\omega_3]$ (coprime to 3, and where $c \equiv 1 \mod 3$) is a non-trivial Hecke character of modulus $9n$, and 
we apply Lemma \ref{PV-Hinz}. It is not necessarily primitive, but we work with the primitive character induced by  $\left( \frac{n}{c} \right)_3$, whose conductor divides $9\tilde{n}$, where $\tilde{n}$ is the radical of $n$. Since $[\Q(\omega_3):\Q]=2$, we get by Lemma \ref{PV-Hinz} that
\[
\sum_{\substack{ c \in \Z[\omega_3], c \equiv 1 \mod 3  \\ N(c) \leq X/N(d e \ell^2)}} \left( \frac{n}{c} \right)_3 \ll N(\tilde{n})^{\frac13} \log^2{N(\tilde{n})}  \left( \frac{X}{N(de\ell^2)} \right)^{\frac13} \ll \tilde{n}^{\frac 23} \left( \frac{X}{N(de\ell^2)} \right)^{\frac13} (\log{\tilde{n}})^2,\]
and then
\begin{align*}
\sum_{\chi \in \cF_3(X)} \chi(n)  &\ll \sum_{\substack{d \in \Z,\; d  \equiv 1 \mod 3 \\ |d| \leq \sqrt{X}\\e \mid d, \; e \equiv 1 \mod 3}} 
\sum_{\substack{\ell \in \Z[\omega_3], \ell \equiv 1 \mod 3\\ (\ell, d)=1 \\ N(\ell) \leq \sqrt{X/N(d)}}} \min \left( \tilde{n}^{\frac 23} \left( \frac{X}{N(de\ell^2)} \right)^{\frac13} (\log{\tilde{n}})^2, \frac{X}{N(ed\ell^2)} \right) \\
&\ll X^{\frac13} \tilde{n}^{\frac 23} (\log{\tilde{n}})^2 \sum_{\substack{d \in \Z, d  \equiv 1 \mod 3 \\ |d| \leq \sqrt{X}\\e \mid d, \; e \equiv 1 \mod 3}} \frac{1}{N(ed)^{\frac13}}
\sum_{\substack{\ell \in \Z[\omega_3], \ell \equiv 1 \mod 3\\ N(\ell) \leq \ell_0}} \frac{1}{N(\ell)^{\frac 23}} \\ & + X \sum_{\substack{d \in \Z, d  \equiv 1 \mod 3 \\ |d| \leq \sqrt{X}\\e \mid d, \; e \equiv 1 \mod 3}} \frac{1}{N(ed)}\sum_{\substack{\ell \in \Z[\omega_3], \ell \equiv 1 \mod 3\\ N(\ell) \geq \ell_0}}  \frac{1}{N(\ell)^2}  \\
&\ll X^{\frac12} \tilde{n}^{\frac12} (\log{\tilde{n}})^{\frac32} \sum_{\substack{d \in \Z, d  \equiv 1 \mod 3 \\ |d| \leq \sqrt{X}\\e \mid d, \; e \equiv 1 \mod 3}} \frac{1}{N(ed)^{\frac12}} \ll X^{\frac12}\log X \tilde{n}^{\frac12} (\log{\tilde{n}})^{\frac32},
\end{align*}
where $\ell_0 = X^{\frac{1}{2}} \tilde{n}^{-1/2} (\log \tilde{n})^{-3/2} N(ed)^{-1/2}$.
\end{proof}

\begin{proof}[Proof of Theorem \ref{Asympformula}.]
Let $y=X^\alpha$ for some $\alpha<\frac{1}{2}$  and let $z$ be a complex number such that $|z|\leq\log y/(4\log_2X\log_3X )$. 
Let $$\widetilde{\mathcal{F}}_{\epsilon}(X)=\{\chi\in\mathcal{F}_3(X): N(1-\epsilon, 2(\log q)^{2/\epsilon},\chi)=0\},$$ 
where $N(1-\epsilon, 2(\log q)^{2/\epsilon},\chi)$ is defined by \eqref{count-zeros}. Then,
\[\frac1{|\mathcal{F}_3(X)|}\sum_{\chi\in\mathcal{F}_3(X)}|L(1,\chi)|^{2z}=\frac1{|\mathcal{F}_3(X)|}\sum_{\chi\in\widetilde{\mathcal{F}}_{\epsilon}(X)}|L(1,\chi)|^{2z}+\frac1{|\mathcal{F}_3(X)|}\sum_{\chi\in\mathcal{F}_3(X)\setminus\widetilde{\mathcal{F}}_{\epsilon}(X) }|L(1,\chi)|^{2z}.\]
We first  show that the second sum is negligible using Lemma \ref{zerodensMont}. Afterwards, we apply Lemma \ref{DL-momentsprop} to the first sum to obtain the desired main term.
First, if $\frac45\leq1-\epsilon$ then 
\begin{align*}
|\mathcal{F}_3(X)\setminus\widetilde{\mathcal{F}}_{\epsilon}(X)|&=\sum_{q\leq X} \sum_{\substack{\chi \text{ cubic} \\ N(1-\epsilon,  2(\log q)^{2/\epsilon},\chi)\neq 0}}1\\
&\ll \sum_{q\leq X}\,\,\sideset{}{^*}\sum_{\chi \mod q}N(1-\epsilon, 2(\log q)^{2/\epsilon},\chi)\\
&\ll  (X^22(\log X)^{2/\epsilon})^{\frac{2\epsilon}{1-\epsilon}}(\log (X2(\log X)^{2/\epsilon}))^{14} \ll X^{\frac45+\epsilon},
\end{align*}
where the last line comes from choosing \footnote{It suffices to choose any $\epsilon < \frac15$. The choice of   $\epsilon=\frac16$ is motivated by convenience of the final expression.}  $\epsilon=\frac16$.
Thus we have, by the standard bound $|L(1,\chi)|\ll \log \mathrm{cond}(\chi)$,  
\[\sum_{\chi\in\mathcal{F}_3(X)\setminus\widetilde{\mathcal{F}}_{\epsilon}(X) }|L(1,\chi)|^{2z}\ll X^{\frac45+\epsilon}\exp(2z\log_2 X).\]
Recall that $y=X^{\alpha}$, for some $0<\alpha<\frac{1}{2}$, which gives
\[\sum_{\chi\in\mathcal{F}_3(X)\setminus\widetilde{\mathcal{F}}_{\epsilon}(X) }|L(1,\chi)|^{2z}\ll X^{\frac45+\epsilon}\exp\left(\frac{\alpha \log X}{2\log_3X}\right)\ll X^{\frac45+2\epsilon}.\]
So far we have seen, 
\[\frac1{|\mathcal{F}_3(X)|}\sum_{\chi\in\mathcal{F}_3(X)}|L(1,\chi)|^{2z}=\frac1{|\mathcal{F}_3(X)|}\sum_{\chi\in\widetilde{\mathcal{F}}_{1/6}(X)}|L(1,\chi)|^{2z}+O\left(X^{-\frac15+\epsilon}\right).
\]
Now, we apply Lemma \ref{DL-momentsprop} to obtain
\begin{align*}\frac1{|\mathcal{F}_3(X)|}\sum_{\chi\in\widetilde{\mathcal{F}}_{{1/6}}(X)}|L(1,\chi)|^{2z}=\frac{1}{|\mathcal{F}_3(X)|}\sum_{\chi\in \widetilde{\mathcal{F}}_{1/6}(X)}\sum_{m,n\ge 1}\frac{d_z(m)d_z(n)\chi(n)\overline{\chi}(m)}{mn}e^{-mn/y}\\+ O\left(\exp\left(-\frac{\alpha\log X}{4\log_2 q}\right)\right).
\end{align*}
In order to apply the orthogonality relation we require the sum to be over $\mathcal{F}_3(X)$ and not $\widetilde{\mathcal{F}}_{1/6}(X)$. We can extend our sum to the full family by noting that
 \begin{align*}
 \sum_{\chi\in \mathcal{F}_3(X)\setminus\widetilde{\mathcal{F}}_{1/6}(X)}\sum_{m,n\ge 1}\frac{d_z(m)d_z(n)\chi(n)\overline{\chi}(m)}{mn}e^{-mn/y}&\ll |\mathcal{F}_3(X)\setminus\widetilde{\mathcal{F}}_{1/6}(X)|\sum_{m,n\ge 1}\frac{d_z(m)d_z(n)}{mn}e^{-mn/y}\\
 &\ll X^{\frac{4}{5}+\epsilon}(\log(3y))^{2\lceil|z|\rceil}\ll X^{\frac{4}{5}+2\epsilon},
 \end{align*}
where the second line follows from \cite[Eq.~(2.4)]{GS} and the fact  that $e^{-mn/y} \leq e^{-(m+n)/y}$ for $m,n>1$,
and the last calculation comes from the assumptions on $y$ and $|z|$. 
 
 Therefore, 
 \begin{align*}
 \frac1{|\mathcal{F}_3(X)|}\sum_{\chi\in\mathcal{F}_3(X)}&|L(1,\chi)|^{2z}=\sum_{m,n\ge 1}\frac{d_z(m)d_z(n)}{mn}e^{-mn/y} \frac1{|\mathcal{F}_3(X)|} \sum_{\chi\in {\mathcal{F}}_{3}(X)}\chi(nm^2)+ O\left(\exp\left(-\frac{\alpha\log X}{4\log_2 q}\right)\right),
 \end{align*}
 and applying \eqref{quotient-family}, we have
 \begin{align*}
 &\frac1{|\mathcal{F}_3(X)|}\sum_{\chi\in\mathcal{F}_3(X)}|L(1,\chi)|^{2z} =\sum_{\substack{m,n\ge 1\\ nm^2=\tinycube}}\frac{d_z(m)d_z(n)e^{-mn/y}}{mn}\prod_{\substack{ p \equiv 1 \mod 3\\p \mid nm}} \frac{p}{p+2} \\
 &+ O\Bigg(X^{-3/8+\epsilon}\sum_{\substack{m,n\ge 1\\ nm^2=\tinycube}}\frac{d_z(m)d_z(n)(d(mn))^2}{mn}e^{-mn/y}\Bigg)\\
 &+\frac1{|\mathcal{F}_3(X)|}\sum_{\substack{m,n\ge 1\\ nm^2\neq\tinycube}}\frac{d_z(m)d_z(n)}{mn}e^{-mn/y}\sum_{\chi\in {\mathcal{F}_3}(X)}\chi(nm^2) + O\left(\exp\left(-\frac{\alpha\log X}{4\log_2X}\right)\right).
 \end{align*}
We recognize the first sum as the main term listed in Theorem \ref{Asympformula} with a smoothing factor $e^{-mn/y}$, which will be removed later.  Therefore, it is necessary to prove that the other terms are negligible.  For the second sum, we have
\begin{align*}
\nonumber X^{-\frac{3}{8}+\epsilon}\sum_{\substack{m,n\ge 1\\ nm^2=\tinycube}}\frac{d_z(m)d_z(n)(d(mn))^2}{mn}e^{-mn/y}&\ll X^{-\frac{3}{8}+\epsilon}\bigg(\sum_{\substack{n\ge 1 }}\frac{d_{\lceil|z|\rceil+2}(n)}{n}e^{-n/y}\bigg)^2\\
&\ll X^{-\frac{3}{8}+\epsilon}(\exp(3y))^{2(\lceil|z|\rceil+2)} \ll X^{-\frac{3}{8}+2\epsilon},
\end{align*}
where we have applied that $e^{-mn/y}\leq  e^{-(m+n)/y}$, $d(mn)\leq d(m)d(n)$ and $d_a(m)d_b(m)\leq d_{a+b}(m)$ for $a, b$ positive integers.

For the sum over non-cubes we apply Lemma \ref{Lem41GS} together with the fact that the radical of $nm^2$ is bounded by $nm$ to get
\begin{align}\label{noncube1}
\sum_{\substack{m,n\ge 1\\ nm^2\neq\tinycube}}\frac{d_z(m)d_z(n)}{mn}e^{-mn/y}\sum_{\chi\in {\mathcal{F}_3}(X)}\chi(nm^2)\ll X^{\frac12}\log X\bigg(\sum_{\substack{n\ge 1}}\frac{d_z(n)(\log(n))^{\frac32}}{\sqrt{n}}e^{-n/y}\bigg)^2.
\end{align}
We now further split \eqref{noncube1} into two sums, for $n\leq y(\log y)^2$ and for $n>y(\log y)^2$. For the first case, using $1\leq\frac{\sqrt{y}\log(y)}{\sqrt{n}}$, we get
\begin{align*}
\sum_{\substack{n\leq y(\log y)^2}}\frac{d_z(n)(\log n)^{\frac32}}{\sqrt{n}}e^{-n/y}\leq\sqrt{y}(\log y)^3 \sum_{n\leq y(\log y)^2}\frac{d_z(n)}{n}e^{-n/y}\ll \sqrt{y}(\log y)^3(\log(3y))^{\lceil|z|\rceil}.
\end{align*}
Replacing in \eqref{noncube1} and using $y=X^{\alpha}$ we see the contribution of $n\leq y(\log y)^2$ is
\begin{align*}
\ll X^{\frac12+\alpha}(\log X)^7\exp\left(\frac{\alpha}{2}\frac{\log(X)}{\log_3 X}\right)\ll X^{\frac12+\alpha+\epsilon},
\end{align*}
which is $O\left(X^{1-\epsilon'}\right)$ when $\alpha<\frac{1}{2}$. 
For the remaining range, we  note that if $m\geq y(\log y)^2$ then $e^{-m/(2y)}\leq e^{-(\log y)^2/2}$, so that 
\begin{align*}
\Bigg(\sum_{\substack{n\geq y(\log y)^2 }}\frac{d_z(n)(\log(n))^{\frac32}}{\sqrt{n}}e^{-n/y}\Bigg)^2\ll e^{-(\log y)^2}(\log X)^3\Bigg(\sum_{\substack{n\geq y(\log y)^2 }}\frac{d_z(n)}{\sqrt{n}}e^{-n/(2y)}\Bigg)^2.
\end{align*}
Now, if $k$ is a positive integer and $x>3$, we have $d_{k}(n)e^{-n/x} \leq e^{k/x}\sum_{a_1a_2\cdots a_k=n}e^{-(a_1+a_2+\cdots+a_k)/x}$ so that 
\begin{align*}
\Bigg(\sum_{\substack{n\geq y(\log y)^2 }}\frac{d_z(n)}{\sqrt{n}}e^{-n/(2y)}\Bigg)^2 \leq \Bigg(e^{1/(2y)}\sum_{a=1}^{\infty}\frac{e^{-a/(2y)}}{a^{\frac12}}\Bigg)^{2\lceil|z|\rceil} \ll y^{2\lceil|z|\rceil}.
\end{align*}
Replacing in \eqref{noncube1},  we see the contribution of  $n>y(\log y)^2$ is
\begin{align*}
\ll X^{\frac12}(\log X)^4e^{-(\log y)^2}y^{2\lceil|z|\rceil}&=\exp\left(\frac12\log X +4\log_2 X -(\log y)^2 +2\lceil|z|\rceil \log y\right)\\
&=\exp\left(\frac12\log X +4\log_2 X -(\alpha\log X)^2 +  \frac{\alpha^2(\log X)^2}{2\log_2X\log_3X}\right)\\  
& \ll X^{\frac12+\epsilon}.
\end{align*}
Thus,  we have proven that
\begin{align}  \label{with-ET}
 \frac1{|\mathcal{F}_3(X)|}\sum_{\chi\in\mathcal{F}_3(X)}|L(1,\chi)|^{2z} &=  \sum_{\substack{m,n\ge 1\\ nm^2=\tinycube}}\frac{d_z(m)d_z(n)e^{-mn/y}}{mn}
 \prod_{\substack{ p \equiv 1 \mod 3\\p \mid nm}} \frac{p}{p+2} + O\left(\exp\left(-\frac{\alpha\log X}{4\log_2X}\right)\right),
 \end{align}
and the last remaining step is the removal of the smoothing factor. Note that $1-e^{-t}\ll t^{\beta}$ for any $\beta,t>0$, thus 
\begin{align} \nonumber
\sum_{\substack{m,n\ge 1\\ nm^2=\tinycube}}\frac{d_z(m)d_z(n)(1-e^{-mn/y})}{mn}  \prod_{\substack{ p \equiv 1 \mod 3\\p \mid nm}} \frac{p}{p+2}
&\ll \sum_{\substack{m,n\ge 1\\ nm^2=\tinycube}}\frac{d_z(m)d_z(n) (\frac{mn}y)^{\beta}}{mn}\\ \label{removing-smoothing}
&\ll y^{-\beta}\sum_{r\geq 1}\frac{d_{|z|}(r)^2}{r^{2-2\beta}}\left(\sum_{s\geq1}\frac{d_{|z|}(s)^3}{s^{3-3\beta}}\right)^2,
\end{align}
where we have used $nm^2=\cube$ and the substitution $n=rs^3$ and $m=rt^3$.   For the sum over $r$, we use \cite[Lemma 3.3]{Lam2011}, with the choice $\beta=\frac{3}{\log_2 X}$,   to obtain
\[\sum_{r\geq1}\frac{d_{|z|}(r)^2}{r^{2-2\beta}}\leq \exp((2+o(1))\lceil|z|\rceil \log_2\lceil|z|\rceil).\]
For the sum over $s$, we use  the multiplicativity of $d_z(n)$, which gives
$$
\sum_{s=1}^{\infty}\frac{d_{|z|}(s)^3}{s^{3-3\beta}} =\prod_p \left(1+\frac{d_{|z|}(p)^3}{p^{3-3\beta}}+\frac{d_{|z|}(p^2)^3}{p^{6-6\beta}}+\cdots\right),
$$
the recursive property $\Gamma(z+1)=z\Gamma(z)$, and the fact that for any integer $r\geq 1$, we have
$$
d_{|z|}(p^r)=\frac{\Gamma(|z|+r)}{\Gamma(|z|)\, r!}=\frac{|z|(|z|+1))\cdots (|z|+r-1)}{r!}.
$$
Splitting the Euler product into two pieces depending on the size of $p$ relative to $|z|$, we have 
\begin{align*}
\sum_{s=1}^{\infty}\frac{d_{|z|}(s)^3}{s^{3-3\beta}} &=\prod_{p\leq |z|+2}\left(1+\frac{d_{|z|}(p)^3}{p^{3-3\beta}}+\frac{d_{|z|}(p^2)^3}{p^{6-6\beta}}+\cdots\right)\prod_{p>|z|+2}\left(1+O\left(\frac{|z|^3}{p^{3-3\beta}}\right)\right)\\
&\ll \prod_{p\leq |z|+2}\left(1+\frac{d_{|z|}(p)}{p^{1-\beta}}+\frac{d_{|z|}(p^2)}{p^{2-2\beta}}+\cdots\right)^3 \prod_{p>|z|+2}\left(1+O\left(\frac{|z|^3}{p^{3-3\beta}}\right)\right)\\
&\ll \prod_{p\leq |z|+2}\left(1+\frac{3(|z|+1)}{p^{1-\beta}}+\cdots\right)^3 \prod_{p>|z|+2}\left(1+O\left(\frac{|z|^3}{p^{3-3\beta}}\right)\right)\\
&\ll \prod_{p\leq |z|+2}\left(1-\frac{1}{p^{1-\beta}}\right)^{-3|z|}\prod_{p>|z|+2}\left(1+O\left(\frac{|z|^3}{p^{3-3\beta}}\right)\right), 
\end{align*} 
and using Mertens' theorem $\prod_{p \leq y}\left(1-\frac{1}{p}\right)\sim \frac{e^{-\gamma}}{\log y}$, we  obtain
\begin{align*}
\sum_{s=1}^{\infty}\frac{(d_{|z|}(s))^3}{s^{3-3\beta}} &\ll (e^{\gamma}\log(|z|+2))^{3|z|}\prod_{p>|z|+2}\left(1+O\left(\frac{|z|^3}{p^3}\right)\right)\\
&\ll (e^2 \log(|z|+2))^{3|z|}.
\end{align*}
Replacing in \eqref{removing-smoothing}, the error term from removing the smoothing is 
 \begin{align*}
 &\ll \exp(-\beta\log y +12|z|+6|z|\log_2|z| +(2+o(1))\lceil|z|\rceil \log_2\lceil|z|\rceil) ) \\
&\ll \exp\left(-\beta \alpha \log X+\frac{3\alpha \log X}{\log_2X\log_3X} +\frac{3\alpha\log X}{2\log_2X\log_3X}\log \left(\log_2X -\log_3X-\log_4X\right)\right. \\  
&\hspace{1in}  + (2+o(1))\frac{\alpha\log X}{4\log_2X\log_3X}\log \left(\log_2X -\log_3X-\log_4X\right)\Bigg) \\
& \ll \exp\left(-\frac{\alpha \log X}{\log _2 X}\right),
\end{align*}
which is smaller than the error term from \eqref{with-ET}. 
Recalling that $\alpha<1/2$, we choose $\alpha=1/4$. 
This completes the proof. 
\end{proof}

\section{Working with the random variables}\label{working with the RV}

Let $\X(m)$ and $L(1, \X)$ be the random variables defined in the introduction. The main result of this section is the following theorem, which will be used to prove Theorems \ref{thm-RV} and  \ref{theorem-relating-to-RV}.

\begin{thm}\label{main thm for Distribution of RV} 
For $\tau$ large,  we have
\begin{align*}
\Phi(\tau) = \frac{\E \left( |L(1, \X)|^{2\kappa} \right) ( e^\gamma \tau)^{-2\kappa}}{\kappa \sqrt{2 \pi \mathcal{L}''(\kappa)}} \left( 1 +  O\left(\sqrt{\frac{\log \kappa}{\kappa}}\right) \right)
\end{align*}
and
\[
\Psi(\tau)=\frac{\E \left( |L(1, \X)|^{-2\widetilde{\kappa}} \right)\zeta(3)^{ \widetilde{\kappa}} }{\widetilde{\kappa} \sqrt{2 \pi \widetilde{\mathcal{L}}''(\widetilde{\kappa})} {( e^\gamma \tau^2)^{\widetilde{\kappa}}}} \left( 1 +  O\left(\sqrt{\frac{\log \widetilde{\kappa}}{ \widetilde{\kappa}}}\right) \right),
\]
where $\Phi(\tau)$ and $\Psi(\tau)$ are given by \eqref{def:Phi} and \eqref{def:Psi} respectively, $\mathcal{L}, \widetilde{\mathcal{L}}$ are defined by \eqref{mathcalL} and \eqref{tilde-mathcalL}, and  $\kappa, \widetilde{\kappa}$ are the unique solutions to \eqref{saddlept} and \eqref{saddleptmin}.

\end{thm}

\subsection{Expectations}

\begin{lem} \label{expectation}
Let  $m= p_1^{a_1}\cdots p_k^{a_k}$, and let $\X(p)$ be the independent random variables defined in the introduction, taking the value 0 with probability 
$$\delta_p := \begin{cases} 0 & \text{if}\; p \equiv 2 \mod{3} \;\text{or}\; p=3, \\
\frac{2}{p+2} & \text{if}\; p \equiv 1 \mod{3}, \end{cases}$$
 and each of the values $1, \omega_3, \omega_3^2$ with probability $\frac{\alpha_p}{3}$ where
\begin{equation}\label{eq:alphas}
 \alpha_p := \begin{cases} 1 & \text{if}\; p \equiv 2 \mod{3} \;\text{or}\; p=3, \\
\frac{p}{p+2} & \text{if}\; p \equiv 1 \mod{3}. \end{cases}
 \end{equation}
 Notice that $\alpha_p=1-\delta_p$.
Then
$$\mathbb{E}(\mathbb{X}(m)) = \begin{cases}
                            \prod_{j=1}^k (1-\delta_{p_j}) = \prod_{\substack{p \equiv 1 \mod 3 \\ p \mid m}} \frac{p}{p+2} 
                            & \mbox{if $m=\cube$,}\\
                            0 & \mbox{otherwise.}
                           \end{cases}$$
Furthermore,  $\sum_{m=1}^\infty \mathbb{E}\left(\frac{\mathbb{X}(m)}{m}\right)\leq \zeta(3)$ and $ \sum_{m=1}^\infty \mathrm{Var}\left(\frac{\mathbb{X}(m)}{m}\right)\leq \zeta(6)$.
\end{lem}

\begin{proof}
By independence,
$$\mathbb{E}(\mathbb{X}(m))= \prod_{j=1}^k \mathbb{E}(\mathbb{X}(p_j)^{a_j}).
$$
 If $a_j$ is divisible by 3, then 
$$\mathbb{E}(\mathbb{X}(p_j)^{a_j})=\frac{\alpha_{p_j}}{3}+\frac{\alpha_{p_j}}{3}+\frac{\alpha_{p_j}}{3}=1-\delta_{p_j}.$$
On the other hand, if $a_j\equiv \pm 1 \mod{3}$, 
$$\mathbb{E}(\mathbb{X}(p_j)^{a_j})=\frac{\alpha_{p_j}}{3}+\omega_3^{\pm 1} \frac{\alpha_{p_j}}{3}+\omega_3^{\mp 1} \frac{\alpha_{p_j}}{3}=0.$$
Finally, we have 
\[\sum_{m=1}^\infty \mathbb{E}\left(\frac{\mathbb{X}(m)}{m}\right)=\sum_{n=1}^\infty \frac{\mathbb{E}(\mathbb{X}(n^3))}{n^3}\leq \zeta(3),\]
and 
\begin{align*}
\sum_{m=1}^\infty \mathrm{Var}\left(\frac{\mathbb{X}(m)}{m}\right)=&\sum_{m\not= \tinycube}
\frac{\mathbb{E}(\mathbb{X}(m)^2)}{m^2}+\sum_{n=1}^\infty\frac{1}{n^6} \left[\mathbb{E}\left(\mathbb{X}(n^3)^2\right)-\mathbb{E}\left(\mathbb{X}(n^3)\right)^2\right]\\
=&\sum_{n=1}^\infty \frac{1}{n^6} \left[\mathbb{E}\left(\mathbb{X}(n^3)\right)-\mathbb{E}\left(\mathbb{X}(n^3)\right)^2\right]\leq  \zeta(6).
\end{align*}

\end{proof}

\begin{rem} Recall that Kolmogorov's three-series Theorem states that for a family $\{X_m\}_{m\in \mathbb{N}}$ of independent random variables, the random series $\sum_{m=1}^\infty X_m$ converges almost surely in $\R$ if only if the following conditions hold for some $A>0$:
\begin{enumerate}
\item[(i)] $\displaystyle{\sum_{m=1}^\infty \mathbb{P}\left(\left|X_m\right|\geq A\right)<\infty.}$
\item[(ii)]   $\displaystyle{\sum_{m=1}^\infty \mathbb{E}\left(X_m {\bf 1}_{\{|X_m|\leq A\}} \right)<\infty.}$
\item[(iii)]  $\displaystyle{\sum_{m=1}^\infty \mathrm{Var}\left(X_m {\bf 1}_{\{|X_m|\leq A\}}\right)<\infty.}$
\end{enumerate}

 Lemma \ref{expectation} guarantees the conditions for Kolmogorov's three-series Theorem for the family of random variables $X_m:=\frac{\mathbb{X}(m)}{m}$ in the case $A=1$.  We thus conclude that $\sum_{m=1}^\infty \frac{\mathbb{X}(m)}{m}$ converges almost surely in $\R$.
\end{rem}

For any $z \in \C$, let $$E_p(z) := \mathbb{E}\left(\left|1 - \frac{\mathbb{X}(p)}{p} \right|^{-2z} \right).$$

\begin{lem} 
For $z\in \C$,
\[
E_p(z)=1-\alpha_p + \frac{\alpha_p}{3} \left( 1-\frac{2}{p} + \frac{1}{p^2}\right)^{-z}  + \frac{2 \alpha_p}{3} \left(1+\frac{1}{p} + \frac{1}{p^2} \right)^{-z}.
\]
\end{lem}

\begin{proof} We compute
  \begin{align*}E_p(z) = \mathbb{E}\left(\left|1 - \frac{\mathbb{X}(p)}{p} \right|^{-2z} \right) &= \mathbb{E}\left(\left|1 - \frac{\mathbb{X}(p) + \overline{\mathbb{X}}(p)}{p} +  \frac{\mathbb{X}(p)  \overline{\mathbb{X}}(p)}{p^2}\right|^{-z} \right)\\
&=  1-\alpha_p + \frac{\alpha_p}{3} \left( 1-\frac{2}{p} + \frac{1}{p^2}\right)^{-z}  + \frac{2 \alpha_p}{3} \left(1+\frac{1}{p} + \frac{1}{p^2} \right)^{-z},
               \end{align*}
               and the statement follows.
\end{proof} 

For any real number $r>0$, we define
\begin{align} \label{mathcalL}
\mathcal{L}(r) &:=\log(\mathbb{E}(|L(1,\mathbb{X})|^{2r}))=\sum_p\log E_p(r),\\ \label{tilde-mathcalL}
\widetilde{\mathcal{L}}(r) &:=\log(\mathbb{E}(|L(1,\mathbb{X})|^{-2r}))=\sum_p\log E_p(-r).
\end{align}

\begin{prop}\label{L-est-prop}
For any real number $r\geq 4$ we have 
\begin{align}\label{L-estimate}
&\mathcal{L}(r)=2r \log \log r +2r \gamma + \frac{2r(C_{\max}-1)}{\log r}+
 O\left( \frac{r}{(\log r)^2}\right), \\
 &\widetilde{\mathcal{L}}(r)=r \log \log r +r \gamma + \frac{r(C_{\min}-1)}{\log r}-r\log \zeta(3)+
 O\left( \frac{r}{(\log r)^2}\right), \label{L-estimatetilde}
\end{align}
and 
 \begin{align}
  \begin{split}\label{Lprime-estimate}
  &\mathcal{L}'(r)  = 2 \log \log r +2\gamma+\frac{2C_{\max}}{\log r} +O\left(\frac{1}{(\log r)^2} \right),\\
  &\widetilde{\mathcal{L}}'(r)= \log \log r +\left(\gamma-\log \zeta(3)\right)+\frac{C_{\min}}{\log r} +O\left(\frac{1}{(\log r)^2} \right),
  \end{split}
 \end{align}
 where
 \begin{align}\label{constantC0}
2C_{\max}:=\int_0^1 \frac{2e^{2t}-2e^{-t}}{t(e^{2t}+2e^{-t})}dt+\int_1^\infty \frac{-6}{t(e^{3t}+2)}dt\approx1.97455\ldots
 \end{align}
 and 
 \begin{align}\label{minimum value}
 C_{\min}:= \int_{0}^1 \frac{2e^{t}-2e^{-2t}}{t(e^{-2t}+2e^{t})}dt +\int_1^{\infty} \frac{-3e^{-2t}}{t(e^{-2t}+2e^{t})}dt\approx 1.40459\ldots.
 \end{align}
Moreover, for all real numbers $y, t$ such that $|y|\geq 3$ and $|t|\leq|y|$ we have 
\begin{align}\label{Lhigherderivs-estimate} 
\mathcal{L}''(y)\asymp \frac{2}{|y| \log|y|} & \text{ and } \mathcal{L}'''(y+it)\ll \frac{1}{|y|^2 \log|y|}.
\end{align}
The same holds for the derivatives of $\widetilde{\mathcal{L}}(r)$.
\end{prop}
We will prove Proposition \ref{L-est-prop} in Section \ref{section-L-est-prop}.

\begin{rem}
 The difference between the coefficients of $r\log \log r$ in equations \eqref{L-estimate} and \eqref{L-estimatetilde} reflects the distinction between the maximal values (achieved when 
 $\mathbb{X}(p)=1$, and when the term $\left(1-\frac{2}{p}+\frac{1}{p^2}\right)^{-r}$ dominates in $E_p(r)$) and the minimal values (achieved when $\re(\mathbb{X}(p))=-\frac{1}{2}$, and when the term $\left(1+\frac{1}{p}+\frac{1}{p^2}\right)^{-r}$ dominates in $E_p(-r)$). 
\end{rem}

\subsection{Proof of Proposition \ref{L-est-prop}}
\label{section-L-est-prop}

We first prove some  preliminary results. 

\begin{lem}\label{Ep-bound-lem}
Let $r\geq 4$ be a real number. Then we have
\begin{equation}\label{logEp-bound}
\log E_p(r)=\begin{cases}
-2r \log \left(1-\frac{1}{p} \right) +O(1) & \text{if}\; p <r^{\frac{2}{3}},\\
            \log \left( \frac{e^{2r/p}+2e^{-r/p}}{3}\right) +O\left(\frac{r}{p^2}\right) & \text{if}\; p>r^{\frac{2}{3}},
             \end{cases}
\end{equation}
and 
\begin{equation*}
\frac{E'_p(r)}{E_p(r)}=\begin{cases}
-2\log \left( 1-\frac{1}{p}\right) \left( 1+O\left(\exp(-3r^{\frac{1}{3}} \right)\right) & \text{ if } p\leq r^{\frac{2}{3}},\\
\frac{\frac{2}{p} \left(e^{2r/p}-e^{-r/p}\right)}{e^{2r/p}+2e^{-r/p}}\left( 1+O\left(\frac{1}{p}+\frac{r}{p^2} \right)\right) & \text{ if } p>r^{\frac{2}{3}}.
\end{cases}
\end{equation*}
We also have 
	\begin{equation*}
	\log E_p(-r)=\begin{cases}
	r \log \left(1+\frac{1}{p}+\frac{1}{p^2} \right) +O(1) & \text{ if } p <r^{\frac{2}{3}},\\
	\log \left( \frac{e^{-2r/p}+2e^{r/p}}{3}\right) +O\left(\frac{r}{p^2}\right) & \text{ if } p>r^{\frac{2}{3}},
	\end{cases}
	\end{equation*}
	and 
	\begin{equation*}
	\frac{E'_p(-r)}{E_p(-r)}=\begin{cases}
	\log \left( 1+\frac{1}{p}+\frac{1}{p^2}\right) \left( 1+O\left(\exp(-3r^{\frac{1}{3}} \right)\right) & \text{ if } p\leq r^{\frac{2}{3}},\\
	\frac{\frac{2}{p} \left(e^{r/p}-e^{-2r/p}\right)}{e^{-2r/p}+2e^{r/p}}\left( 1+O\left(\frac{1}{p}+\frac{r}{p^2} \right)\right) & \text{ if } p>r^{\frac{2}{3}}.
	\end{cases}
	\end{equation*}

\end{lem}
 \begin{proof}
We begin with \eqref{logEp-bound}.  For $p \leq r^{\frac{2}{3}}$, we have
 \begin{align} \nonumber
 E_p(r) &=     \frac{\alpha_p}{3} \left( 1-\frac{2}{p} + \frac{1}{p^2} \right)^{-r}  \left( 1 + \frac{1 - \alpha_p + \frac{2 \alpha_p}{3} \left(1+\frac{1}{p} + \frac{1}{p^2} \right)^{-r}}
{\frac{\alpha_p}{3} \left( 1-\frac{2}{p} + \frac{1}{p^2}\right)^{-r}} \right) \\ \label{exp-3}
&=     \frac{\alpha_p}{3} \left( 1-\frac{2}{p} + \frac{1}{p^2} \right)^{-r}  \left( 1 + O \left(  \exp{(- 3r^{\frac{1}{3}})} \right) \right),
\end{align}
since
\[\frac{1 - \alpha_p + \frac{2 \alpha_p}{3} \left(1+\frac{1}{p} + \frac{1}{p^2} \right)^{-r}}
{\frac{\alpha_p}{3} \left( 1-\frac{2}{p} + \frac{1}{p^2}\right)^{-r}} 
\sim 2  \left(1+\frac{1}{p} \right)^{-r} \left(1+\frac{1}{p} \right)^{-2r} \sim 2 \exp (-3r/p) .
\]
This proves \eqref{logEp-bound} when $p \leq r^{\frac{2}{3}}$.
For $p > r^{\frac{2}{3}}$, we have
\begin{align*}
E_p(r) &=  1-\alpha_p + \frac{\alpha_p}{3} \left( 1-\frac{2}{p} + \frac{1}{p^2}\right)^{-r}  + \frac{2 \alpha_p}{3} \left(1+\frac{1}{p} + \frac{1}{p^2} \right)^{-r}\\
&=1-\alpha_p + \frac{\alpha_p}{3} \exp\left(\frac{2r}{p}+O\left(\frac{r}{p^2}\right) \right) + \frac{2 \alpha_p}{3} \exp\left(-\frac{r}{p}+O\left(\frac{r}{p^2}\right) \right) \\
&=1+\frac{\alpha_p}{3} \left(e^{2r/p}+2e^{-r/p}-3\right) \exp\left(O\left(\frac{r}{p^2}\right)\right) \\
&\sim \frac{1}{3}\left(e^{2r/p}+2e^{-r/p}\right) \left(1+O\left(\frac{r}{p^2}\right)\right)
\end{align*}
and \eqref{logEp-bound} follows for this case too.

Now we have 
\begin{align*}
 E_p'(r) &=-\frac{\alpha_p}{3} \left( 1-\frac{2}{p} + \frac{1}{p^2}\right)^{-r} \log \left( 1-\frac{2}{p} + \frac{1}{p^2}\right)- \frac{2 \alpha_p}{3} \left(1+\frac{1}{p} + \frac{1}{p^2} \right)^{-r}\log  \left(1+\frac{1}{p} + \frac{1}{p^2} \right).
\end{align*}

For $p\leq r^{\frac{2}{3}}$,  we use \eqref{exp-3} and get
\begin{align*}
 \frac{E'_p(r)}{E_p(r)}=& -\log \left( 1-\frac{2}{p} + \frac{1}{p^2}\right) \left( 1+O\left(\exp(- 3r^{\frac{1}{3}} \right)\right)\\
 =&-2\log \left( 1-\frac{1}{p}\right) \left( 1+O\left(\exp(-3r^{\frac{1}{3}} \right)\right).
\end{align*}
For $p>r^{\frac{2}{3}}$, we get
\begin{align*}
 E_p'(r) &=-\frac{\alpha_p}{3} e^{2r/p}\left(1+O\left(\frac{r}{p^2}\right) \right) 
 \left( -\frac{2}{p}+O\left(\frac{1}{p^2} \right)\right)- \frac{2 \alpha_p}{3} e^{-r/p}\left(1+O\left(\frac{r}{p^2}\right) \right) \left( \frac{1}{p}+O\left(\frac{1}{p^2} \right)\right)\\
 &=\frac{2\alpha_p}{3p} \left(e^{2r/p}-e^{-r/p}\right)\left( 1+O\left(\frac{1}{p}+\frac{r}{p^2} \right)\right) \\
&=\frac{2}{3p} \left(e^{2r/p}-e^{-r/p}\right)\left( 1+O\left(\frac{1}{p}+\frac{r}{p^2} \right)\right).
\end{align*}
This gives
\begin{align*}
 \frac{E'_p(r)}{E_p(r)}=& \frac{\frac{2\alpha_p}{3p} \left(e^{2r/p}-e^{-r/p}\right)}{1+\frac{\alpha_p}{3} \left(e^{2r/p}+2e^{-r/p}-3\right)}\left( 1+O\left(\frac{1}{p}+\frac{r}{p^2} \right)\right)\\
=& \frac{\frac{2}{p} \left(e^{2r/p}-e^{-r/p}\right)}{\left(e^{2r/p}+2e^{-r/p}\right)}\left( 1+O\left(\frac{1}{p}+\frac{r}{p^2} \right)\right).
\end{align*}

We now proceed to consider the $E_p(-r)$ case. 
	For $p \leq r^{\frac{2}{3}}$, we have
	\begin{align*} \nonumber
	E_p(-r) &=     \frac{2\alpha_p}{3} \left( 1+\frac{1}{p} + \frac{1}{p^2} \right)^{r}  \left( 1 + \frac{1 - \alpha_p + \frac{ \alpha_p}{3} \left(1-\frac{2}{p} + \frac{1}{p^2} \right)^{r}}
	{\frac{2\alpha_p}{3} \left( 1+\frac{1}{p} + \frac{1}{p^2}\right)^{r}} \right) \\ 
	&=     \frac{2\alpha_p}{3} \left( 1+\frac{1}{p} + \frac{1}{p^2} \right)^{r}  \left( 1 + O \left(  \exp{(- 3r^{\frac{1}{3}})} \right) \right),
	\end{align*}
	since
	\[\frac{1 - \alpha_p + \frac{\alpha_p}{3} \left(1-\frac{2}{p} + \frac{1}{p^2} \right)^{r}}
	{\frac{2\alpha_p}{3} \left( 1+\frac{1}{p} + \frac{1}{p^2}\right)^{r}} 
	\sim \frac12  \left(1+\frac{1}{p} \right)^{-2r} \left(1+\frac{1}{p} \right)^{-r} \sim \frac12 \exp (-3r/p) .
	\]
	The rest of the proof proceeds similarly as in the case $E_p(r)$. 
	\end{proof}

Define 
\begin{align}
f(t)=\begin{cases}\label{f(t)}
             \log \left( \frac{e^{2t}+2e^{-t}}{3}\right) & \text{ if } 0\leq t <1,\\ 
             \log \left( \frac{e^{2t}+2e^{-t}}{3}\right) -2t & \text{ if } 1\leq t.\\ 
               \end{cases}
               \end{align}
               and 
              \begin{align}\label{f(t) hat}
              \widetilde{f}(t)=\begin{cases}
              \log \left( \frac{e^{-2t}+2e^{t}}{3}\right) & \text{ if } 0\leq t <1,\\ 
              \log \left( \frac{e^{-2t}+2e^{t}}{3}\right) -t & \text{ if } 1\leq t.\\ 
              \end{cases}
              \end{align} 
\begin{lem}\label{f-bound-lem}
 The functions $f(t)$ and $\widetilde{f}(t)$ are bounded on $[0,\infty)$, $f(t)=t^2+t^3/3+O(t^4)$ and $\widetilde{f}(t)=t^2-t^3/3+O(t^4)$ if $0\leq t<1$. Moreover, 
\begin{equation*}
\widetilde{f}'(t)=f'(t)=\begin{cases}
2t+O(t^2) &  \text{ if } 0\leq t <1,\\
O(e^{-3t}) & \text{ if } t\geq 1. 
\end{cases}
\end{equation*}

\end{lem}
\begin{proof}
We see that $f(t)$ is bounded on $[0,\infty)$ since $\frac{e^{2t}}{3}\leq  \frac{e^{2t}+2e^{-t}}{3} \leq e^{2t}$, and the Taylor expansion gives  $f(t)=t^2+t^3/3+O(t^4)$ if $0\leq t <1$.  
For the derivative,  if $0\leq t <1$, then
 $f'(t)=\frac{2e^{2t}-2e^{-t}}{e^{2t}+2e^{-t}}=2t+O(t^2)$, 
and if $t \geq 1$, then
 $f'(t)=\frac{-6e^{-t}}{e^{2t}+2e^{-t}}=O(e^{-3t})$.  The proof is similar for $\widetilde{f}(t)$.
\end{proof}
 
\begin{proof}[Proof of Proposition \ref{L-est-prop}] 
We will first prove \eqref{L-estimate} and \eqref{L-estimatetilde}, beginning with \eqref{L-estimate}. 

Lemmas \ref{Ep-bound-lem} and \ref{f-bound-lem} give, as in \cite[Eq.~(4.13)]{DL},
\begin{align} \nonumber \mathcal{L}(r)=&-2r \sum_{p\leq r^{\frac{2}{3}}}  \log \left(1-\frac{1}{p} \right) +\sum_{p> r^{\frac{2}{3}}}\log \left( \frac{e^{2r/p}+2e^{-r/p}}{3}\right)  + O\left(r^{\frac{2}{3}} \right)\\
\label{compute-Lr}
=&-2r \sum_{p\leq r}  \log \left(1-\frac{1}{p}\right) +\sum_{r^{\frac{2}{3}}<p<r^{\frac{4}{3}}} f\left(\frac{r}{p}\right)  + O\left(r^{\frac{2}{3}} \right).
 \end{align}
Using the Prime Number Theorem $\pi(t)-\mathrm{li}(t)\ll \frac{t}{(\log t)^3}$, together with partial summation, we obtain 
\begin{align*}
 \sum_{r^{\frac{2}{3}}<p<r^{\frac{4}{3}}}f\left(\frac{r}{p}\right)   =&\int_{r^{\frac{2}{3}}}^{r^{\frac{4}{3}}} f\left(\frac{r}{t}\right) \frac{dt}{\log t} + O\left( \frac{r}{(\log r)^2}\right)\\
 =& \frac{r}{\log r} \int_{r^{-1/3}}^{r^{\frac{1}{3}}} \frac{f(u)}{u^2} du +  O\left( \frac{r}{(\log r)^2}\right).
\end{align*}
We have 
\[ \int_{r^{-1/3}}^{r^{\frac{1}{3}}} \frac{f(u)}{u^2} du = \int_{0}^{\infty} \frac{f(u)}{u^2} du +O \left( r^{-1/3}\right),\]
and
\begin{align*}
\int_{0}^{\infty} \frac{f(u)}{u^2} du& =\int_{0}^{\infty} \frac{f'(u)}{u} du-\left(\lim_{x\to1^-}\frac{f(x)}{x}-\lim_{x\to0^+}\frac{f(x)}{x}+\lim_{x\to\infty}\frac{f(x)}{x}-\lim_{x\to1^+}\frac{f(x)}{x}\right)  \\
& =\int_{0}^{\infty} \frac{f'(u)}{u} du  - 2 \\
&= \int_0^1 \frac{2e^{2t}-2e^{-t}}{t(e^{2t}+2e^{-t})}dt+\int_1^\infty \frac{-6}{t(e^{3t}+2)}dt - 2 = 2C_{\max} - 2.
\end{align*}
 Then, collecting the above estimates and replacing in \eqref{compute-Lr}  gives
\begin{align*}\mathcal{L}(r)
=&-2r \sum_{p\leq r}  \log \left(1-\frac{1}{p}\right) + \frac{r}{\log r}\int_{0}^{\infty} \frac{f'(u)}{u} du- 2\frac{r}{\log r}+
 O\left( \frac{r}{(\log r)^2}\right) \\
 =& 2r \log \log r +2r \gamma + \frac{2r(C_{\max}-1)}{\log r}+
 O\left( \frac{r}{(\log r)^2}\right),
 \end{align*}
 where we have used the explicit form of Mertens' third theorem 
 $$
 \prod_{p \leq r} \left( 1 - \frac{1}{p} \right) = \frac{e^{-\gamma}}{\log{r}} + O \left( \frac{1}{\log^3{r}} \right)
 $$
due to Rosser and Schoenfeld \cite{Rosser-Schoenfeld}.

 We now proceed to \eqref{L-estimatetilde}, the proof is very similar to \eqref{L-estimate} but we include the details as we arrive at a different constant.

   Lemmas \ref{Ep-bound-lem} and \ref{f-bound-lem} give, as in the case $\mathcal{L}(r)$,
\begin{align*}\widetilde{\mathcal{L}}(r)=&r \sum_{p\leq r^{\frac{2}{3}}}  \log \left(1+\frac{1}{p} +\frac{1}{p^2}\right) +\sum_{p> r^{\frac{2}{3}}}\log \left( \frac{e^{-2r/p}+2e^{r/p}}{3}\right)  + O\left(r^{\frac{2}{3}} \right)\\
=&r \sum_{p\leq r}  \log \left(1+\frac{1}{p} +\frac{1}{p^2}\right) +\sum_{r^{\frac{2}{3}}<p<r^{\frac{4}{3}}}\widetilde{f}\left(\frac{r}{p}\right)  + O\left(r^{\frac{2}{3}} \right),
 \end{align*}
and 
\begin{align*}
 \sum_{r^{\frac{2}{3}}<p<r^{\frac{4}{3}}}\widetilde{f}\left(\frac{r}{p}\right)  
 &= \frac{r}{\log r} \int_{r^{-1/3}}^{r^{\frac{1}{3}}} \frac{\widetilde{f}(u)}{u^2} du +  O\left( \frac{r}{(\log r)^2}\right).
\end{align*}
We have 
\[ \int_{r^{-1/3}}^{r^{\frac{1}{3}}} \frac{\widetilde{f}(u)}{u^2} du = \int_{0}^{\infty} \frac{\widetilde{f}(u)}{u^2} du +O \left( r^{-1/3}\right),\]
and 
\begin{align*}
\int_{0}^{\infty} \frac{\widetilde{f}(u)}{u^2} du &=\int_{0}^{\infty} \frac{\widetilde{f}'(u)}{u} du-\left(\lim_{x\to1^-}\frac{\widetilde{f}(x)}{x}-\lim_{x\to0^+}\frac{\widetilde{f}(x)}{x}+\lim_{x\to\infty}\frac{\widetilde{f}(x)}{x}-\lim_{x\to1^+}\frac{\widetilde{f}(x)}{x}\right) \\
 &=
 \int_{0}^1 \frac{2e^{t}-2e^{-2t}}{t(e^{-2t}+2e^{t})}dt +\int_1^{\infty} \frac{-3e^{-2t}}{t(e^{-2t}+2e^{t})}dt-1 = C_{\min} - 1.
\end{align*}
 Then, collecting the above estimates gives
\begin{align*}\widetilde{\mathcal{L}}(r)
=&r \sum_{p\leq r} \left( \log \left(1-\frac{1}{p^3}\right)- \log \left(1-\frac{1}{p}\right)\right)+ \frac{r}{\log r}\int_{0}^{\infty} \frac{\widetilde{f}'(u)}{u} du-\frac{r}{\log r}+
 O\left( \frac{r}{(\log r)^2}\right) \\
 =& r \log \log r +r \gamma -r\log\zeta(3)+ \frac{r(C_{\min}-1)}{\log r}+
 O\left( \frac{r}{(\log r)^2}\right),
 \end{align*}
where we have applied Mertens' third theorem as before.

 We will only prove the first equation in \eqref{Lprime-estimate} as the proofs are identical for both. We have
 \begin{align*}
  \mathcal{L}'(r)=& -2\sum_{p\leq r^{\frac{2}{3}}} \log \left( 1-\frac{1}{p}\right) +\sum_{r^{\frac{2}{3}}<p}  \frac{2 \left(e^{2r/p}-e^{-r/p}\right)}{p \left(e^{2r/p}+2e^{-r/p}\right)}+O\left(r^{-1/3}\right)\\
  =& -2\sum_{p\leq r} \log \left( 1-\frac{1}{p}\right) +\sum_{r^{\frac{2}{3}}<p<r^{\frac{4}{3}}}  \frac{1}{p }f'\left(\frac{r}{p}\right)+O\left(r^{-1/3}\right),
 \end{align*}
and
\[\sum_{r^{\frac{2}{3}}<p<r^{\frac{4}{3}}}  \frac{1}{p }f'\left(\frac{r}{p}\right)=\frac{1}{\log r} \int_0^\infty \frac{f'(u)}{u} du +O\left(\frac{1}{(\log r)^2} \right),\]
 which gives
 \begin{align*}
  \mathcal{L}'(r)
 =& -2\sum_{p\leq r} \log \left( 1-\frac{1}{p}\right) +\frac{1}{\log r} \int_0^\infty \frac{f'(u)}{u} du +O\left(\frac{1}{(\log r)^2} \right)\\
  =& 2 \log \log r +2\gamma+\frac{2C_{\max}}{\log r} +O\left(\frac{1}{(\log r)^2} \right).
 \end{align*}
The approach for the results in \eqref{Lhigherderivs-estimate} is the same, although the calculations are more technical. 
\end{proof}

The following result corresponds to  Lemma 4.5 from \cite{DL}.
 \begin{lem} \label{DL-lemma4.5} 
  Let $r$ be large. If $p\gg r$, then for some positive constant $b_1$ we have 
  \[\frac{|E_p(r+it)|}{E_p(r)},\frac{|E_p(-r+it)|}{E_p(-r)}\leq \exp \left( -b_1\left( 1- \cos\left(t \log \left(\frac{p^2+p+1}{p^2-2p+1}\right)\right)\right) \right).\]
 \end{lem}
\begin{proof}
 Let $x_1,x_2,x_3$ be positive real numbers and $\theta_2,\theta_3$ be real numbers. The following inequality is established in the proof of \cite[Lemma 3.2]{GS}:
 \[|x_1+x_2e^{i\theta_2}+x_3e^{i\theta_3}|\leq (x_1+x_2+x_3) \exp \left(-\frac{x_1x_3(1-\cos \theta_3)}{(x_1+x_2+x_3)^2}\right).\]
 We apply the above inequality to $x_1=\frac{\alpha_p}{3}\left(1-\frac{2}{p}+\frac{1}{p^2} \right)^{-r}$, $x_2=1-\alpha_p$, $x_3=\frac{2\alpha_p}{3}\left(1+\frac{1}{p}+\frac{1}{p^2} \right)^{-r}$ and $\theta_2=t\log \left(1-\frac{2}{p}+\frac{1}{p^2} \right)$, and $\theta_3=t\log \left(\frac{p^2-2p+1}{p^2+p+1}\right)$. We conclude since $\cos(\log(x))=\cos(\log(1/x))$. The proof for $-r$ is completely analogous. 
\end{proof}

\begin{lem} \label{DL-lemma4.6} Let $r$ be large. Then, there exists a constant $b_2 >0$ such that
\begin{align*}
\frac{\left\vert \E \left( |L(1, \X)|^{r + it} \right) \right\vert}{ \E \left( |L(1, \X)|^{r} \right)},\frac{\left\vert \E \left( |L(1, \X)|^{-r + it} \right) \right\vert}{ \E \left( |L(1, \X)|^{-r} \right)} \ll \begin{cases}
\exp \left( - b_2 \frac{t^2}{r \log r} \right) & \text{if $|t| \leq r/4$,} \\ \exp \left( - b_2 \frac{|t|}{ \log |t|} \right) & \text{if $|t| > r/4$.} \end{cases}
\end{align*}
\end{lem}
\begin{proof} This follows from Lemma \ref{DL-lemma4.5} with exactly the same proof as in \cite[Lemma 4.6]{DL}. \end{proof}

\begin{lem} \label{approx-phi}
For any $\lambda>0$, let
\begin{align} \label{def-f}
f(y)&:=\frac{1}{2\pi i}\int_{(c)}y^s \frac{e^{\lambda s}-1}{\lambda s}\frac{ds}{s}. \end{align}
Then,
\begin{align*} 
f(y)
&=\frac{1}{\lambda}\int_{0}^\lambda \left(\frac{1}{2\pi i}\int_{(c)}(ye^u)^s\frac{ds}{s}\right)du=
\begin{cases}
1, &\mbox{ if } y>1,\\ 
1+\frac{\log y}{\lambda}, &\mbox{ if } e^{-\lambda}\leq y\leq 1,\\
0, &\mbox{ if } y<e^{-\lambda},
\end{cases}
\end{align*}
and we have
\begin{align*}
f(ye^{-\lambda})\leq \phi(y)\leq f(y),
\end{align*}
where $\phi(y)$ is the indicator function
$$
\phi(y) := \begin{cases}
1, &\mbox{if $y > 0$} ,\\
0, &\mbox{otherwise}.
\end{cases}
$$
\end{lem}

\begin{proof}[Proof of Theorem \ref{main thm for Distribution of RV}]

We will prove Theorem  \ref{main thm for Distribution of RV} with the saddle-point technique, from which we deduce that $\Phi(\tau)$ is given by a formula involving $\E (|L(1, \X)|^{2r})$ evaluated at the saddle point 
$\kappa  = \kappa(\tau)$ defined below. For $\Psi(\tau)$, the distribution is related to  $\E (|L(1, \X)|^{-2r})$ evaluated at the saddle point $\tilde{\kappa}  = \tilde{\kappa}(\tau)$.

Let $\tau$ be a large real number and consider the differential equations 
\begin{equation}\label{saddlept}
\left(\mathbb{E}(|L(1,\mathbb{X})|^{2r})(e^{\gamma}\tau)^{-2r}\right)'=0 \Leftrightarrow \mathcal{L}'(r)=2(\log \tau +\gamma),
\end{equation}
and 
\begin{equation}\label{saddleptmin}
\left(\mathbb{E}(|L(1,\mathbb{X})|^{-2r})\left(\frac{e^{\gamma}\tau^{2}}{ \zeta(3)}\right)^{-r}\right)'=0 \Leftrightarrow \widetilde{\mathcal{L}}'(r)=2 \log \tau +\gamma -\log \zeta(3),
\end{equation}
where the derivative is taken with respect to the real variable $r$.  It follows from \eqref{Lprime-estimate} that $\lim_{r\to\infty}\mathcal{L}'(r)=\infty$. Moreover, a straightforward calculation shows that 
\begin{align} \label{simple-calc} E''_p(r)E_p(r)>(E'_p(r))^2 \end{align}
for all primes $p$, so that $\mathcal{L}''(r)>0.$ The same calculation leads to \eqref{simple-calc}
with $r$ replaced $-r$, and then we also have that
$\widetilde{\mathcal{L}}''(r) > 0$. Thus, we deduce that \eqref{saddlept} and \eqref{saddleptmin} have unique solutions $\kappa:=\kappa(\tau)$ and $\widetilde{\kappa}:=\widetilde{\kappa}(\tau)$ respectively.

We  write $\Phi(\tau)$ as 
\[\Phi(\tau) = \P \left( (|L(1, \X)|^2 /(e^\gamma \tau))^2> 1\right).\]
Let $A=\frac{|L(1, \X)|^2}{e^{2\gamma} \tau^2}>1$ be an event. Then  \begin{align*}
\phi(A)=\begin{cases}
1, &\mbox{if the event $A$ occurs} ,\\
0, &\mbox{otherwise},
\end{cases}
\end{align*}
is a random variable and $\E(\phi(A))=\Phi(\tau)=\P(A).$

By Lemma \ref{approx-phi}, where $f$ is the function defined by
\eqref{def-f}, we  have 
\begin{align}\label{Ine 1}
\E(f(Ae^{-\lambda}))\leq \E(\phi(A))=\Phi(\tau)\leq \E(f(A)),\end{align}
since $X\leq Y$ implies that $\E(X)\leq \E(Y)$.
Using the definition of $f$ and $A$, we write
\begin{align*}
\E(f(Ae^{-\lambda}))&=\E\left( \frac{1}{2\pi i}\int_{(c)}\frac{|L(1, \X)|^{2s}}{(e^{2\gamma}\tau^2)^s} e^{-\lambda s}\left(\frac{e^{\lambda s}-1}{\lambda s}\right)\frac{ds}{s}\right)\\
&= \frac{1}{2\pi i}\int_{(c)}\frac{\mathbb{E}(|L(1, \X)|^{2s})}{(e^{2\gamma}\tau^2)^s} e^{-\lambda s}\left(\frac{e^{\lambda s}-1}{\lambda s}\right)\frac{ds}{s},
\end{align*}
where $0 < \lambda < 1/(2 \kappa)$ is a real number to the chosen later. With the above equation and  \eqref{Ine 1}, we have 
\begin{align*}
&\frac{1}{2\pi i}\int_{(\kappa)}\frac{\E(|L(1, \X)|^{2s})}{(e^{2\gamma}\tau^2)^s}e^{-\lambda s}\left(\frac{e^{\lambda s}-1}{\lambda s}\right)\frac{ds}{s}\\
&\nonumber \hspace{1cm}  \leq \Phi(\tau) \leq \frac{1}{2\pi i}\int_{(\kappa)}\frac{\E(|L(1, \X)|^{2s})}{(e^{2\gamma}\tau^2)^s}\left(\frac{e^{\lambda s}-1}{\lambda s}\right)\frac{ds}{s},
\end{align*}
and it follows that
\begin{align*}
0 &\leq \frac{1}{2 \pi i} \int_{\kappa-i \infty}^{\kappa+i \infty} \E \left( |L(1, \X)|^{2s} \right) (e^\gamma \tau)^{-2s} \left(\frac{e^{\lambda s} -1}{\lambda s}\right) \frac{ds}{s} - \Phi(\tau) \\
&\leq \frac{1}{2 \pi i} \int_{\kappa-i \infty}^{\kappa+i \infty} \E \left( |L(1, \X)|^{2s} \right) ( e^\gamma \tau)^{-2s} \left(\frac{e^{\lambda s} -1}{\lambda s}\right) \left(\frac{1 - e^{-\lambda s}}{s}\right) ds.
\end{align*}

We remark that  since $\lambda \kappa < \frac{1}{2}$ by the choice of $\lambda$, we have $|e^{\lambda s} - 1| < 3$ and $|e^{-\lambda s} - 1| < 2$.
For $T\ll \kappa$ 
(to be chosen later), using Lemma \ref{DL-lemma4.6}, we have
\begin{align*}
&\int_{\kappa + i T}^{\kappa + i \infty} \E \left( |L(1, \X)|^{2s} \right) (e^\gamma \tau)^{-2s}\left( \frac{e^{\lambda s} -1}{\lambda s}\right) \frac{ds}{s}\\
&\ll \frac{\E \left( |L(1, \X)|^{2\kappa} \right) ( e^\gamma \tau)^{-2\kappa}}{\lambda}\left(\int_{T}^{\kappa/2}   \exp \left( - b_2 \frac{t^2}{\kappa \log \kappa} \right) \frac{dt}{\kappa^2 + t^2} + \int_{\kappa/2}^\infty   \exp \left( - b_2 \frac{t}{\log t} \right)  \frac{dt}{\kappa^2 + t^2} \right)\\
&\ll \frac{\E \left( |L(1, \X)|^{2\kappa} \right) (e^\gamma \tau)^{-2\kappa}}{\lambda}\left(\frac{e^{-c_1\frac{T^2}{\kappa \log \kappa}}}{T}+\frac{e^{-c_2\frac{\kappa}{\log \kappa}}}{\kappa}\right),
\end{align*}
with some absolute constants $c_1$ and $c_2$.

Furthermore, when $|t| \leq T$, and $s = \kappa + it$ then $| (1 - e^{-\lambda s}) (e^{\lambda s} -1) | \ll \lambda^2 |s|^2$, and we have
\begin{align*}
\int_{\kappa-i T}^{\kappa+i T}  \E \left( |L(1, \X)|^{2s} \right) (e^\gamma \tau)^{-2s} \left(\frac{e^{\lambda s} -1}{\lambda s} \right)\left( \frac{1-e^{-\lambda s}}{s} \right)
ds \ll \lambda T \E \left( |L(1, \X)|^{2\kappa} \right) ( e^\gamma \tau)^{-2\kappa},
\end{align*}
which gives
 \begin{align} \label{approx-Phi}
 \Phi(\tau)  - \frac{1}{2 \pi i} \int_{\kappa-i T}^{\kappa+i T} \E \left( |L(1, \X)|^{2s} \right) (e^\gamma \tau)^{-2s} \left(\frac{e^{\lambda s} -1}{\lambda s}\right) \frac{ds}{s} \\ \nonumber
\ll  \left(\frac{e^{-c_1\frac{T^2}{\kappa \log \kappa }}}{\lambda T}+\frac{e^{-c_2\frac{\kappa}{\log \kappa}}}{\lambda \kappa} +\lambda T\right)\E \left( |L(1, \X)|^{2\kappa} \right) ( e^\gamma \tau)^{-2\kappa}.
 \end{align}
 
 We now give a bound for the integral between $\kappa - iT$ and $\kappa + iT$.
 It follows from \eqref{Lhigherderivs-estimate} of Proposition \ref{L-est-prop}  that when $|t| \leq T$, 
 \begin{align*}
 \cL(\kappa+it) = \cL(\kappa) + it \cL'(\kappa) - \frac{t^2 \cL''(\kappa)}{2} + O \left(   \frac{|t|^3}{\kappa^2 \log{\kappa}} \right).
 \end{align*}
 We also have
 \begin{align*}
 \frac{e^{\lambda s}-1}{\lambda s^2} = \frac{1}{s} \left( 1 + O \left( \lambda \kappa \right) \right) = \frac{1}{\kappa} \left( 1 - i \frac{t}{\kappa} + O \left( \lambda \kappa + \frac{t^2}{\kappa^2} \right) \right).
 \end{align*}
 Writing $\E \left( |L(1, \X ) |^{2s}\right) = \exp \left( \cL(s) \right)$ and $\cL'(\kappa) = 2(\log{\tau} + \gamma)$, and using the above two equations, we have
 \[
 \frac{\E \left( |L(1, \X)|^{2s} \right)}{(e^\gamma \tau)^{2s}}
\left(\frac{e^{\lambda s}-1}{\lambda s^2}\right)=\frac{\E(|L(1, \X)|^{2\kappa})}{\kappa(e^\gamma \tau)^{2\kappa}}\exp\left(-\frac{t^2}{2}\cL''(\kappa)\right)\left(1-i\frac{t}{\kappa}+O\left(\lambda \kappa +\frac{t^2}{\kappa^2}+\frac{|t|^3}{\kappa^2 \log \kappa}\right)\right),
 \]
 since
 \begin{align*} (e^\gamma \tau)^{-2s} \exp{\left( i t \cL'(\kappa) \right)} = (e^\gamma \tau)^{-2 \kappa - 2 i t} \exp{\left( 2 i t \log \tau \right)} \exp{\left( 2 i t \gamma \right)}= ( e^\gamma \tau)^{-2 \kappa}.
 \end{align*}
 Therefore,
 \begin{align*}
& \frac{1}{2\pi i}\int_{\kappa-i T}^{\kappa+iT}\frac{\E \left( |L(1, \X)|^{2s} \right)}{( e^\gamma \tau)^{2s}} \left(\frac{e^{\lambda s}-1}{\lambda s} \right)\; \frac{ds}{s} \\
&=\frac{\E(|L(1, \X)|^{2\kappa})}{\kappa(e^\gamma \tau)^{2\kappa}}\frac{1}{2\pi} \int_{-T}^{T}\exp\left(-\frac{t^2}{2}\cL''(\kappa)\right)\left(1-i\frac{t}{\kappa}+O\left(\lambda \kappa 
+\frac{t^2}{\kappa^2}+\frac{|t|^3}{\kappa^2 \log \kappa}\right)\right)dt.
 \end{align*}
 Since $\cL''(\kappa) \asymp 1/(\kappa \log{\kappa})$ by \eqref{Lhigherderivs-estimate} of Proposition \ref{L-est-prop}, we have
\begin{align*}
 \frac{1}{2\pi} \int_{-T}^{T}\exp\left(-\frac{t^2}{2}\cL''(\kappa)\right)dt &=\frac{1}{\sqrt{2\pi \cL''(\kappa)}}\left(1+O\left(e^{-c_3T^2\cL''(\kappa)}\right)\right),\\
 \int_{-T}^T |t|^n\exp\left(-\frac{t^2}{2}\cL''(\kappa)\right)dt &\ll \frac{1}{\cL''(\kappa)^{\frac{n+1}{2}}} \hspace{1cm} 1 \leq n \leq 3,
 \end{align*}
 and replacing above, we get
 \begin{align*}
 &\frac{1}{2\pi i}\int_{\kappa-i T}^{\kappa+iT}\frac{\E \left( |L(1, \X)|^{2s} \right)}{( e^\gamma \tau)^{2s}} \left(\frac{e^{\lambda s}-1}{\lambda s}\right)\frac{ds}{s}\\
 &=\frac{\E(|L(1, \X)|^{2\kappa})}{\kappa \sqrt{2\pi \cL''(\kappa)}( e^\gamma \tau)^{2\kappa}}\left(1+ O\left(\lambda \kappa+ \frac{1}{\kappa^2\cL''(\kappa)}+\frac{1}{\kappa^2 \log \kappa \cL''(\kappa)^{\frac{3}{2}}}+e^{-c_3T^2\cL''(\kappa)}\right)\right),
 \end{align*}
 for some absolute constant $c_3$.
Replacing in \eqref{approx-Phi}, and using again \eqref{Lhigherderivs-estimate}  of Proposition \ref{L-est-prop} to get
$$
\kappa \sqrt{2 \pi \mathcal{L}''(\kappa)} \ll \frac{\kappa^{\frac{1}{2}}}{(\log{\kappa})^{\frac{1}{2}}},
$$
we finally obtain
 \begin{align*}
 \Phi(\tau)&=\frac{\E(|L(1, \X)|^{2\kappa})}{\kappa \sqrt{2\pi \cL''(\kappa)}( e^\gamma \tau)^{2\kappa}}\left(1+ O\left(\lambda \kappa+ \frac{1}{\kappa^2\cL''(\kappa)}+\frac{1}{\kappa^2 \log \kappa \cL''(\kappa)^{\frac{3}{2}}}+e^{-c_3T^2\cL''(\kappa)}\right. \right.\\
 & \left.\left. +\frac{e^{-c_1\frac{T^2}{\kappa \log \kappa}}}{\lambda T}\frac{\kappa^{\frac{1}{2}}}{(\log \kappa)^{\frac{1}{2}}}+\frac{e^{-c_2\frac{\kappa}{\log \kappa}}}{\lambda \kappa}\frac{\kappa^{\frac{1}{2}}}{(\log \kappa)^{\frac{1}{2}}} +\lambda T\frac{\kappa^{\frac{1}{2}}}{(\log \kappa)^{\frac{1}{2}}} \right)\right).
 \end{align*}
We choose  $T=\kappa^{\delta}$ for any $1/2<\delta<1$ and $\lambda=\kappa^{-2}$, which can be done by taking $\tau$ sufficiently large by \eqref{tau and kappa connection}. We conclude that 
 \[
 \Phi(\tau)=\frac{\E(|L(1, \X)|^{2\kappa})}{\kappa \sqrt{2\pi \cL''(\kappa)}(e^\gamma \tau)^{2\kappa}}\left(1+ O\left(\sqrt{\frac{\log \kappa}{\kappa}}\right)\right).
 \]
The proof is similar for $\Psi(\tau)$. This completes the proof of Theorem \ref{main thm for Distribution of RV}.

\end{proof}

\subsection{Proofs of Theorems \ref{thm-RV} and \ref{theorem-relating-to-RV}.}
\label{proofof-1.2-1.3}

\begin{proof}[Proof of Theorem \ref{thm-RV}]
Recall that $\E\left(|L(1, \X)|^{2r}\right)=\exp\left( \cL(r) \right)$, and from Proposition \ref{L-est-prop}, we have
\begin{align*}
\cL(r)&=2r \log \log r +2 \gamma r +\frac{2r (C_{\max}-1) }{\log r}+O\left(\frac{r}{(\log r)^2}\right), \\
\cL'(r)&=2 \log \log r +2 \gamma +\frac{2C_{\max} }{\log r}+O\left(\frac{1}{(\log r)^2}\right).
\end{align*}
Recall that $\kappa = \kappa(\tau)$ is the unique solution to $\cL'(\kappa)=2(\log \tau+\gamma)$. Therefore, 
\begin{align}\label{tau and kappa connection}
\log \tau =\log \log \kappa +\frac{C_{\max}}{\log \kappa}+O\left(\frac{1}{(\log \kappa)^2}\right).
\end{align}
Using these estimates and Theorem \ref{main thm for Distribution of RV}, we obtain
\begin{align*}
 \Phi(\tau)&=\frac{\E(|L(1, \X)|^{2\kappa})}{\kappa \sqrt{2\pi \cL''(\kappa)}(e^\gamma \tau)^{2\kappa}}\left(1+ O\left(\sqrt{\frac{\log \kappa}{\kappa}}\right)\right)\\
 &= \exp\left(\cL(\kappa)-2\kappa(\log \tau + \gamma)+O(\log \kappa)\right)\\
 &=\exp\left(-\frac{2\kappa}{\log \kappa}+O\left(\frac{\kappa}{(\log \kappa)^2}\right)\right)=\exp\left(-\frac{2e^{\tau-C_{\max}}}{\tau}\left(1+O(\tau^{-1})\right)\right),
\end{align*}
since $ \log \kappa = \tau - C_{\max} + O(\tau^{-1} )$ from \eqref{tau and kappa connection}.

The proof is similar for $\Psi(\tau)$ by noting that 
$\E\left(|L(1, \X)|^{-2r}\right)=\exp\left( \widetilde{\cL}(r) \right)$, where 
\begin{align*}
\widetilde{\cL}(r)&= r\log \log r +r \gamma + \frac{r(C_{\min}-1)}{\log r}-r\log \zeta(3)+
O\left( \frac{r}{(\log r)^2}\right),\\
\widetilde{\cL}'(r)&= \log \log r + ( \gamma - \log{\zeta(3)}) +\frac{C_{\min} }{\log r}+O\left(\frac{1}{(\log r)^2}\right).
\end{align*}
Let $\widetilde{\kappa}$ be the unique solution to $\widetilde{\cL}'(\widetilde{\kappa})=2 \log \tau+\gamma -\log \zeta(3))$.
Then, 
\begin{align}\label{tau and kappa connection-2}
2\log \tau
 =\log \log \tilde{\kappa} +\frac{C_{\min}}{\log \tilde{\kappa}}+O\left(\frac{1}{(\log \tilde{\kappa})^2}\right).
\end{align}
Using these estimates and Theorem \ref{main thm for Distribution of RV}, we obtain
\begin{align*}
 \Psi(\tau)&=\frac{\E(|L(1, \X)|^{-2\tilde{\kappa}})\zeta(3)^{\tilde{\kappa}}}{\tilde{\kappa} \sqrt{2\pi \cL''(\kappa)}(e^\gamma \tau^2)^{\tilde{\kappa}}}\left(1+ O\left(\sqrt{\frac{\log \tilde{\kappa}}{\tilde{\kappa}}}\right)\right)\\
 &= \exp\left(\widetilde{\cL}(\tilde{\kappa})-\tilde{\kappa}(2\log \tau + \gamma-\log \zeta(3))+O(\log \tilde{\kappa})\right)\\
 &=\exp\left(-\frac{\tilde{\kappa}}{\log \tilde{\kappa}}+O\left(\frac{\tilde{\kappa}}{(\log \tilde{\kappa})^2}\right)\right)=\exp\left(-\frac{e^{\tau^2-C_{\min}}}{ \tau^2}\left(1+O(\tau^{-1})\right)\right),
\end{align*}
since $ \log \tilde{\kappa} =  \tau^2 - C_{\min} + O(\tau^{-2} )$ from \eqref{tau and kappa connection-2}.

\end{proof}

	\begin{proof}[Proof of Theorem \ref{theorem-relating-to-RV}]

	We adapt the  proof of \cite[Theorem 1]{GS06} to our context.
		For any $r\geq 1$, we have
		\begin{align*}
		2r\int_{0}^{\infty}t^{2r-1}\phi_X(t)dt&=\frac{2r}{|\mathcal{F}_3(X)|} \sum_{\chi\in \mathcal{F}_3(X)}\int_{0}^{|L(1, \chi)|e^{-\gamma}}t^{2r-1}dt=\frac{e^{-2r\gamma}}{|\mathcal{F}_3(X)|}\sum_{\chi\in \mathcal{F}_3(X)}|L(1, \chi)|^{2r}.
		\end{align*}
		Then, by Theorem \ref{Asympformula}, uniformly for $r\leq \frac{\log X}{16 \log_2 X \log_3 X}$, we get 
		\begin{align*}
		&2r \int_{0}^{\infty}t^{2r-1}\phi_X(t)dt=e^{-2\gamma r} E\left(|L(1, \X)|^{2r}\right)+e^{-2\gamma r}\exp\left(-\frac{\log X}{16\log_2X}\right),
		\end{align*}
		 and by Proposition \ref{L-est-prop}, we have
		 \begin{align*}
		 E\left(|L(1, \X)|^{2r}\right)&=\exp\left(2r \log \log r +2r \gamma + \frac{2r(C_{\max}-1)}{\log r}+
		 O\left( \frac{r}{(\log r)^2}\right)\right)\\
		 &=(\log r)^{2r}e^{2\gamma r} \exp\left(\frac{2r(C_{\max}-1)}{\log r}+
		 O\left( \frac{r}{(\log r)^2}\right)\right).
		 \end{align*}
		
		This implies that 
		\begin{align}\label{asymp for kappa}
		\int_{0}^{\infty}t^{2r-1}\phi_X(t)dt &=\frac{(\log r)^{2r}}{2r}\exp\left(\frac{2r}{\log r}(C_{\max}-1)+O\left(\frac{r}{(\log r)^2}\right)\right)+\frac{1}{2r}\exp\left(-2\gamma r-\frac{\log X}{16\log_2X}\right) \nonumber\\
		&=(\log r)^{2r}\exp\left(\frac{2r}{\log r}(C_{\max}-1)+O\left(\frac{r}{(\log r)^2}\right)\right).
		\end{align}

		Let  $r_2=r e^{\delta}$, where $\delta>0$ is sufficiently small to be determined later. We apply  \eqref{asymp for kappa} with $r_2$ to obtain  
		\begin{align*}
		\int_{\tau+\delta}^{\infty}t^{2r-1}\phi_X(t)dt&\leq (\tau +\delta)^{2r-2r_2}\int_{\tau+\delta}^{\infty}t^{2r_2-1}\phi_X(t)dt  \leq (\tau +\delta)^{2r-2r_2}\int_{0}^{\infty}t^{2r_2-1}\phi_X(t)dt \\
		&= (\tau +\delta)^{2r(1-e^{\delta})}(\log r+\delta)^{2r e^{\delta}}\exp\left(\frac{2r e^{\delta}}{\log r +\delta}(C_{\max}-1)+O\left(\frac{r}{(\log r)^2}\right)\right) \\
		&\leq \tau^{2r(1-e^{\delta})}(\log r)^{2re^{\delta}}  \exp\left(2r(1-e^{\delta})\log \left(1+\frac{\delta}{\tau}\right)+2r e^{\delta}\log \left(1+\frac{\delta}{\log r}\right)\right) \\
		& \times \exp\left(\frac{2r e^{\delta}}{\log r}\left(C_{\max}-1\right)+O\left(\frac{r}{(\log r)^2}\right)\right) \\
		&=  (\log{r} )^{2 r} \exp\left(2r(1-e^{\delta})\log \left(1+\frac{\delta}{\tau}\right)+2r e^{\delta}\log \left(1+\frac{\delta}{\log r}\right)\right) \\
		& \times \exp\left(\frac{2r e^{\delta}}{\log r}\left(C_{\max}-1\right)+O\left(\frac{r}{(\log r)^2}\right)\right)  \exp \left( 2r (1-e^\delta) \left( \log{\tau} - \log_2{r} \right) \right), \end{align*}
		which holds uniformly for 
		\begin{align} \label{uniform-range}
		r_2 \leq \frac{\log X}{16 \log_2 X \log_3 X} \iff r\leq  e^{-\delta}\frac{\log X}{16 \log_2 X \log_3 X}.
		\end{align}
		We now choose $r = r(\tau)$ by $\log{r} = \tau - C_{\max}$. This gives
				\begin{align*}
		\int_{\tau+\delta}^{\infty}t^{2r-1}\phi_X(t)dt&\leq (\log r)^{2r} \exp\left(2r(1-e^{\delta})\log \left(1+\frac{\delta}{\tau}\right)+2r e^{\delta}\log \left(1+\frac{\delta}{\log r}\right)\right) \\
		& \times \exp \left( \frac{2r}{\log{r}} \left( C_{\max} - e^{\delta} + O \left( \frac{1}{\log{r}} \right) \right) \right)\\
		&\leq (\log r)^{2r}\exp\left(\frac{2r}{\log r}\left(C_{\max}+\delta-e^{\delta}\right)+O\left(\frac{r}{(\log r)^2}\right)\right).
		\end{align*}
		Using $\delta=\frac{c}{\sqrt{\log r}}$ in the last equation, and the fact that  $1 + \delta - e^\delta  = - \frac12 \delta^{2} + O(\delta^3)$, the last line becomes
						\begin{align*}
		\int_{\tau+\delta}^{\infty}t^{2r-1}\phi_X(t)dt& \leq (\log r)^{2r}\exp\left(\frac{2r}{\log r}\left(C_{\max}-1\right)+O\left(\frac{r}{(\log r)^2}\right)\right) \exp \left( {-\frac{c^2r}{(\log r)^2}} \right)		\end{align*}
		adjusting the constant $c$. Using \eqref{asymp for kappa}, we have proven that
		\begin{align} \label{1st ineqC}
		\int_{\tau+\delta}^{\infty}t^{2r-1}\phi_X(t)dt &\leq \left(\int_{0}^{\infty}t^{2r-1}\phi_X(t)dt\right) \exp \left( {-\frac{c^2r}{(\log r)^2}} \right),
		\end{align}
	and similarly, we can prove that
		\begin{align}\label{2nd ineqC}
		\int_0^{\tau-\delta}t^{2r-1}\phi_X(t)dt\leq \left(\int_{0}^{\infty}t^{2r-1}\phi_X(t)dt\right) \exp \left( {-\frac{c^2r}{(\log r)^2}} \right).
		\end{align}
		The exponential decay of \eqref{1st ineqC} and \eqref{2nd ineqC} implies that most of the mass of the integral is between $\tau - \delta$ and $\tau+\delta$ when $r$ grows.		
		More precisely, {again using \eqref{asymp for kappa}}, we have
		\begin{align} 
		\int_{\tau-\delta}^{\tau+\delta}t^{2r-1}\phi_X(t)dt=& \int_0^{\infty} t^{2r-1}\phi_X(t)dt \left( 1 + O \left( \exp \left( {-\frac{c^2r}{(\log r)^2}} \right)\right) \right) \nonumber\\
		&= (\log r)^{2r}\exp\left(\frac{2r}{\log r}(C_{\max}-1)+O\left(\frac{r}{(\log r)^2}\right)\right) \nonumber\\ & \;\;\;\;\;\; \times \exp\left( \log \left( 1 + O \left( \exp \left( {-\frac{c^2r}{(\log r)^2}} \right)\right) \right) \right) \nonumber \\
		&= (\log r)^{2r}\exp\left(\frac{2r}{\log r}(C_{\max}-1)+O\left(\frac{r}{(\log r)^2}\right)\right). \label{main distn term}
		\end{align}
We have		\[
		\int_{\tau-\delta}^{\tau+\delta}t^{2r-1}dt=\frac{(\tau + \delta)^{2r} - (\tau-\delta)^{2r}}{2r} = \tau^{2r} \exp\left(O\left(\frac{\delta r}{\tau}\right)\right),
		\]
		and since $\phi_X(t) = \P \left( |L(1, \chi)| > e^\gamma t \right)$ is non-increasing, this gives
		\[
		\tau^{2r}\exp\left(O\left(\frac{\delta r}{\tau}\right)\right)\phi_X(\tau+\delta)\leq \int_{\tau-\delta}^{\tau+\delta}t^{2r-1}\phi_X(t)dt \leq \tau^{2r}\exp\left(O\left(\frac{\delta r}{\tau}\right)\right)\phi_X(\tau-\delta).
		\]
		Replacing the middle term by  \eqref{main distn term}, we have 
		\[
		\phi_X(\tau+\delta)\leq \left(\frac{\log r}{\tau}\right)^{2r}\exp\left(\frac{2r}{\log r}(C_{\max}-1)+O\left(\frac{r}{(\log r)^2}\right)+O\left(\frac{\delta r}{\tau}\right)\right)\leq \phi_X(\tau-\delta).
		\]
			Since $\tau=\log r+C_{\max} \iff r=e^{\tau- C_{\max}}$, we have that the middle expression above is 
		\begin{align*}
		&=\exp\left(\frac{2r}{\log r}(C_{\max}-1) + 2r \log \left(1-\frac{C_{\max}}{\tau} \right)+O\left(\frac{\delta r}{\tau}\right)\right)\\
		& = \exp\left(\frac{2r}{\log r}(C_{\max}-1)-\frac{2rC_{\max}}{\tau}+O\left(\frac{r}{\tau^2}\right)+O\left(\frac{\delta r}{\tau}\right)\right)\\
		&=\exp\left(-\frac{2r}{\log r}+O\left(\frac{r}{\tau^2}\right)+O\left(\frac{\delta r}{\tau}\right)\right)\\
			&=\exp\left(-\frac{2e^{\tau-C_{\max}}}{\tau}\left(1+O\left(\delta\right)+O(1/\tau)\right)\right).
\end{align*}		Therefore,
		 \[
		\phi_X(\tau+\delta)\leq \exp\left(-\frac{2e^{\tau-C_{\max}}}{\tau}\left(1+O\left(\delta\right)\right)\right)\leq \phi_X(\tau-\delta),
		\]
		and we have
				\[
		\phi_{X}(\tau)=\exp\left(-\frac{2e^{\tau-C_{\max}}}{\tau}\left(1+O\left(\frac{1}{\tau^{\frac{1}{2}}}\right)\right)\right).
		\]
				Recalling \eqref{uniform-range}, this holds uniformly in the range
		\begin{align*}
		\tau&=\log r+C_{\max}\leq \log \left( e^{-\delta}\frac{\log X}{16 \log_2 X \log_3 X}\right) + C_{\max} \\
&\leq \log_2 X-\log_3 X-\log_4 X-2.
		\end{align*}
		A similar estimate holds for  $\psi_{X}(\tau)$.
			\end{proof}	

\section{Extreme values of $|L(1,\chi)|$} \label{sec:omega}
In this section, we investigate extreme values of $|L(1,\chi)|$ (under GRH). Our first goal is to prove Proposition \ref{thm:littlewood-gen},  exhibiting the upper and lower bounds for characters of prime order $\ell \geq 3$
following the work of Littlewood \cite{Littlewood-1928} for quadratic characters.

We will focus on the lower bound, since the upper bound can be obtained by the same arguments given by Littlewood. 
We first recall the following result. 
\begin{lem}\cite[Lemma 7]{Littlewood-1928}
 Let $\chi$ be a non-principal character of conductor $q$. Assume GRH for $L(s, \chi)$. Then, as $y\rightarrow 0$, 
\begin{align}\label{log equation at 1}
-\log |L(1, \chi)|=-\re\left(\sum_{n=1}^\infty\frac{\Lambda_1(n)\chi(n)}{n}e^{-ny}\right)+O_\epsilon(y^{1/2-\epsilon}\log q)+o(1),
\end{align}
where $\Lambda_1(n)=\frac{\Lambda(n)}{\log n}$.
\end{lem}

\begin{proof}[Proof of Proposition \ref{thm:littlewood-gen}]
 To prove the lower bound, we follow the proof of  \cite[Theorem 1]{Littlewood-1928}. Taking $y=(\log q)^{-2-7\epsilon}$ and $x=(\log q)^{2+8\epsilon}$ in \eqref{log equation at 1}, we have that
 \[O_\epsilon(y^{1/2-\epsilon}\log q)=o(1)\qquad \mbox{ and } \qquad  \sum_{n\geq x}\frac{\Lambda_1(n)\chi(n)}{n}e^{-ny}=o(1),
\] 
and we obtain
 \begin{align*}
 -\log |L(1, \chi)|&=-\re\left(\sum_{p^m\leq x}\frac{\Lambda_1(p^m)\chi(p^m)}{p^m}e^{-p^my}\right)+o(1)\\
 &=-\re \left(\sum_{p\leq x}\frac{\chi(p)}{p}e^{-py}\right)-\re\left(\sum_{\substack{p^m\leq x\\ m>1}}\frac{\chi(p)^m}{mp^m}e^{-p^my}\right)+o(1)\\
 &= -\sum_{p\leq x} \left(\frac{\re(\chi(p))}{p}e^{-py}+\sum_{m=2}^{M_p}\frac{\re(\chi(p))^m)}{mp^m}\right)+o(1)\\
 &=:-\sum_{p\leq x} A_p+o(1),
 \end{align*}
 where $M_p=[\log x/\log p]$ and we have used the fact that 
 \[\sum_{\substack{p^m\leq x\\ m>1}}\frac{1-e^{-p^my}}{mp^m}=O(yx^{\frac{1}{2}})=o(1).\]
 Note that $\chi(p)\in \{0, 1, \omega_\ell, \dots, \omega_\ell^{\ell-1}\}$. If $\chi(p)=0$ or 1,   then $A_p\geq  0$. 
 Let $\chi(p)=\omega_\ell^h$. Then $\chi(p)^m=\omega_\ell^{hm}$ for any $m\geq 1$. Thus,
 \begin{align*}
 A_p=&\frac{(\omega_\ell^h+\omega_\ell^{-h})(e^{-py}-1)}{2p}+\frac{1}{2} \left(\frac{\omega_\ell^h+\omega_\ell^{-h}}{p}+ \frac{\omega_\ell^{2h}+\omega_\ell^{-2h}}{2p^2}+\frac{\omega_\ell^{3h}+\omega_\ell^{-3h}}{3p^3}-\cdots \right)\\
 =&\frac{\cos\left(\frac{2\pi h}{\ell}\right)(e^{-py}-1)}{p}-\frac{1}{2}\log\left(1-\frac{2\cos\left(\frac{2\pi h}{\ell}\right)}{p}+\frac{1}{p^2}\right).
 \end{align*}
 When $\cos\left(\frac{2\pi h}{\ell}\right)<0$, the first term above is positive, and we can bound $A_p$ below by the second term involving the $\log$. 
 Otherwise, notice that $\frac{e^{-py}-1}{p}\geq -y=-(\log q)^{-2+7\epsilon}$. Summing over $p$, we obtain
 \[
 -\log |L(1, \chi)|\leq \frac12 \sum_{p\leq x}\log\left(1-\frac{2\cos\left(\frac{2\pi h}{\ell}\right)}{p}+\frac{1}{p^2}\right) + O\left((\log q)^{-2+7\epsilon}\right).
 \]
We remark that, since $\ell$ is odd, the real part of  $\omega_\ell^h$ is minimized when $h=\frac{\ell-1}{2}$ or $h=\frac{\ell+1}{2}$. In other words, 
$-\cos\left(\frac{2\pi h}{\ell}\right)\leq -\cos\left(\frac{\pi (\ell-1)}{\ell}\right)=\cos\left(\frac{\pi}{\ell}\right)$. Thus we have 
 \begin{align*}
 &\frac12 \sum_{p\leq x} \log \left(1-\frac{2\cos\left(\frac{2\pi h}{\ell}\right)}{p}+\frac{1}{p^2}\right)
 \leq \frac12 \log \prod_{p\leq x}\left(1+\frac{2\cos\left(\frac{\pi}{\ell}\right)}{p}+\frac{1}{p^2}\right)\\
& =\frac{1}{2}\log  \prod_{p\leq x} \left(1-\frac{1}{p}\right)^{-2\cos\left(\frac{\pi}{\ell}\right)}  +\frac12 \log \prod_{p\leq x}\left[\left(1+\frac{2\cos\left(\frac{\pi}{\ell}\right)}{p}+\frac{1}{p^2}\right) \left(1-\frac{1}{p}\right)^{2\cos\left(\frac{\pi}{\ell}\right)} \right]\\
&=\cos\left(\frac{\pi }{\ell}\right) \log \left(e^{\gamma} \log x\right) +\log \left(1+O\left(\frac{\log x}{ \sqrt{x}}\right)\right)-\log \left(C_\ell\right),
 \end{align*}
 where  
 \[C_\ell=\prod_{p}\left[ \left(1-\frac{1}{p}\right)^{-\cos\left(\frac{\pi }{\ell}\right)}\left(1+\frac{2\cos\left(\frac{\pi}{\ell}\right)}{p}+\frac{1}{p^2}\right)^{-1/2}\right],\]
 and we have applied Mertens' third theorem (with the error term coming from assuming RH). Taking the exponential and recalling our choice of $x=(\log q)^{2+5\epsilon}$, we conclude that
 \[ |L(1, \chi)|\geq  \frac{C_\ell}{\left(2e^{\gamma}\log \log q\right)^{\cos\left(\frac{\pi}{\ell}\right)}}\left(1+O\left(\frac{\log_2 q}{\log q}\right)\right). \]
 
 \end{proof}

To prove the $\Omega$-results, we need the following auxiliary statements. 
\begin{prop}\label{main prop for omega result}
	Assume GRH for Hecke $L$-functions over $\Q(\omega_3)$.  Let $P(z) = \prod_{\substack{p \leq z}} p^2 = e^{2z + o(z)}.$ Let $\epsilon_p \in \{ 1, \omega_3, \omega_3^2 \}$ for each rational prime $p \leq z$, and let $$\mathcal{P}(X, \{ \epsilon_p \},z)=\{\chi \in \mathcal{F}_3(X) \,|\, \mathrm{cond}(\chi) \text{ is prime and } \chi(p)=\epsilon_p  \text{ for each } p \leq z\}.$$
	
	Then, for $z \ll X^\frac{1}{6}$,
	\begin{align*}
	\sum_{\chi \in \mathcal{P}(X, \{ \epsilon_p \},z)} \log \mathrm{cond}(\chi) = \frac{X}{ 3^{\pi(z)}} + O(X^{\frac{1}{2}} \log^2 (XP(z)),
	\end{align*}
	and 
	\begin{align*}
	\sum_{\chi \in \mathcal{P}(X, \{ \epsilon_p \},z)} L(1, \chi) \log  \mathrm{cond}(\chi) = \frac{X}{3^{\pi(z)}} \zeta(3) \prod_{\substack{p\leq z}}\left(1+\frac{\epsilon_p^{-2}}{p}+\frac{\epsilon_p^{-1}}{p^2}\right)+O(X^{\frac{1}{2}}\log^3 (XP(z))).
	\end{align*}	
	
	\end{prop}

	\begin{proof}
	By Lemma \ref{BY-1to1-cubic-chars}, studying $\mathcal{P}(X, \{ \epsilon_p \},z)$ is equivalent to studying $\mathcal{P}_{\Z[\omega_3]}(X, \{\epsilon_p\},z)$ where 
	\begin{align*}
	& \mathcal{P}_{\Z[\omega_3]}(X, \{\epsilon_p\},z)\\
	&=\left\{\mathfrak{q}\in \Z[\omega_3]\setminus \Z :  \mathfrak{q}\equiv 1\mod{3}, \mathfrak{q} \text{ is prime, } N(\mathfrak{q})\le X \text{ and }\left( \frac{p}{\mathfrak{q}} \right)_3=\epsilon_p  \text{ for each } p \leq z\right\}.
	\end{align*}
	
Then we have 	
\begin{align}\label{eq:gettingridofq=2}
&\#\mathcal{P}_{\Z[\omega_3]}(X, \{\epsilon_p\},z)	\nonumber \\&=\# \left\{\mathfrak{q}\in \Z[\omega_3] :  \mathfrak{q}\equiv 1\mod{3}, \mathfrak{q} \text{ is prime, } N(\mathfrak{q})\le X \text{ and }\left( \frac{p}{\mathfrak{q}} \right)_3=\epsilon_p  \text{ for each } p \leq z\right\} \nonumber\\
	&+ O \left( \# \left\{q\in \Z_{>0} :  q\equiv 2 \mod{3}, q \text{ is prime, } N(q)\le X \right\} \right)\nonumber \\
	&=\#\left\{\mathfrak{q}\in \Z[\omega_3] :  \mathfrak{q}\equiv 1\mod{3}, \mathfrak{q} \text{ is prime, } N(\mathfrak{q})\le X \text{ and }\left( \frac{p}{\mathfrak{q}} \right)_3=\epsilon_p  \text{ for each } p \leq z\right\} \nonumber \\& + O \left( X^{\frac{1}{2}} \right).
	 \end{align}

	We first write a characteristic function for $\mathcal{P}_{\Z[\omega_3]}(X, \{ \epsilon_p \},z)$.  For a rational $m \mid P(z)$, we define $\epsilon_m = \prod_{p \mid m} \epsilon_p^{v_p(m)}.$ Then,
	$$
	\frac{1}{3^{\pi(z)}} \sum_{m \mid P(z)} \epsilon_m^{-1}\left( \frac{m}{\mathfrak{q}} \right)_3  = \begin{cases} 1 & \text{if $\mathfrak{q} \in \mathcal{P}_{\Z[\omega_3]}(X, \{ \epsilon_p \},z)$}, \\ 0 & \text{otherwise.} \end{cases}
	$$

	We write 
	\begin{align*} 
	\sum_{\chi \in \mathcal{P}(X, \{ \epsilon_p \},z)} \log \mathrm{cond}(\chi) =&\sum_{\mathfrak{q} \in \mathcal{P}_{\Z[\omega_3]}(X, \{ \epsilon_p \},z)} \log N(\mathfrak{q})\nonumber = \frac{1}{3^{\pi(z)}}      
	\sum_{m \mid P(z)} \epsilon_m^{-1}
	\sum_{\substack{\mathfrak{q}\in\Z[\omega_3] \setminus \Z\\ N(\mathfrak{q}) \leq X\\ \mathfrak{q} \equiv 1 \mod{3}\\ \mathfrak{q} \text{ is prime}}} \left( \frac{m}{\mathfrak{q}} \right)_3  \log{N(\mathfrak{q})} \nonumber\\
	=&\frac{1}{3^{\pi(z)}}      
	\sum_{m \mid P(z)} \epsilon_m^{-1}
	\sum_{\substack{\mathfrak{q}\in\Z[\omega_3]  \\ N(\mathfrak{q}) \leq X\\ \mathfrak{q} \equiv 1 \mod{3}\\ \mathfrak{q} \text{ is prime}}} \left( \frac{m}{\mathfrak{q}} \right)_3  \log{N(\mathfrak{q})} +O(X^{\frac{1}{2}}),
	\end{align*}
	by \eqref{eq:gettingridofq=2}. The main term comes from $m=1$ and, by the Prime Number Theorem, it equals $\frac{X}{ 3^{\pi(z)}}$ with an error term of $O(X^{\frac{1}{2}}\log X)$ if we assume GRH. When $m \neq 1$, then $\left( \frac{m}{\cdot } \right)_3$ is a non-trivial character modulo at most  $9m$ since $m$ is not a cube, and using the Chebotarev density theorem under GRH, we have that the sum over $\mathfrak{q}$ is $O(X^{\frac{1}{2}} \log^2(XP(z)^2))$. This gives the first statement.   

For the second result, first note that  
\begin{equation}\label{L-func}
	L(1, \chi) = \sum_{n \leq N} \frac{\chi(n)}{n} +  \sum_{n > N} \frac{\chi(n)}{n} =  \sum_{n \leq N} \frac{\chi(n)}{n} 
	+ O \left( \frac{\mathrm{cond}(\chi)}{N} \right),
\end{equation}
where 
\begin{align} \label{choice-of-N} N:= X^{3/2}P(z)^{3/2} .\end{align}
	Thus we get 
	\begin{align}
	\sum_{\chi \in \mathcal{P}(X, \{ \epsilon_p \},z)} L(1, \chi) \log \mathrm{cond}(\chi) &= \sum_{\chi \in \mathcal{P}(X, \{ \epsilon_p \},z)}\log \mathrm{cond}(\chi) \left(\sum_{n \leq N} \frac{\chi(n)}{n} 
	+ O \left( \frac{\mathrm{cond}(\chi)}{N} \right)\right)\nonumber\\
	& =\frac{1}{3^{\pi(z)}}      
	\sum_{m \mid P(z)} \epsilon_m^{-1}\sum_{n\leq N}\frac{1}{n}
	\sum_{\substack{\mathfrak{q}\in\Z[\omega_3] \\ N(\mathfrak{q}) \leq X\\ \mathfrak{q} \equiv 1 \mod{3}\\ \mathfrak{q} \text{ is prime}}} \left( \frac{n m}{\mathfrak{q}} \right)_3  \log{N(\mathfrak{q})}+O\left(\frac{X^2 \log X}{N}\right). \label{eq:boundL}
	\end{align}
	The character $\left( \frac{nm}{\mathfrak{q}} \right)_3$ is principal only when $nm=\cube$. 
	In this case, applying the Prime Number Theorem (assuming RH over $\Z[\omega_3]$), we get that
\begin{align*}
\sum_{\substack{\mathfrak{q} \in \Z[\omega_3]\\N(\mathfrak{q}) \leq X\\ \mathfrak{q}\equiv 1 \mod{3}\\ \mathfrak{q} \text{ is prime}}} \left( \frac{n m }{ \mathfrak{q} } \right)_3\log N(\mathfrak{q}) &=
\sum_{\substack{\mathfrak{q} \in \Z[\omega_3]\\N(\mathfrak{q}) \leq X\\ \mathfrak{q}\equiv 1 \mod{3}\\ \mathfrak{q} \text{ is prime}}} \log N(\mathfrak{q})  + O \Bigg( \sum_{\substack{\mathfrak{q} \mid n m \\ \mathfrak{q}\equiv 1 \mod{3}\\ \mathfrak{q} \text{ is prime}}} 1 \Bigg)\\
&=X + O \left( X^{\frac{1}{2}} \log^2{X} + \log{ ( N P(z))} \right) \\
&= X + O \left( X^{\frac{1}{2}} \log^2(XP(z)) \right).
\end{align*}

Replacing the above in  \eqref{eq:boundL}, we have 
\begin{align} \label{PNT-ET}  \left( X + O \left( X^{\frac{1}{2}} \log^2(XP(z)) \right) \right)
\frac{1}{3^{\pi(z)}}      
	\sum_{m \mid P(z)} \epsilon_m^{-1}
	\sum_{\substack{n \leq N\\ n m = \tinycube }}\frac{1}{n}.\end{align}

	We write $n=n_1 n_2^2 n_3^3$, where $n_1, n_2$ are square-free and coprime, and $m=m_1 m_2^2$, where $m_1, m_2$ are square-free and coprime.
Then, $m_1, m_2\mid \widehat{P}(z)$, where $\widehat{P}(z):=\prod_{\substack{p\leq z}}p$.
The only possibility leading to  $nm = \cube$ is to have $n_1=m_2, n_2=m_1$, and we get
\begin{align*}
\frac{1}{3^{\pi(z)}}      
	\sum_{m \mid P(z)} \epsilon_m^{-1}
	\sum_{\substack{n \leq N\\ n m = \tinycube }}\frac{1}{n}
&=\frac{1}{3^{\pi(z)}}  \sum_{n_3 \leq N^{\frac{1}{3}}} \frac{1}{n_3^3} \sum_{\substack{n_2 \leq (N/n_3^3)^{\frac{1}{2}} \\ n_2 \mid \widehat{P}(z)}} \frac{\epsilon_{n_2}^{-1}}{n_2^2} 
 \sum_{\substack{n_1 \leq N/n_2^2 n_3^3 \\ (n_2, n_1) = 1 \\ n_1 \mid \widehat{P}(z)}} \frac{\epsilon_{n_1}^{-2}}{n_1}.  
 \end{align*}
 The condition $\widehat{P}(z) \leq N/n_2^3 n_3^3$ is equivalent to $n_2^2 n_3^3 \leq N/\widehat{P}(z) = X^{3/2} P(z)$ with the choice of $N$ given by \eqref{choice-of-N}. In this case, we have no condition on the size of $n_1$ in the inside sum, and we compute the above sum as
\begin{align*}
&= \frac{1}{3^{\pi(z)}} \prod_{p \leq z} \left( 1 + \frac{\epsilon_p^{-2}}{p} \right) \sum_{n_3 \leq N^{\frac{1}{3}}} \frac{1}{n_3^3} \sum_{\substack{n_2 \leq {(N/n_3^3)^{\frac{1}{2}}} \\ n_2 \mid \widehat{P}(z)}} \frac{\epsilon_{n_2}^{-1} }{n_2^2} \prod_{\substack{p \leq z \\ p \mid n_2}}
 \left( 1 + \frac{\epsilon_p^{-2}}{p} \right)^{-1}.
\end{align*}
When  $n_2^2 n_3^3 \geq N/\widehat{P}(z) = X^{3/2} P(z)$, we get the bound
\begin{align*}
&\ll   \frac{\log z}{3^{\pi(z)}} \sum_{n_3 \leq N^{\frac{1}{3}}} \frac{1}{n_3^3} \sum_{\substack{\left( \frac{X^{3/2} P(z)}{n_3^3} \right)^{\frac{1}{2}} \leq n_2 \leq {(N/n_3^3)^{\frac{1}{2}}} }}\frac{1}{n_2^2} \\
& \ll  \frac{\log z}{3^{\pi(z)}} \sum_{n_3 \leq N^{\frac{1}{3}}} \frac{1}{n_3^3} \left( \frac{n_3^{3/2}}{X^{3/4} P(z)^{\frac{1}{2}}} \right) \ll \frac{\log{z} \,X^{-\frac{3}{4}}}{3^{\pi(z)} P(z)^{\frac{1}{2}}},
 \end{align*}
which is $O(X^{-\frac{1}{2}})$. Working similarly for the $n_2$-sum, separating the range for $n_3^{3/2} \leq N/\widehat{P}(z) = X^{\frac{3}{2}} P(z)$, we 
get
\begin{align*}
\frac{1}{3^{\pi(z)}}      
	\sum_{m \mid P(z)} \epsilon_m^{-1}
	\sum_{\substack{n \leq N\\ n m = \tinycube }}\frac{1}{n} &=
 \frac{1}{3^{\pi(z)}} \prod_{p \leq z} \left( 1 + \frac{\epsilon_p^{-2}}{p} \right) \sum_{n_3 \leq N^{\frac{1}{3}}} \frac{1}{n_3^3} \sum_{\substack{n_2 \mid \widehat{P}(z)}} \frac{\epsilon_{n_2}^{-1} }{n_2^2} \prod_{\substack{p \leq z \\ p \mid n_2}}
 \left( 1 + \frac{\epsilon_p^{-2}}{p} \right)^{-1}  + O \left( X^{-1/2} \right) \\
 &= \frac{1}{3^{\pi(z)}} \prod_{p \leq z} \left(1+\frac{\epsilon_p^{-2}}{p} +\frac{\epsilon_p^{-1}}{p^2}\right) \left( \zeta(3) + O\left(\frac{1}{N^{\frac{2}{3}}} \right) \right)+ O \left(X^{-1/2} \right) \\
 &= \frac{\zeta(3)}{3^{\pi(z)}} \prod_{p \leq z} \left(1+\frac{\epsilon_p^{-2}}{p}+\frac{\epsilon_p^{-1}}{p^2}\right) + O \left(X^{-1/2} \right). 
 \end{align*}
Replacing in \eqref{PNT-ET}, we have
\begin{align*}
 \left( X + O \left( X^{\frac{1}{2}} \log^2(XP(z)) \right) \right)
\frac{1}{3^{\pi(z)}}      
	\sum_{m \mid P(z)} \epsilon_m^{-1}
	\sum_{\substack{n \leq N\\ n m = \tinycube }}\frac{1}{n} =& \frac{X}{3^{\pi(z)}} \zeta(3) \prod_{p \leq z} \left(1+\frac{\epsilon_p^{-2}}{p}+\frac{\epsilon_p^{-1}}{p^2}\right) \\& + O \left( X^{\frac{1}{2}} \log^2(XP(z)) \right).
\end{align*}

When $nm \neq \cube$ then $\left(\frac{n m}{\cdot} \right)_3$  is a non-trivial character of modulo at most $9n m$  and using again the GRH, we have that the sum over $\mathfrak{q}$ in \eqref{eq:boundL} is $O(X^{\frac{1}{2}} \log^2 (XN^2P(z)^2)).$ 
Therefore the non-cubic contribution can be bounded by
\[
\ll X^{\frac{1}{2}} \log^2 (XNP(z)) \frac{1}{3^{\pi(z)}}      
\sum_{m \mid P(z)} \epsilon_m^{-1}\sum_{n\leq N}\frac{1}{n}\ll X^{\frac{1}{2}} \log^2 (XNP(z)) \log N \ll X^{\frac{1}{2}} \log^3 (XP(z)).
\]
 This concludes the proof.
	\end{proof}

Our next goal is to prove a result  analogous to Proposition \ref{main prop for omega result} that can be applied to obtain a lower bound. Since in this case we are interested in  characters of order $\ell$ with $\ell$ prime, we need to consider this more general setting. First, we need to generalize the ideas exposed in Lemma \ref{lemma-BY}.

The characters of order $\ell$ are supported on primes $q \equiv 1 \mod{\ell}$, which are those that split completely in $\Q(\omega_\ell)$, say $q\Z[\omega_\ell]=\mathfrak{q}_1\cdots  \mathfrak{q}_{\ell-1}$. For $n \in \Z$, we have that
\[\left \{ \chi\, :\, \text{ $\chi$ is a primitive character of conductor $q$}\right\}=\left \{\left( \frac{\cdot}{\mathfrak{q}_1}\right)_\ell, \dots, \left( \frac{\cdot}{\mathfrak{q}_{\ell-1}}\right)_\ell\right\},\]
where the $\ell$th residue symbol is defined for $q\nmid n$ by 
\[\left( \frac{n}{\mathfrak{q}}\right)_\ell=\omega_\ell^k, \qquad \mbox{ where }\qquad
n^{\frac{q-1}{\ell}}\equiv \omega_\ell^k \mod{\mathfrak{q}}.\]

\begin{prop}\label{main prop for small values}
Assume GRH for Hecke $L$-functions over $\Q(\omega_{\ell})$.  Let $P_\ell(z)=\prod_{p\leq z} p^{\ell-1}=e^{(\ell-1)z+o(z)}$. Let $\epsilon_p \in\{\omega_\ell^k\, :\, k=1,\dots,\ell\}$ for each rational prime $p\leq z$, and let
	$$\mathcal{P}_\ell(X, \{ \epsilon_p \},z)=\{\chi \in \mathcal{F}_\ell(X) \,|\, \mathrm{cond}(\chi) \text{ is prime and } \chi(p)=\epsilon_p  \text{ for each } p \leq z\}.$$
 Then for $z\ll X^{\frac{1}{6}}$, we have
	\begin{align*}
	\sum_{\chi \in \mathcal{P}(X, \{ \epsilon_p \},z)} \frac{1}{L(1, \chi)} \log  \mathrm{cond}(\chi) = \frac{X}{\ell^{\pi(z)}} \prod_{\substack{p\leq z}}\left(1-\frac{\epsilon_p^{-(\ell-1)}}{p}\right)+O(X^{\frac{1}{2}}\log^3 (XP_\ell(z))).
	\end{align*} 
	\end{prop}

\begin{proof}

	As an application of Perron's formula, under GRH, it is known that for any $\epsilon>0$,
	 \[
	 \sum_{n\leq x}\mu(n)\chi(n)\ll x^{\frac{1}{2}+\epsilon}.
	 \]
	 Using this, for large $N$, we have
	 \begin{align*}
	 \frac{1}{L(1, \chi)}&=\sum_{n\leq N}\frac{\mu(n)\chi(n)}{n}+\int_{N}^{\infty}\left(\sum_{n\leq t}\mu(n)\chi(n)\right)\frac{dt}{t^2}\\
	 & =\sum_{n\leq N}\frac{\mu(n)\chi(n)}{n}+O\left(N^{-\frac{1}{2}+\epsilon}\right).
	 \end{align*}

By the discussion before the statement of Proposition \ref{main prop for small values},  studying  $\mathcal{P}_\ell(X, \{ \epsilon_p \},z)$ is equivalent to studying $\mathcal{P}_{\Z[\omega_\ell]}(X, \{\epsilon_p\},z)$ where 
	\begin{align}\label{eq:Pz2}& \mathcal{P}_{\Z[\omega_\ell]}(X, \{\epsilon_p\},z)=\left\{\mathfrak{q} \text{ ideal in } \Z[\omega_\ell] : N(\mathfrak{q})=q \text{ is prime, } q\le X \text{ and }\left( \frac{p}{\mathfrak{q}} \right)_\ell=\epsilon_p  \text{ for each } p \leq z\right\}.
	\end{align}
	
Then we have 	
\begin{align*}
\#\mathcal{P}_{\Z[\omega_\ell]}(X, \{\epsilon_p\},z)	&=\# \left\{\mathfrak{q} \text{ ideal in } \Z[\omega_\ell] : N(\mathfrak{q})\le X \text{ and }\left( \frac{p}{\mathfrak{q}} \right)_\ell=\epsilon_p  \text{ for each } p \leq z\right\} \\& + O \left( X^{\frac{1}{2}} \right).
	 \end{align*}
We will use the detector
	 $$
	\frac{1}{\ell^{\pi(z)}} \sum_{\substack{m \mid P_\ell(z)}} \epsilon_m^{-1}\left( \frac{m}{\mathfrak{q}}\right)_\ell= \begin{cases} 1 & \text{if $\mathfrak{q} \in \mathcal{P}_{\Z[\omega_\ell]}(X, \{ \epsilon_p \},z)$}, \\ 0 & \text{otherwise.} \end{cases}
	$$
	 
	 We apply \eqref{eq:Pz2} and $\frac{1}{|L(1, \chi)|}\ll \log_2 \mathrm{cond}(\chi)$ under GRH to get 
	 \begin{align} 
	 &\sum_{\chi \in \mathcal{P}_\ell(X, \{ \epsilon_p \},z)} \frac{1}{L(1, \chi)} \log \mathrm{cond}(\chi) = \sum_{\chi \in \mathcal{P}_\ell(X, \{ \epsilon_p \},z)}\log \mathrm{cond}(\chi) \left(\sum_{n \leq N} \frac{\mu(n)\chi(n)}{n} 
	 + O \left( N^{-\frac{1}{2}+\epsilon}\right)\right)\nonumber\\
	 & =\frac{1}{\ell^{\pi(z)}}      \sum_{m \mid P_\ell(z)} \epsilon_m^{-1}\sum_{n\leq N}\frac{\mu(n)}{n}
	 \sum_{\substack{\mathfrak{q} \text{ ideal in }\Z[\omega_\ell] \\\mathfrak{q} \text{ prime }\\  N(\mathfrak{q})  \leq X}} \left( \frac{n m}{\mathfrak{q}} \right)_\ell \log{N(\mathfrak{q})}+O\left(\frac{X \log X}{N^{\frac{1}{2}-\epsilon}}\right)+O\left(X^{\frac{1}{2}}\log X \log_2 X\right). \label{eq:boundLmin} 
	 \end{align}
 The map  $\mathfrak{q} \rightarrow \left( \frac{nm}{\mathfrak{q}} \right)_\ell$ is a Hecke character over 	 $\Q(\omega_\ell)$ that becomes principal when $nm$ in an $\ell$th power. In this case, applying the Prime Number Theorem and assuming RH over $\Z[\omega_\ell]$, we get
	 \begin{align*}
	 \sum_{\substack{\mathfrak{q} \text{ ideal in }\Z[\omega_\ell] \\ \mathfrak{q} \text{ prime }\\  N(\mathfrak{q})  \leq X}}\left( \frac{n m }{ \mathfrak{q} } \right)_\ell\log N(\mathfrak{q}) &=\sum_{\substack{\mathfrak{q} \text{ ideal in }\Z[\omega_\ell] \\ \mathfrak{q} \text{ prime }\\  N(\mathfrak{q})  \leq X}}
	 \log N(\mathfrak{q})  + O \Bigg( 
	  \sum_{\substack{\mathfrak{q} \text{ ideal in }\Z[\omega_\ell] \\ N(\mathfrak{q}) \text{ prime }\\  N(\mathfrak{q})  \leq X\\ \mathfrak{q}\mid nm }}1 \Bigg)\\
	 &=X + O_\ell\left( X^{\frac{1}{2}} \log^2{X} + \log{ ( N P_\ell(z))} \right) \\
	 &= X + O_\ell \left( X^{\frac{1}{2}} \log^2(XP_\ell(z)) \right),
	 \end{align*}where we take $N=XP_\ell(z)$ and we use the assumption that $z\ll X^{\frac{1}{6}}$. 
	 Replacing in  \eqref{eq:boundLmin}, we have 
	 \begin{align} \label{PNT-ETmin}  \left( X + O_\ell \left( X^{\frac{1}{2}} \log^2(XP_\ell(z)) \right) \right)
	 \frac{1}{\ell^{\pi(z)}}      
	 \sum_{m \mid P_\ell(z)} \epsilon_m^{-1}
	 \sum_{\substack{n \leq N\\ n m = \ell\text{th power}  }}\frac{\mu(n)}{n}.\end{align}
	 We write $m=m_1 m_2^2\cdots m_{\ell-1}^{\ell-1}$, where the  $m_1, \dots, m_{\ell-1}$ are square-free and pairwise coprime and $m_1, \dots, m_{\ell-1}\mid \widehat{P}(z)$, where we recall that  $\widehat{P}(z)=\prod_{\substack{p\leq z}}p$.
	 Since $n$ is square-free, then the only possibility for $nm$ to be an $\ell$th power is to have $m_{\ell-1}=n$ and $m_k=1$ for $k=1,\dots,\ell-2$. Therefore \eqref{PNT-ETmin} is equal to 
	 \begin{align*}
	 &\left( X + O_\ell \left( X^{\frac{1}{2}} \log^2(XP_\ell(z)) \right) \right)\frac{1}{\ell^{\pi(z)}}\sum_{\substack{n\leq N\\ n\mid \widehat{P}(z)}}\frac{\mu(n)\epsilon_n^{-(\ell-1)}}{n}\\
	 &=\frac{\left( X + O_\ell \left( X^{\frac{1}{2}} \log^2(XP_\ell(z)) \right)\right)}{\ell^{\pi(z)}}\prod_{p\leq z}\left(1-\frac{\epsilon_p^{-(\ell-1)}}{p}\right),
	 \end{align*}
	 where the last equality follows from the choice of $N$.

	 The map  $\mathfrak{q} \rightarrow \left( \frac{nm}{\mathfrak{q}} \right)_\ell$ is a non-trivial Hecke character over $\Q(\omega_\ell)$ of conductor a multiple of $(nm)$ when 
	 $nm$ is not an $\ell$th power in $\Z$, since $n$ is square-free and $m$ is $\ell$-power free, which follows from the fact that $m \mid P_\ell(z)$.  Using the GRH again, we have that the sum over $\mathfrak{q}$ in \eqref{eq:boundLmin} is $O_\ell (X^{\frac{1}{2}}\log^2(XNP_\ell(z)^2))$. Therefore, the terms when  $nm$ is not an $\ell$th power in $\Z$ contribute 
	 \[
\ll_\ell X^{\frac{1}{2}} \log^2 (XNP_\ell(z)) \frac{1}{\ell^{\pi(z)}}      
\sum_{m \mid P_\ell(z)} \epsilon_m^{-1}\sum_{n\leq N}\frac{1}{n}\ll_\ell X^{\frac{1}{2}} \log^2 (XNP_\ell(z)) \log N \ll X^{\frac{1}{2}} \log^3 (XP_\ell(z)).
\]
    This finishes the proof.
	\end{proof}
	
\begin{proof}[Proof of Eq.~\eqref{eq:maximum value under GRH} in Theorem \ref{maxmin value under GRH}]
	Let us consider the case $\epsilon_p=1$ for all primes $p\leq z$.
	 If $3^{\pi(z)}\leq X^{\frac{1}{2}-\epsilon}$ then using Proposition \ref{main prop for omega result}, we have
	 	\begin{align*}
	 \sum_{\chi \in \mathcal{P}(X, \{ 1 \},z)} \log \mathrm{cond}(\chi) = \frac{X}{ 3^{\pi(z)}}\left(1+O(X^{-\epsilon})\right),
	 \end{align*}
	 and
	 \begin{align*}
	 \sum_{\chi \in \mathcal{P}(X, \{ 1 \},z)} L(1, \chi) \log  \mathrm{cond}(\chi) = \frac{X}{3^{\pi(z)}} \zeta(3) \prod_{\substack{p\leq z}}\left(1+\frac{1}{p}+\frac{1}{p^2}\right)\left(1+O( X^{-\epsilon})\right).
	 \end{align*}
	 Applying the triangular inequality, we have
	 \begin{align*}
	 \sum_{\chi \in \mathcal{P}(X, \{ 1 \},z)} |L(1, \chi)| \log  \mathrm{cond}(\chi)& \geq \bigg|\sum_{\chi \in \mathcal{P}(X, \{ 1 \},z)} L(1, \chi) \log  \mathrm{cond}(\chi)\bigg|\\
	 & = \frac{X}{3^{\pi(z)}} \zeta(3) \prod_{\substack{p\leq z}}\left(1+\frac{1}{p}+\frac{1}{p^2}\right)\left(1+O(X^{-\epsilon})\right).
	 \end{align*}
We take $z:= \frac{\log X \log_2 X}{2\log 3}$ in order to have  $3^{\pi(z)}\leq  X^{\frac{1}{2}-\epsilon}$.   For convenience of notation, we denote 
\[\alpha_{\epsilon_p}:=	\zeta(3) \prod_{\substack{p\leq z}}\left(1+\frac{\epsilon_p}{p}+\frac{\epsilon_p}{p^2}\right).\]
Recall that from \eqref{L-func}, we also have that $|L(1, \chi)|\leq 2 \log  \mathrm{cond}(\chi)$ for $\mathrm{cond}(\chi)$ sufficiently large.
Let $\delta>0$ arbitrary, and assume that the proportion of $\chi \in \mathcal{P}(X, \{ \epsilon_p \},z)$ such that $|L(1, \chi)|\geq  \alpha_{\epsilon_p}-\delta$ is $\Delta_{\geq \alpha-\delta}$. Then we have 
\begin{align*}
\frac{X}{3^{\pi(z)}} \alpha_{\epsilon_p} \leq  & \sum_{\substack{\chi \in \mathcal{P}(X, \{ \epsilon_p \},z)\\|L(1, \chi)|<\alpha_{\epsilon_p}-\delta}} |L(1, \chi)| \log  \mathrm{cond}(\chi) + \sum_{\substack{\chi \in \mathcal{P}(X, \{ \epsilon_p \},z)\\|L(1, \chi)|\geq \alpha_{\epsilon_p}-\delta}} |L(1, \chi)| \log  \mathrm{cond}(\chi) \\
\leq & 
(1-\Delta_{\geq \alpha-\delta}) \frac{X}{3^{\pi(z)}}   (\alpha_{\epsilon_p}-\delta)  +\frac{X}{3^{\pi(z)}} 2 \log^2 X  \Delta_{\geq \alpha-\delta}
\end{align*}
and thus
\[\alpha_{\epsilon_p} \leq(1- \Delta_{\geq \alpha-\delta}) (\alpha_{\epsilon_p}-\delta)  +2 \log^2 X  \Delta_{\geq \alpha-\delta},\]
and therefore 
\[\frac{\delta}{2\log^2X-\alpha_{\epsilon_p}+\delta}\leq \Delta_{\geq \alpha-\delta}.\]
Multiplying $\Delta_{\geq \alpha-\delta}$ by the total number of elements in $\mathcal{P}(X, \{ \epsilon_p \},z)$, which is $X/3^{\pi(z)}\gg X^{\frac{1}{2}+\epsilon}$, we get that  there are $\gg X^{\frac{1}{2}}$  
	 cubic characters $\chi\in \mathcal{F}_3(X)$ with prime conductor such that 
	 \begin{align*}
	 |L(1, \chi)|&\geq \zeta(3) \prod_{\substack{p\leq z}}\left(1+\frac{1}{p}+\frac{1}{p^2}\right)= \zeta(3) \prod_{\substack{p\leq z}} \left(1-\frac{1}{p}\right)^{-1} \prod_{\substack{p\leq z}}\left(1-\frac{1}{p^3}\right)\\
	 &\geq e^{\gamma-O(1/z)} \log z \prod_{\substack{p> z}}\left(1-\frac{1}{p^3}\right)\geq e^{\gamma}\left(\log_2 X+\log_3 X-\log(2\log 3)+o(1)\right).
	 \end{align*}
	 \end{proof}

\begin{proof}[Proof of Theorem \ref{minimum value under GRH-gen}]
We fix the values  $\epsilon_p=\omega_\ell^\frac{\ell-1}{2}$ for all $p\leq z$.  Following the proof of 
Eq.~\eqref{eq:maximum value under GRH} in Theorem \ref{maxmin value under GRH}, if $\ell^{\pi(z)}\leq X^{\frac{1}{2}-\epsilon}$ then applying Proposition \ref{main prop for small values}, we have
\begin{align*}
\sum_{\chi \in \mathcal{P}_\ell(X, \{ \omega_\ell \},z)} 
\frac{1}{L(1, \chi)} \log  \mathrm{cond}(\chi) = \frac{X}{\ell^{\pi(z)}} \prod_{\substack{p\leq z}}\left(1-\frac{\omega_\ell^\frac{\ell-1}{2}}{p}\right)\left(1+O( X^{-\epsilon})\right).
\end{align*}
 Applying the triangular inequality, we have
\begin{align*}
\sum_{\chi \in \mathcal{P}_\ell(X, \{ \omega_\ell \},z)} \frac{1}{|L(1, \chi)|} \log  \mathrm{cond}(\chi)& \geq \bigg|\sum_{\chi \in \mathcal{P}_\ell(X, \{ \omega_\ell \},z)} \frac{1}{L(1, \chi)} \log  \mathrm{cond}(\chi)\bigg|\\
& = \frac{X}{\ell^{\pi(z)}}\prod_{\substack{p\leq z}}\left|\left(1-\frac{\omega_\ell^\frac{\ell-1}{2}}{p}\right)\right|\left(1+O(X^{-\epsilon})\right).
\end{align*}
Following the choice of $z:= \frac{\log X \log_2 X}{2\log \ell}$ and similar arguments as in the proof of 
Eq.~\eqref{eq:maximum value under GRH} in  Theorem \ref{maxmin value under GRH}, we conclude that
there are $\gg X^{\frac{1}{2}}$ $\ell$th  power characters $\chi\in \mathcal{F}_\ell(X)$ with prime conductor such that
\begin{align*}
\frac{1}{|L(1, \chi)|}\geq \prod_{p\leq z}\left|\left(1-\frac{\omega_\ell^\frac{\ell-1}{2}}{p}\right)\right| =&\prod_{p\leq z}\left(1+\frac{2\cos\left(\frac{\pi}{\ell}\right)}{p}+\frac{1}{p^2}\right)^{\frac{1}{2}}
=  \frac{(e^\gamma \log z)^{\cos\left(\frac{\pi}{\ell}\right)}}{C_\ell \left(1+O\left(\frac{\log z}{\sqrt{z}}\right)\right)}\\
=& \frac{(e^\gamma (\log_2X+\log_3X-\log(2\log \ell)))^{\cos\left(\frac{\pi}{\ell}\right)}}{C_\ell \left(1+O\left(\frac{\log_2X}{\sqrt{\log X}}\right)\right)},
\end{align*}
where we have employed the same simplifications as in the proof of Proposition \ref{thm:littlewood-gen}.

\end{proof}

\section{Strengthening the range in the distribution result via the short Euler product under GRH} \label{sec:1.4}
Our goal in this section is to obtain the distribution result  Theorem \ref{main distribution result under GRH}, by strengthening the range of $\tau$ in Theorem \ref{theorem-relating-to-RV} under GRH. We obtain a range similar to the $\Omega$-result under GRH (as the term $\log_2 X+\log_3 X$ is achieved).
We remark that Theorem \ref{Asympformula} plays an integral role in the proof of  Theorem \ref{theorem-relating-to-RV}.  Under GRH, we can extend the range of $|z|$ in Theorem \ref{Asympformula} by approximating $L(1, \chi)$ with a short Euler product, defined as  \[L(s, \chi; y):=\prod_{p\leq y} \left(1-\frac{\chi(p)}{p^s} \right)^{-1}.\]
 We also define the short random Euler product
 $$L(1, \X; y) = \prod_{p \leq y} \left( 1 - \frac{\X(p)}{p} \right)^{-1}.$$
  \begin{thm}\label{asymp moment under GRH} Assume GRH for $L(s,\chi)$ over $\Q(\omega_3)$.
Let $C>0$, $e^{40} \leq B \leq (\log_2 X)^{C}$, $z$ be a real number and $y:= B\log X \log_2 X$. Then, uniformly in the region $|z|\leq \frac{\log X \log_2 X}{e^{37} \log B}$, we have when $X$ goes to infinity,
\begin{align*}
\frac{1}{ |\mathcal{F}_3(X)|}\sum_{\chi \in \mathcal{F}_3(X)} |L(1, \chi; y)|^{2z} = \mathbb{E}(|L(1,\mathbb{X};y)|^{2z})+O\left( X^{-\frac{23}{100}} \mathbb{E}(|L(1,\mathbb{X};y)|^{2z})\right).
\end{align*} 
 	\end{thm} 
	Results of the same quality were obtained for quadratic characters in \cite{GS}.

We start with the some auxiliary results. 
The following lemma relies on upper bounds for $|\log L(1, \chi_d)|$ proven in \cite{GS-JAMS} 
for quadratic characters, and extended by Lamzouri \cite{Lam2010}  to general $L$-functions. 
For the case of quadratic characters, this is analogous to \cite[Lemma 4.5]{GS} replacing $\sigma =\frac{3}{4}$ with $\sigma=\frac{2}{3}$.

\begin{lem}\label{character sum under GRH} Assume GRH for Hecke $L$-functions over $\Q(\omega_3)$.
Let $m\in \mathbb{Z}$ that is not a cube. Then  we have
\[	\sum_{\chi\in \mathcal{F}_3(X)}\chi(m)\ll X^{\frac23 +\epsilon}\exp\left((\log m)^{\frac23-\epsilon}\right).\]
	\end{lem}
\begin{proof}
	Recall from \eqref{eq:sieve} that
	\begin{align}\label{15/09/2022}
		&\sum_{\chi \in \cF_3(X)} \chi(m)  = \sum_{\substack{d \in \Z, d  \equiv 1 \mod 3 \\ |d| \leq \sqrt{X}}} \mu_\Z(d)  \chi_d(m) \sum_{\substack{\ell \in \Z[\omega_3], \ell \equiv 1 \mod 3\\(\ell,d)=1\\ N(\ell) \leq \sqrt{X/N(d)}}}  \mu_{\Z[\omega_3]} (\ell) \chi_{\ell^2}(m)
		\sum_{\substack{ n \in \Z[\omega_3], n \equiv 1 \mod 3 \\ (n,d)=1 \\ N(n) \leq X/N(d \ell^2)}} \chi_n(m).
	\end{align}
The function $\psi_m: (n) \mapsto \chi_n(m) = \left( \frac{m}{n} \right)_3$ defined on ideals $(n) \subset \Z[\omega_3]$ (coprime to 3, and where $m \equiv 1 \mod 3$) gives a Hecke character of modulus $9m$. Let
\begin{align*}
L(s, \psi_m) = \sum_{(n)} \psi_m((n)) N(n)^{-s} = \sum_{\substack{n \in \Z[\omega_3]\\n \equiv 1 \mod 3}} \chi_n(m) N(n)^{-s}
\end{align*}
be the corresponding Hecke $L$-function.

Using Perron's summation formula  we have, for some $T\geq 1$,
\begin{align}\label{eq:Perron}
\sum_{\substack{ n \in \Z[\omega_3], n \equiv 1 \mod 3 \\ (n,d)=1 \\ N(n) \leq Z}} \chi_n(m)=\frac{1}{2\pi i}\int_{(1+\epsilon)}f_{m, d}(s)Z^s\frac{ds}{s}=\frac{1}{2\pi i}\int_{1+\epsilon-iT}^{1+\epsilon+iT}f_{m, d}(s)Z^s\frac{ds}{s}+O\left(\frac{Z^{1+\epsilon}}{T}\right),
\end{align}
where $Z:=\frac{X}{N(d\ell^2)}$ and \[
f_{m, d}(s)=\sum_{\substack{n \in \Z[\omega_3], n \equiv 1 \mod 3 \\ (n,d)=1}}\frac{\chi_m(n)}{N(n)^s}=\prod_{\mathfrak{p}\mid d}\left(1-\frac{\psi_m(\mathfrak{p})}{N(\mathfrak{p})^s}\right)L(s,\psi_m).
\]
A special case of Lemmas 4.2 and 4.3 in \cite{Lam2010}  implies that if $\sigma_1=\min\left(\sigma_0+\frac{1}{\log y}, \frac{\sigma+\sigma_0}{2}\right)$,
where $y \geq 2$, $|t| \geq y+3$ and $\sigma>\sigma_0\geq \frac12$, then
\[
\log L(\sigma+it, \psi_m)=\sum_{N(n)\leq y}\frac{b_n}{N(n)^{\sigma+it}}+O\left(\frac{\log |t|}{(\sigma_1-\sigma_0)^2}y^{\sigma_1-\sigma}\right),
\]
where $$b_n = \begin{cases} \frac{\psi_m((n))}{k} & \mbox{if }(n) = \fp^k, \\ 0 & \text{otherwise.} \end{cases}$$
Then, for $|t| \leq T$ and $y=(\log T)^{\beta}$, we have
\begin{align*}
|\log L(\sigma+it, \psi_m)|
&\ll y^{1-\sigma}+\frac{\log T}{(\sigma_1-\sigma_0)^2}y^{\sigma_1-\sigma}\\
&\ll (\log T)^{\beta(1-\sigma)}+ (\log T)^{1+\beta(\sigma_0-\sigma)}\log_2 T.
\end{align*}
Inserting this estimate after shifting the line of integration in the above Perron's integral \eqref{eq:Perron} to $\re(s)=\sigma>\sigma_0\geq \frac{1}{2}$, we obtain 
\[
\sum_{\substack{ n \in \Z[\omega_3], n \equiv 1 \mod 3 \\ (n,d)=1 \\ N(n) \leq Z}} \chi_n(m)\ll Z^{\sigma}2^{2\omega(d)}\exp\left((\log T)^{\beta(1-\sigma)}+ (\log T)^{1+\beta(\sigma_0-\sigma)}\log_2 T\right)+\frac{Z^{1+\epsilon}}{T}.
\]
We may suppose that $Z\leq N(m)$, since by periodicity $\sum_{N(n)\leq Z}\chi_m(n)\ll N(m)$. This implies that $m\geq Z^{\frac{1}{2}}$ and hence we can take $T:=m$, since in this case $\frac{Z^{1+\epsilon}}{T}\ll Z^{\frac{1}{2}+\epsilon}$, which is permissible. Also equating $\beta(1-\sigma)$ and $1+\beta(\sigma_0-\sigma)$, we can choose $\sigma_0=\frac{\beta-1}{\beta}$ with $\beta\geq 2$. 
Therefore, taking  $\sigma= \frac{2}{3}+\epsilon$ and $\beta=2$, we have
 \[
 \sum_{\substack{ n \in \Z[\omega_3], n \equiv 1 \mod 3 \\ (n,d)=1 \\ N(n) \leq Z}} \chi_n(m)\ll Z^{\frac{2}{3}+\epsilon}2^{2\omega(d)}\exp\left((\log m)^{\frac{2}{3}-\epsilon}\right).
 \]
 Plugging this estimate in \eqref{15/09/2022} together with $Z=\frac{X}{N(d\ell^2)}$, we conclude that
 \begin{align*}
 \sum_{\chi \in \cF_3(X)} \chi(m) \ll& X^{\frac{2}{3}+\epsilon}\sum_{\substack{d \in \Z, d  \equiv 1 \mod 3 \\ d \leq X}} \frac{2^{2\omega(d)}}{d^{\frac{4}{3}}} \sum_{\substack{\ell \in \Z[\omega_3], \ell \equiv 1 \mod 3\\(\ell,d)=1\\ N(\ell) \leq \sqrt{X/N(d)}}} \frac{1}{N(\ell)^{\frac{4}{3}}}\exp\left((\log m)^{\frac{2}{3}-\epsilon}\right)\\
 \ll& X^{\frac{2}{3}+\epsilon}\exp\left((\log m)^{\frac{2}{3}-\epsilon}\right),
 \end{align*}
 which finishes the proof.
 	\end{proof}

The following is the analogue of Theorem 6.1 from \cite{GS} for the case of cubic characters. 
\begin{prop}\label{moment of SEP}  Assume GRH for Hecke $L$-functions over $\Q(\omega_3)$.
	Let $2\leq y \leq \exp(\sqrt{\log X})$ and let $z$ be a real number. Then we have 
	\begin{align*}
	\frac{1}{|\mathcal{F}_3(X)|}\sum_{\chi \in \mathcal{F}_3(X)} |L(1, \chi; y)|^{2z} = \mathbb{E}(|L(1,\mathbb{X};y)|^{2z})\left(1+O\left(E(z,y,X)\right)\right), \end{align*}
	where $E(z,y,X)$ is a positive real number such that 
	for $y\geq 4(|z|+1)$, 
	\begin{align*}
	E(z,y,X)\ll X^{-\frac14}\exp\left(\frac{16(|z|+1)}{\log (4(|z|+1))}+
	\frac{120(|z|+1)}{\sqrt{\log X}}+2\log_2 y+6e^{30} |z|\log\left(\frac{\log y}{\log(4(|z|+1))}\right)\right),
	\end{align*}
and otherwise,
	\[
		E(z,y,X)\ll X^{-\frac14}\exp\left(\frac{34 y}{\log y}\right).
	\]
	
\end{prop}
\begin{proof} We have
	\begin{align} \label{sum-chi}
	\sum_{\chi \in \mathcal{F}_3(X)} |L(1, \chi; y)|^{2z} = \sum_{\substack{n_1, n_2=1\\n_1,n_2\in \mathcal{S}(y)}}^\infty \frac{d_z(n_1)d_z(n_2)}{n_1n_2} S(X; n_1n_2^2),
	\end{align}
	where $$S(X; n) = \sum_{\chi \in \mathcal{F}_3(X)} \chi(n),$$ and $\mathcal{S}(y)$ denotes the set of integers whose prime factors are all $\leq y$.
	We decompose $n_1$ and $n_2$ as  $n_1= r_1 r_2^2 r_3^3 r_4^3$ and $n_2=s_1s_2^2s_3^3s_4^3$, where $r_1, r_2, r_3$ are square-free and coprime in pairs, and $s_1, s_2, s_3$ are as well. In this decomposition we also have $p \mid r_4 \Rightarrow p \mid r_1 r_2 r_3$ and   $p \mid s_4 \Rightarrow p \mid s_1 s_2 s_3$.

	Then, $S(X; n_1 n_2^2) = S(X; r_1 s_1^2 r_2^2 s_2 r_3^3 s_3^3)$ and we rewrite \eqref{sum-chi} as
	\begin{align}\label{eq:sieve2}
	\sum_{r_1 \in \mathcal{S}(y)} \mu^2(r_1) \sum_{\substack{r_2 \in \mathcal{S}(y)\\ (r_1, r_2)=1}} \mu^2(r_2)  \sum_{\substack{r_3 \in \mathcal{S}(y)\\(r_3, r_1 r_2)=1}} \mu^2(r_3)\sum_{s_1 \in \mathcal{S}(y)} \mu^2(s_1) \sum_{\substack{s_2 \in \mathcal{S}(y)\\ (s_1, s_2)=1}} \mu^2(s_2)  \sum_{\substack{s_3 \in \mathcal{S}(y)\\(s_3, s_1 s_2)=1}} \mu^2(s_3) \nonumber\\
	\times
	S(X; r_1 s_1^2 r_2^2 s_2 r_3^3 s_3^3) \sum_{p \mid r_4 \Rightarrow p \mid r_1 r_2 r_3}  \frac{d_z(r_1 r_2^2 r_3^3 r_4^3)}{r_1 r_2^2 r_3^3 r_4^3}\sum_{p \mid s_4 \Rightarrow p \mid s_1 s_2 s_3}  \frac{d_z(s_1 s_2^2 s_3^3 s_4^3)}{s_1 s_2^2 s_3^3 s_4^3}.
	\end{align}
	We first evaluate the contribution 
when $n_1 n_2^2 = \cube \iff r_1s_1^2r_2^2s_2=\cube.$\\
 If we write $(r_1,s_1)=U_1$, $(r_2,s_2)=U_2$, then $r_1=U_1r_{1,1}$, $s_1=U_1s_{1,1}$, $r_2=U_2r_{2,1}$, $s_2=U_2s_{2,1}$. Then we get that the main term comes from $r_{1,1}s_{1,1}^2r_{2,1}^2s_{2,1}=\cube$. Because of coprimality (we have that $(r_{1,1},s_{1,1})=(r_{2,1},s_{2,1})=(r_{1,1},r_{2,1})=(s_{1,1},s_{2,1})=1$) and  square-free conditions ($r_{1,1}, r_{2,1}, s_{1,1}, s_{2,1}$ are square-free) we get that $r_{1,1}=r_{2,1}=s_{1,1}=s_{2,1}=1$. This gives $r_1=s_1=r$ (say) and $r_2=s_2=s$ (say). 
 Using Proposition \ref{prop-gives-RVs}, this results in
	\begin{align} \label{eq-5/8}
	S(X; r_1 s_1^2 r_2^2 s_2 r_3^3 s_3^3)=S(X; r^3 s^3 r_3^3 s_3^3) = |\mathcal{F}_3(X)| \prod_{p \mid r s r_3 s_3} \alpha_p + O \left( 3^{\omega(r s r_3  s_3)} X^{\frac{1}{2} + \epsilon} \right),
	\end{align}
	where the factors $\alpha_p$ in the Euler product are defined in \eqref{eq:alphas} and depend on $p \mod 3$.
	We can now write (for $r_1, r_2, r_3$ square-free and coprime)
	\begin{align*}
	\sum_{\substack{r_4 \geq1 \\ p \mid r_4 \Rightarrow p \mid r_1 r_2 r_3}}  \frac{d_z(r_1 r_2^2 r_3^3 r_4^3)}{r_1 r_2^2 r_3^3 r_4^3} &= 
	\prod_{p \mid r_1} \sum_{k =0}^{\infty} \frac{d_z(p^{3k+1})}{p^{3k+1}}
	\prod_{p \mid r_2}  \sum_{k =0}^{\infty} \frac{d_z(p^{3k+2})}{p^{3k+2}}
	\prod_{p \mid r_3}   \sum_{k =0}^{\infty} \frac{d_z(p^{3k+3})}{p^{3k+3}} \\
	&= \prod_{p \mid r_1} c_{p, \omega_3^2} (z) \prod_{p \mid r_2} c_{p, \omega_3} (z)  \prod_{p \mid r_3} \left( c_{p,1} (z) - 1 \right),
	\end{align*}
	where
	\begin{align*}
	c_{p,\omega_3^j}(z) &= \frac{\left(1-\frac{1}{p}\right)^{-z}+\omega_3^j\left(1-\frac{\omega_3}{p}\right)^{-z}+\omega_3^{2j}\left(1-\frac{\omega_3^2}{p}\right)^{-z}}{3}. 
	\end{align*}
	We remark that $c_{p,\omega_3^j}(z) \in \R$ since $z \in \R$.

A similar treatment applies to the sum over  $s_4$. Replacing in \eqref{eq:sieve2}, the contribution of the main term of \eqref{eq-5/8} to \eqref{sum-chi}  
   is
	\begin{align*}
& =|\mathcal{F}_3(X)|\sum_{r \in \mathcal{S}(y)}\sum_{\substack{s \in \mathcal{S}(y)\\ (r, s)=1}}\sum_{\substack{r_3, s_3 \in \mathcal{S}(y)\\ (r_3, rs)=1\\ (s_3, rs)=1}}\mu^2(r)\mu^2(s)\mu^2(r_3)\mu^2(s_3)\
\prod_{p \mid r s r_3 s_3} \alpha_p \prod_{p\mid r}c^2_{p, \omega_3^2}(z)\prod_{p\mid s}c^2_{p, \omega_3}(z)\\
& \hspace{1cm} \times \prod_{p\mid r_3}\left(c_{p, 1}(z)-1\right)\prod_{p\mid s_3}\left(c_{p, 1}(z)-1\right)\\
& =|\mathcal{F}_3(X)|  \prod_{p \leq y} \Bigg( 1 + \sum_{\substack{k_1, k_2, k_{3}, k_{4} \in \left\{ 0,1 \right\} \\  (k_1, k_2, k_3, k_4) \neq (0,0,0,0) \\ k_1 = 1 \Rightarrow k_2, k_3, k_4 = 0 \\ \\  k_2 = 1 \Rightarrow k_1, k_3, k_4 = 0 }}
  \alpha_p c_{p, \omega_3^2}(z)^{2k_1} c_{p, \omega_3}(z)^{2k_2} \left(  c_{p,1}(z)-1 \right)^{k_3}  \left(c_{p,1}(z) -1 \right)^{k_4} \Bigg) \\
& =|\mathcal{F}_3(X)| \prod_{\substack{p\leq y}}\left(1+\alpha_p c^2_{p, \omega_3^2}(z)+\alpha_p c^2_{p, \omega_3}(z) +2\alpha_p (c_{p, 1}(z)-1)+\alpha_p (c_{p, 1}(z)-1)^2\right)\\
&=|\mathcal{F}_3(X)| \prod_{\substack{p\leq y}}\left(1-\alpha_p +\alpha_p\left(c^2_{p, 1}(z)+c^2_{p, \omega_3}(z)+c^2_{p, \omega_3^2}(z)\right)\right)\\
&= |\mathcal{F}_3(X)|	\mathbb{E}(|L(1, \mathbb{X}; y)|^{2z}),
	\end{align*}
      where
in the last equality, we have applied the fact that 
	\begin{align*}
	\mathbb{E}(|L(1, \mathbb{X}; y)|^{2z})&=\prod_{p\leq y} \mathbb{E}\left(\left|1 - \frac{\mathbb{X}(p)}{p} \right|^{-2z} \right) =\prod_{p\leq y} \mathbb{E}\left(\left|1 - \frac{\mathbb{X}(p) + \overline{\mathbb{X}}(p)}{p} +  \frac{\mathbb{X}(p)  \overline{\mathbb{X}}(p)}{p^2}\right|^{-z} \right)\\
	&= \prod_{p\leq y} \left(1-\alpha_p + \frac{\alpha_p}{3} \left( 1-\frac{2}{p} + \frac{1}{p^2}\right)^{-z}  + \frac{2 \alpha_p}{3} \left(1+\frac{1}{p} + \frac{1}{p^2} \right)^{-z}\right)
	\end{align*}
	and
	\begin{align} \label{useful}
	 c^2_{p, 1}(z)+c^2_{p, \omega_3}(z)+c^2_{p, \omega_3^2}(z)=\frac13 \left(1-\frac{2}{p}+\frac{1}{p^2}\right)^{-z}+\frac23 \left(1+\frac{1}{p}+\frac{1}{p^2}\right)^{-z}.
	\end{align}
	Now replacing the error term of \eqref{eq-5/8} in \eqref{eq:sieve2}, and using \eqref{useful}, the contribution to \eqref{sum-chi}  is
\begin{align*}
&\ll X^{\frac{1}{2}+\epsilon} \prod_{p\leq y} \left( c_{p,1}^2(z)+c_{p, \omega_3^2}^2(z)+c_{p, \omega_3}^2(z)\right) \\
&= X^{\frac{1}{2}+\epsilon}  \mathbb{E}(|L(1, \mathbb{X}; y)|^{2z}) \prod_{p\leq y} \left( 1 + O \left( \frac{2}{p} \right) \right) 
\ll X^{\frac{1}{2}+\epsilon} \log^2{y} \;\mathbb{E}(|L(1, \mathbb{X}; y)|^{2z}).
\end{align*}
We now proceed to bound the contribution
when $m = n_1 n_2^2 = r_1s_1^2r_2^2s_2 r_3^3 s_3^3 \neq  \cube$ in \eqref{eq:sieve2}.
 By Lemma \ref{character sum under GRH}, we have
 \begin{align} \label{bound-char-sum}
|S(X; m )|\ll X^{\frac{2}{3}+\epsilon}\exp\left( (\log{m})^{\frac23 - \epsilon} \right) \ll  X^{\frac{3}{4}} m^{\frac{5}{\sqrt{\log{X}}}},  
\end{align}
where the second bound follows by considering the two cases, namely,
 $\log m \leq  \left(\frac{\log X}{25}\right)^{\frac32}$ and
 $\log m >  \left(\frac{\log X}{25}\right)^{\frac32}$. In the first case,
 $$
 \exp\left( (\log{m})^{\frac23 - \epsilon} \right) \leq \exp \left( \frac{\log{X}}{25} \right) = X^{\frac{1}{25}}.
 $$
  In the second case,
 $$
 \exp\left( (\log{m})^{\frac23} \right) =   \exp\left( \frac{\log{m}}{(\log{m})^{\frac13}} \right) \leq \exp\left(\frac{5\log m}{\sqrt{\log X}}\right) = m^{\frac{5}{\sqrt{\log{X}}}} .
 $$
Replacing in  \eqref{bound-char-sum} with $m =  r_1s_1^2r_2^2s_2 r_3^3 s_3^3$ in \eqref{eq:sieve2}, the contribution when $n_1 n_2^2 \neq \cube$ is bounded by
	\begin{align*}
	& \ll X^\frac{3}{4}    \sum_{\substack{r_1, r_2, r_3\in \mathcal{S}(y)\\ (r_2, r_1)=1\\ (r_3, r_1 r_2)=1\\ r_1, r_2, r_3 \; \text{square-free}}}
	\sum_{\substack{s_1, s_2, s_3\in \mathcal{S}(y)\\ (s_2, s_1)=1\\ (s_3, s_1 s_2)=1\\  s_1, s_2, s_3 \; \text{square-free}}}
	\left| \prod_{p\mid r_1}(c_{p, \omega_3^2}(z) p^\frac{5}{\sqrt{\log X}})\prod_{p\mid r_2}(c_{p, \omega_3}(z)p^\frac{10}{\sqrt{\log X}})\prod_{p\mid r_3}((c_{p, 1}(z)-1)p^\frac{15}{\sqrt{\log X}}) \right. \\
& \times  \left .
 \prod_{p\mid s_1}(c_{p, \omega_3^2}(z)p^\frac{10}{\sqrt{\log X}})\prod_{p\mid s_2}(c_{p, \omega_3}(z)p^\frac{5}{\sqrt{\log X}})\prod_{p\mid s_3}((c_{p, 1}(z)-1)p^\frac{15}{\sqrt{\log X}})\right| \\
	  &= X^\frac{3}{4} \left| \prod_{p\leq y} \left( 1 + c_{p, \omega_3^2}(z)  p^\frac{5}{\sqrt{\log X}}+c_{p, \omega_3}(z) p^\frac{10}{\sqrt{\log X}}+(c_{p, 1}(z)-1)p^\frac{15}{\sqrt{\log X}}\right)\right.\\
 &  \times  \left. \left( 1 + c_{p, \omega_3^2}(z) p^\frac{10}{\sqrt{\log X}}+c_{p, \omega_3}(z) p^\frac{5}{\sqrt{\log X}}+(c_{p, 1}(z)-1)p^\frac{15}{\sqrt{\log X}}\right)\right| \\
   &\ll X^\frac{3}{4}  \prod_{p\leq y} \left( 1 + |c_{p, \omega_3^2}(z)|  p^\frac{15}{\sqrt{\log X}}+|c_{p, \omega_3}(z)| p^\frac{15}{\sqrt{\log X}}+|c_{p, 1}(z)-1|p^\frac{15}{\sqrt{\log X}}\right)^2.
 \end{align*}
Now we split the product into the primes that are $\leq 4(|z|+1)$ and those that are  $> 4(|z|+1)$. 
When $p\leq 4(|z|+1)$, we use the bound (which is valid for all $p$)
 \begin{align*}
 |c_{p,\omega_3^j}(z)| \leq&\frac{1}{3} \left(\left(\left|1-\frac{1}{p}\right|^2\right)^{-\frac{z}{2}} +\left(\left|1-\frac{\omega_3 }{p}\right|^2\right)^{-\frac{z}{2}}+\left(\left|1-\frac{\omega_3^2 }{p}\right|^2\right)^{-\frac{z}{2}}\right)\nonumber\\
= &\frac{1}{3}\left(1-\frac{2}{p}+\frac{1}{p^2}\right)^{-\frac{z}{2}}+\frac{2}{3}\left(1+\frac{1}{p}+\frac{1}{p^2}\right)^{-\frac{z}{2}}.
\end{align*}
                        
When  $p\geq 4(|z|+1)$ we use the bounds 
	 \begin{align*}
  	|c_{p,\omega_3^2}(z)|\leq\frac{4|z|}{3p}, \quad |c_{p,\omega_3}(z)|\leq\frac{4|z|^2}{3p^2}, \quad |c_{1,p}(z)-1|\leq \frac{4|z|^3}{3p^3}.
  	\end{align*}
We have 
 \begin{align*}
   &\ll X^\frac{3}{4}  \prod_{p\leq \min\{ y, 4(|z|+1)\}} \left( 1 + |c_{p, \omega_3^2}(z)|  p^\frac{15}{\sqrt{\log X}}+|c_{p, \omega_3}(z)| p^\frac{15}{\sqrt{\log X}}+|c_{p, 1}(z)-1|p^\frac{15}{\sqrt{\log X}}\right)^2\\
 &\times \prod_{\min\{ y, 4(|z|+1)\}<p<y} \left( 1 + |c_{p, \omega_3^2}(z)|  p^\frac{15}{\sqrt{\log X}}+|c_{p, \omega_3}(z)| p^\frac{15}{\sqrt{\log X}}+|c_{p, 1}(z)-1|p^\frac{15}{\sqrt{\log X}}\right)^2\\
  &\ll X^\frac{3}{4}  \prod_{p\leq \min\{ y, 4(|z|+1)\}} p^\frac{30}{\sqrt{\log X}}\left(2+\left(1-\frac{2}{p}+\frac{1}{p^2}\right)^{-\frac{z}{2}}+2\left(1+\frac{1}{p}+\frac{1}{p^2}\right)^{-\frac{z}{2}}\right)^2\\&\times \prod_{\min\{ y, 4(|z|+1)\}<p<y}\left(1+\frac{6p^\frac{30}{\sqrt{\log X}}|z|}{p}\right).
 \end{align*}
 
 Now we use that $1\leq \left(1-\frac{2}{p}+\frac{1}{p^2}\right)^{-\frac{z}{2}}+\left(1+\frac{1}{p}+\frac{1}{p^2}\right)^{-\frac{z}{2}}$ to get
 \begin{align*}
 & \left(2+\left(1-\frac{2}{p}+\frac{1}{p^2}\right)^{-\frac{z}{2}} +2\left(1+\frac{1}{p}+\frac{1}{p^2}\right)^{-\frac{z}{2}}\right)^2\\
 & \leq  54 \left(1-\alpha_p + \frac{\alpha_p}{3} \left( 1-\frac{2}{p} + \frac{1}{p^2}\right)^{-z}  + \frac{2 \alpha_p}{3} \left(1+\frac{1}{p} + \frac{1}{p^2} \right)^{-z}\right) \left(1+O\left(\frac{2}{p}\right)\right).
 \end{align*}
 
 Then we get for $4(|z|+1)<y$,
\begin{align*}
  &\ll X^\frac{3}{4} \mathbb{E}(|L(1,\mathbb{X};y)|^{2z}) \prod_{p\leq4(|z|+1)} 54p^\frac{30}{\sqrt{\log X}}\left(1+O\left(\frac{2}{p}\right)\right)\prod_{4(|z|+1)<p<y}\left(1+\frac{6p^\frac{30}{\sqrt{\log X}}|z|}{p}\right)\\
  &\ll X^\frac{3}{4} \mathbb{E}(|L(1,\mathbb{X};y)|^{2z}) \log^2 y \prod_{p\leq 4(|z|+1)} 54 p^\frac{30}{\sqrt{\log X}} \prod_{4(|z|+1)<p<y}\left(1+\frac{6p^\frac{30}{\sqrt{\log X}}|z|}{p}\right)\\
  &\ll X^\frac{3}{4} \mathbb{E}(|L(1,\mathbb{X};y)|^{2z})
 \exp \left(\frac{16(|z|+1)}{\log(4(|z|+1))} +\frac{120(|z|+1)}{\sqrt{\log X}}+ 2 \log_2 y + 6 e^{30} |z|\log \left( \frac{\log y}{\log (4(|z|+1))}\right)\right).
 \end{align*}
 For $y\leq 4(|z|+1)$, we have 
 \begin{align*}
   &\ll X^\frac{3}{4} \mathbb{E}(|L(1,\mathbb{X};y)|^{2z}) \exp\left(\frac{34 y}{\log y} \right),
 \end{align*}
since $\log y \leq \sqrt{\log X}$. This finishes the proof.

	\end{proof}

\begin{proof}[Proof of Theorem   \ref{asymp moment under GRH}.]
 We directly use Proposition \ref{moment of SEP} with $y = B \log{X} \log_2{X}$ and $z=\frac{\log{X} \log_2{X}}{e^{37} \log{B}}$, where $e^{40} \leq B \leq (\log_2{X})^C$ to  prove the statement. Indeed,  from the choice of $y$ and $z$, we see that $y\geq 4(|z|+1)$. Therefore, applying the first part of Proposition \ref{moment of SEP}, we have
 	\begin{align*}
 	E(z, y, X)&\ll X^{-\frac{1}{4}}\exp\left(\frac{16 \log X \log_2 X}{e^{37}\log_2 X \log B}+\frac{120 \log X \log_2 X}{e^{37}\log B \sqrt{\log X}}+2\log_2 X\right.\\
 	&\left. +\frac{6e^{30}\log X \log_2 X}{e^{37}\log B}\log \frac{\log_2 X+\log_3 X +\log B}{\log_2 X+\log_3 X-37-\log_2 B+\log 4}\right).
 	\end{align*}
 Now notice that 
	\begin{align*}
 	\log \frac{\log_2 X +\log_3 X +\log B}{\log_2 X+\log_3 X-37-\log_2 B+\log 4}
 	&=\log\left( \left(1+\frac{\log B}{\log_2 X+\log_3 X}\right)\left(1-\frac{37-\log 4+\log_2 B}{\log_2 X +\log_3 X}\right)^{-1}\right)\\
 	&\leq \frac{2\log B}{\log_2 X},
 	\end{align*}
where for the last inequality we applied the condition $37-\log 4+\log_2 B<\log B$, which is true provided that $B\geq e^{40}$.   	
 	
Incorporating this observation, we obtain 	
 	\begin{align*}
 	E(z, y, X)	&\ll X^{-\frac{1}{4}}\exp\left(\frac{17 \log X \log_2 X}{e^{37}\log_2 X \log B}+\frac{6e^{30} \log X \log_2 X}{e^{37}\log B} \frac{2\log B}{\log_2 X}\right)\\
 	&\ll X^{-\frac{1}{4}} \exp\left(\frac{17 \log X}{e^{37}\log B}+\frac{12}{e^{7}}\log X\right)\ll X^{-\frac{23}{100}}.
 	\end{align*}

\end{proof}
We now proceed to establish a connection between $L(1, \chi)$ and $L(1, \chi; y)$ under GRH. 
\begin{prop}\label{l function via SEP}
Let $C > 0$, and let $2 \leq A \leq (\log_2 X)^{C}$ be a real number and $y:= A^4\log X \log_2 X$.	Under GRH, we have 
	\[
	L(1, \chi)=L(1, \chi; y)\left(1+O\left(\frac{1}{A\log_2 X}\right)\right)
	\]
	for all but at most $X^{\frac{14}{15}}$ cubic characters in $\mathcal{F}_3(X)$.
	 \end{prop} 
 To prove this statement we need the following auxiliary result. 
 \begin{lem} Let $k$ be a positive integer.\label{moments over primes} 
 	Under GRH,
 	\begin{align*}
 	\sum_{\chi \in \mathcal{F}_3(X)}\bigg|\sum_{y< p\leq z}\frac{\chi(p)}{p}\bigg|^{2k}&\ll X2^k k! \left(\sum_{y<p\leq z}\frac{1}{p^2}\right)^{k}+X^{2/3+\epsilon}\exp\left((3k\log z)^{\frac{2}{3}}\right)\left(\frac{\log (z/y)}{\log y}\right)^{2k}\\
 	& \ll X\left(\frac{2k}{y\log y}\right)^k +X^{2/3+\epsilon}\exp\left((3k\log z)^{\frac{2}{3}}\right)\left(\frac{\log (z/y)}{\log y}\right)^{2k}.
 	\end{align*}
 \end{lem}
 \begin{proof}
 	We write 
 	\[
 	\left(\sum_{y< p\leq z}\frac{\chi(p)}{p}\right)^k=\sum_{y^k<m\leq z^k}\frac{a_k(m)\chi(m)}{m},
 	\]
 	where
 	\[
 	a_k(m)=\sum_{\substack{y<p_1, \dots, p_k\leq z\\ p_1\cdots p_k=m}}1.
 	\]
 	This gives us
 	\[
 	\sum_{\chi \in \mathcal{F}_3(X)}\bigg|\sum_{y< p\leq z}\frac{\chi(p)}{p}\bigg|^{2k}=\sum_{\substack{y^k<m_1, m_2\leq z^k}}\frac{a_k(m_1)a_k(m_2)}{m_1m_2}\sum_{\chi \in \mathcal{F}_3(X)}\chi(m_1 m_2^2).
 	\]
 	Separating the above sum into the cubic and the non-cubic part,  and applying Lemma \ref{character sum under GRH}, we obtain the following bound
 	\begin{align*}
 	&\ll X\sum_{\substack{y^k<m_1, m_2\leq z^k\\ m_1m_2^2=\tinycube}}\frac{a_k(m_1)a_k(m_2)}{m_1m_2}+X^{2/3+\epsilon}\exp\left((3k\log z)^{\frac{2}{3}}\right)\left(\sum_{y<p\leq z}\frac{1}{p}\right)^{2k}\\
 	&\ll X\sum_{\substack{y^k<m_1, m_2\leq z^k\\ m_1m_2^2=\tinycube}}\frac{a_k(m_1)a_k(m_2)}{m_1m_2}+X^{2/3+\epsilon}\exp\left((3k\log z)^{\frac{2}{3}}\right)\left(\frac{\log (z/y)}{\log y}\right)^{2k}.
 	\end{align*}
 	We notice the following facts. For $m=p_1^{e_1}\cdots p_k^{e_k}$ 
 	\begin{equation}\label{eq:ak}
 	a_k(m)=\binom{k}{e_1,\dots,e_r}.
 	\end{equation}
 	For $m_1, m_2$ with $\Omega(m_i)=k_i$, we have  
 	\begin{equation}\label{eq:akk}
 	a_{k_1+k_2}(m_1m_2)\leq \binom{k_1+k_2}{k_2} a_{k_1}(m_1)a_{k_2}(m_2).
 	\end{equation}
 	
 	Now for $m_1m_2^2=\cube$, write $m=(m_1,m_2)$, and $m_1=mm_{1,1}$, $m_2=mm_{2,1}$. We have that $m_{1,1}m_{2,1}^2=\cube$, with $(m_{1,1},m_{2,1})=1$, and this implies $m_{1,1}=m_{1,2}^3$, $m_{2,1}=m_{2,2}^3$.  Setting $\ell=\Omega(m)$ and combining \eqref{eq:ak} and \eqref{eq:akk} several times, one obtains
 	\[a_k(m_1)a_k(m_2)\leq k! \binom{k}{3\ell} a_\ell(m_{1,2})a_\ell(m_{2,2}) a_{k-3\ell}(m),\]
 	and from there one can deduce that
 	\[
 	\sum_{\substack{y^k<m_1, m_2\leq z^k\\ m_1m_2^2=\tinycube}}\frac{a_k(m_1)a_k(m_2)}{m_1m_2}\ll 2^kk!  \left(\sum_{y<p\leq z}\frac{1}{p^2}\right)^{k}.
 	\]
 \end{proof}
\begin{proof}[Proof of Proposition \ref{l function via SEP}]
    Specializing by $Q=X$, $y=(\log X)^{32}$, and $A=4$ in Proposition 2.2 of \cite{GS}, we can say that 
	\[
	L(1, \chi)=L(1, \chi; (\log X)^{32})\left(1+O\left(\frac{1}{\log X}\right)\right)
	\]
	holds for all but at most $X^{\frac{1}{2}}$ characters inside $\mathcal{F}_3(X)$.
	Observe  that for $z>y$,  
	\[
	L(1, \chi; z)=L(1, \chi; y)\exp\left(\sum_{y< p\leq z}\left(\frac{\chi(p)}{p}+O\left(\frac{1}{p^2}\right)\right)\right).
	\]
We take $z=	(\log X)^{32}$ and $y:= A^4\log X \log_2 X$, and it suffices to show that 
\[
\bigg|\sum_{y< p\leq z}\frac{\chi(p)}{p}\bigg|\leq \frac{1}{A\log_2 X}
\]
holds for all but $\ll X^{\frac{14}{15}}$ characters inside $\mathcal{F}_3(X)$. To do this we dyadically divide the interval $(y, z]$ into $J$ subintervals of the form $(y_j,y_{j+1}]$ such that  $y_j:=2^{j-1}y$, where 
$J=[32\log_2 X-\log y]$. (We need this to make the quantity $\log (z/y)$ small in Lemma \ref{moments over primes}.)

On the basis of this division, using a version of Chebyshev's inequality, we have
\begin{align}
\#\Big\{\chi \in \mathcal{F}_3(X)\, :\, \bigg|\sum_{y_j< p\leq y_{j+1}}\frac{\chi(p)}{p}\bigg|\geq \frac{1}{A2^{\frac{j}{6}-1}\log_2 X}\Big\} \nonumber\\
\leq \sum_{\chi \in \mathcal{F}_3(X)}\bigg|\sum_{y_j< p\leq y_{j+1}}\frac{\chi(p)}{p}\bigg|^{2k_j}\times (A2^{\frac{j}{6}-1}\log_2 X)^{2k_j}, \label{prob. bound}
\end{align}
where we take 
\begin{align} \label{kj} k_j =\frac{\log X}{30 \log (A2^{j/3+1})}. \end{align}

We now apply Lemma \ref{moments over primes} to bound \eqref{prob. bound}.
From the first term in the bound of Lemma \ref{moments over primes}, we get the bound
\begin{align}
&\ll X \left(\frac{2k_j}{y_j \log y_j}\right)^{k_j}(A2^{\frac{j}{6}-1}\log_2 X)^{2k_j}\nonumber\\
& \ll X \exp\left(k_j\left(-\log 30-2\log A-2j/3\log 2\right)\right)\label{eq:arghhhh}\\
& \ll X \exp\left(- 2k_j \log (5A2^{j/3}) \right)\ll X^{\frac{14}{15}},\nonumber
\end{align}
by \eqref{kj} and using the fact\footnote{In the estimate \eqref{eq:arghhhh} we particularly use that $y=A^4\log X \log_2X$ as opposed to the perhaps more natural choice $y=A\log X \log_2X$.} that $y_j=2^{j-1} A^4\log X \log_2 X$. Similarly, from the second term in the bound of Lemma \ref{moments over primes}, we get the bound
 \begin{align*}
 &\ll X^{2/3+\epsilon}\exp((3k_j \log y_{j+1})^{\frac{2}{3}})\left(\frac{\log (y_{j+1}/y_j)}{\log y_j}\right)^{2k_j} (A2^{\frac{j}{6}-1}\log_2 X)^{2k_j}\\
 &\ll X^{2/3+2\epsilon}\exp(2k_j \log(2^{j/6-1}A)+3k_j^{\frac{2}{3}}(\log_2 X)^{\frac{2}{3}} )) \ll X^{\frac{4}{5}}.
 \end{align*}
 Therefore we have that 
 \[
\sum_{j=1}^{J} \bigg|\sum_{y_j< p\leq z_j}\frac{\chi(p)}{p}\bigg|\leq \frac{1}{A\log_2 X}
 \]
holds for all but at most $X^{\frac{14}{15}}$ characters inside $\mathcal{F}_3(X)$.
	\end{proof}

\begin{prop}
For any $r, y>0$, define
	\begin{align*}
	\mathcal{L}(r; y) &:=\log(\mathbb{E}(|L(1,\mathbb{X}; y)|^{2r}))=\sum_{p\leq y}\log E_p(r),\\ 
	\widetilde{\mathcal{L}}(r; y) &:=\log(\mathbb{E}(|L(1,\mathbb{X}; y)|^{-2r}))=\sum_{p\leq  y}\log E_p(-r).
	\end{align*}
Then
	\begin{align}\label{expectation at y}
	\mathcal{L}(r; y)=2r \log \log \min\{r, y\}+2r \gamma+\frac{2r \left(C_{\max}(r/y)-1\right)}{\log r}+O\left(\frac{r}{(\log r)^2}\right)
	\end{align} 
	and \[
\widetilde{\mathcal{L}}(r; y)= \log \log \min\{r, y\}+r\gamma+\frac{r \left( C_{\min}(r/y)-1\right)}{\log r}-r \log\zeta(3)+O\left(\frac{r}{(\log r)^2}\right),
	\]
	where $2C_{\max}(r/y)=\int_{r/y}^{\infty}\frac{f'(u)}{u}du$ and  $C_{\min}(r/y)=\int_{r/y}^{\infty}\frac{\widetilde{f}'(u)}{u}du$ and $f(u)$ and $\widetilde{f}'(u)$ are defined by \eqref{f(t)} and  \eqref{f(t) hat} respectively. 
	
	Moreover, $2C_{\max}$ and $C_{\min}$ satisfy
	\begin{align}\label{tau and kappa relation at y}
	2C_{\max}(r/y)=2C_{\max}+O(r/y) \quad \text{and} \quad C_{\min}(r/y)=C_{\min}+O(r/y).
	\end{align}
\end{prop}	
\begin{proof}This is a direct adaptation of the proof of Proposition \ref{L-est-prop}.\end{proof}

	\begin{proof}[Proof of Theorem \ref{main distribution result under GRH}] The proof is almost identical to the proof of Theorem \ref{Asympformula} from Section \ref{proofof-1.2-1.3}, replacing $L(1, \chi)$ by the short Euler product $L(1, \chi; y)$, and keeping track of $y$ in the error terms.  We define for any $r \geq 1$
	\begin{align*}
	\phi_X(\tau; y) &:=\frac{1}{|\mathcal{F}_3(X)|}\sum_{\substack{\chi\in \mathcal{F}_3(X)\\ |L(1, \chi; y)|\geq e^{\gamma}\tau}}1, \end{align*} which gives
	\begin{align*} 2r\int_{0}^{\infty}t^{2r-1}\phi_X(t; y)dt
	&=\frac{e^{-2r\gamma}}{|\mathcal{F}_3(X)|}\sum_{\chi\in \mathcal{F}_3(X)}|L(1, \chi; y)|^{2r}.
	\end{align*}
	We apply Theorem \ref{asymp moment under GRH} for $B=A^4$ and $e^{10}\leq A\leq (\log_2 X)^C$ to get 
	\begin{align*}
&2r \int_{0}^{\infty}t^{2r-1}\phi_X(t; y)dt=e^{-2\gamma r} E\left(|L(1, \X; y)|^{2r}\right)+O\left(e^{-2\gamma r}X^{-\frac{23}{100}}  E\left(|L(1, \X; y)|^{2r}\right)\right),
\end{align*}
	uniformly for $r\leq \frac{\log X \log_2 X}{4 e^{37} \log A}$ and $y=A^4 \log X \log_2 X$ (notice that $r< y$). Using \eqref{expectation at y} and \eqref{tau and kappa relation at y},
	we get
	\begin{align}\label{asymp for kappa2-short}
	\int_{0}^{\infty}t^{2r-1}\phi_X(t; y)dt 
	&=(\log r)^{2r}\exp\left(\frac{2r}{\log r}(C_{\max}-1)+O\left(\frac{r}{(\log r)^2}\right)+O\left(\frac{r^2}{y \log r}\right)\right).
	\end{align}
	Let  $r_2=r e^{\delta}$, where $\delta>0$ is sufficiently small to be determined later. We apply  \eqref{asymp for kappa2-short} with $r_2$ to obtain  
	\begin{align*}
	\int_{\tau+\delta}^{\infty}t^{2r-1}\phi_X(t; y)dt &\leq
	 (\log{r} )^{2 r} \exp\left(2r(1-e^{\delta})\log \left(1+\frac{\delta}{\tau}\right)+2r e^{\delta}\log \left(1+\frac{\delta}{\log r}\right)\right) \\
	& \hspace{-1.5cm}  \times \exp\left(\frac{2r e^{\delta}}{\log r}\left(C_{\max}-1\right)+O\left(\frac{r}{(\log r)^2}\right)+O\left(\frac{r^2}{y\log r}\right)\right)  \exp \left( 2r (1-e^\delta) \left( \log{\tau} - \log_2{r} \right) \right), \end{align*}
	which holds uniformly for 
	\begin{align*} 
	r_2 \leq \frac{\log X \log_2 X}{4 e^{37} \log A} \iff r\leq  e^{-\delta}\frac{\log X \log_2 X}{4 e^{37} \log A}.
	\end{align*}
We now choose $r = r(\tau)$ by $\log{r} = \tau - C_{\max}$, and $\delta=\frac{c}{\sqrt{\log r}}$ 
	to get
	\begin{align*}
	\int_{\tau+\delta}^{\infty}t^{2r-1}\phi_X(t; y)dt& \leq (\log r)^{2r}\exp\left(\frac{2r}{\log r}\left(C_{\max}-1\right)+O\left(\frac{r}{(\log r)^2}\right)+O\left(\frac{r^2}{y\log r}\right)\right) \exp \left( {-\frac{c^2r}{(\log r)^2}} \right)		\end{align*}
	for some  constant $c$. 
	 Using \eqref{asymp for kappa2-short}, we have 
	\begin{align*}
	\int_{\tau+\delta}^{\infty}t^{2r-1}\phi_X(t; y)dt &\leq \left(\int_{0}^{\infty}t^{2r-1}\phi_X(t; y)dt\right) \exp \left( {-\frac{c^2r}{(\log r)^2}} \right),
	\end{align*}
	and similarly, we can prove that
	\begin{align*}
	\int_0^{\tau-\delta}t^{2r-1}\phi_X(t; y)dt\leq \left(\int_{0}^{\infty}t^{2r-1}\phi_X(t; y)dt\right) \exp \left( {-\frac{c^2r}{(\log r)^2}} \right).
	\end{align*}
Working as in the proof of Theorem \ref{Asympformula}, we get
	\begin{align*}
	\phi_X(\tau+\delta; y)\leq
	\exp\left(-\frac{2e^{\tau-C_{\max}}}{\tau}\left(1+O\left(\delta\right)+O\left(\frac{e^{\tau}}{y}\right)\right)\right)
	\leq \phi_X(\tau-\delta; y),
	\end{align*}
	and
	\begin{align}\label{distribution for phi-short}
	\phi_{X}(\tau; y)=\exp\left(-\frac{2e^{\tau-C_{\max}}}{\tau}\left(1+O\left(\frac{1}{\tau^{\frac{1}{2}}}\right)+O\left(\frac{e^{\tau}}{y}\right)\right)\right)
	\end{align}
uniformly in the range
	\begin{align*}
	\tau&=\log r+C_{\max}\leq \log \left( e^{-\delta}\frac{\log X \log_2 X}{4 e^{37} \log A}
	\right) + C_{\max}.
	\end{align*}
To conclude the proof, we know replace $\phi_X(\tau; y)$ by $\phi_X(\tau)$, with the appropriate error terms. 
Using Proposition \ref{l function via SEP}, we can write for $y=A^4 \log X \log_2 X$ and $2< e^{10}\leq A\leq (\log_2 X)^C$,
	\begin{align*}
	\phi_{X}(\tau)=\frac{1}{|\mathcal{F}_3(X)|}\sum_{\substack{\chi\in \mathcal{F}_3(X)\\ |L(1, \chi)|\geq e^{\gamma}\tau}}1&=\frac{1}{|\mathcal{F}_3(X)|}\sum_{\substack{\chi\in \mathcal{F}_3(X)\\ |L(1, \chi; y)|\geq e^{\gamma}\tau \left(1+O\left(\frac{1}{A\log_2 X}\right)\right)}}1+O(X^{-\frac{1}{15}})\\
	&= \phi_{X}\left(\tau\left(1+O\left(\frac{1}{A\log_2 X}\right)\right); y\right)+O(X^{-\frac{1}{15}}),
	\end{align*}
	and from \eqref{distribution for phi-short}, we have
	\begin{align*}
	&\phi_{X}\left(\tau\left(1+O\left(\frac{1}{A\log_2 X}\right)\right); y\right)\\
	&=\exp\left(-\frac{2e^{\tau-C_{\max}+ O(\tau/(A \log_2{X}))}}{\tau}\left(1+O\left(\frac{1}{\tau^{\frac{1}{2}}}\right)+O\left(\frac{1}{A\log_2 X}\right)+O\left(\frac{e^{\tau}}{y}\right)\right)\right)\\
	&=\exp\left(-\frac{2e^{\tau-C_{\max}}}{\tau}\left(1+O\left(\frac{1}{\tau^{\frac{1}{2}}}\right)+O\left(\frac{1}{A}\right)+O\left(\frac{e^{\tau}}{y}\right)\right)\right).
	\end{align*}
	 Hence, we conclude that
	 \begin{align*}
	 \phi_{X}(\tau)&=\exp\left(-\frac{2e^{\tau-C_{\max}}}{\tau}\left(1+O\left(\frac{1}{\tau^{\frac{1}{2}}}\right)+O\left(\frac{1}{A}\right)\right)\right),
	 \end{align*}
	since $\frac{e^{\tau}}{y}\ll \frac{1}{A^4 \log A}$.
	 This completes the proof for $\phi_X(\tau)$. A similar estimate holds for  $\psi_{X}(\tau)$. 
	
\end{proof}	
	
\bibliographystyle{amsalpha}
\bibliography{Bibliography2}

\providecommand{\bysame}{\leavevmode\hbox to3em{\hrulefill}\thinspace}
\providecommand{\MR}{\relax\ifhmode\unskip\space\fi MR }
\providecommand{\MRhref}[2]{%
  \href{http://www.ams.org/mathscinet-getitem?mr=#1}{#2}
}
\providecommand{\href}[2]{#2}
\begin{thebibliography}{AMMP19}

\bibitem[AH20]{AkbaryHamieh}
Amir Akbary and Alia Hamieh, \emph{Value-distribution of cubic {H}ecke
  {$L$}-functions}, J. Number Theory \textbf{206} (2020), 81--122. \MR{4013165}

\bibitem[AH21]{AkbaryHamieh2}
\bysame, \emph{Two dimensional value-distribution of cubic {H}ecke
  {$L$}-functions}, Proc. Amer. Math. Soc. \textbf{149} (2021), no.~11,
  4669--4684. \MR{4310094}

\bibitem[AMM19]{AMM2019}
Christoph Aistleitner, Kamalakshya Mahatab, and Marc Munsch, \emph{Extreme
  values of the {R}iemann zeta function on the 1-line}, Int. Math. Res. Not.
  IMRN (2019), no.~22, 6924--6932. \MR{4032179}

\bibitem[AMMP19]{AMMP2019}
Christoph Aistleitner, Kamalakshya Mahatab, Marc Munsch, and Alexandre Peyrot,
  \emph{On large values of ${L}(\sigma, \chi)$}, Q. J. Math. \textbf{70}
  (2019), no.~3, 831--848. \MR{4009474}

\bibitem[BDFL10]{BDFL}
Alina Bucur, Chantal David, Brooke Feigon, and Matilde Lal\'in,
  \emph{Statistics for traces of cyclic trigonal curves over finite fields},
  Int. Math. Res. Not. IMRN (2010), no.~5, 932--967. \MR{2595014}

\bibitem[BS17]{BS2017}
Andriy Bondarenko and Kristian Seip, \emph{Large greatest common divisor sums
  and extreme values of the {R}iemann zeta function}, Duke Math. J.
  \textbf{166} (2017), no.~9, 1685--1701. \MR{3662441}

\bibitem[BY10]{BY}
Stephan Baier and Matthew~P. Young, \emph{Mean values with cubic characters},
  J. Number Theory \textbf{130} (2010), no.~4, 879--903. \MR{2600408}

\bibitem[Cho49]{chowla48}
S.~Chowla, \emph{Improvement of a theorem of {L}innik and {W}alfisz}, Proc.
  London Math. Soc. (2) \textbf{50} (1949), 423--429. \MR{27302}

\bibitem[Coh54]{Cohn}
Harvey Cohn, \emph{The density of abelian cubic fields}, Proc. Amer. Math. Soc.
  \textbf{5} (1954), 476--477. \MR{0064076}

\bibitem[DL18]{DL}
Alexander Dahl and Youness Lamzouri, \emph{The distribution of class numbers in
  a special family of real quadratic fields}, Trans. Amer. Math. Soc.
  \textbf{370} (2018), no.~9, 6331--6356. \MR{3814332}

\bibitem[GL21]{GranvilleLamzouri}
Andrew Granville and Youness Lamzouri, \emph{Large deviations of sums of random
  variables}, Lith. Math. J. \textbf{61} (2021), no.~3, 345--372. \MR{4313597}

\bibitem[GR90]{GR}
S.~W. Graham and C.~J. Ringrose, \emph{Lower bounds for least quadratic
  nonresidues}, Analytic number theory ({A}llerton {P}ark, {IL}, 1989), Progr.
  Math., vol.~85, Birkh\"{a}user Boston, Boston, MA, 1990, pp.~269--309.
  \MR{1084186}

\bibitem[GS03]{GS}
A.~Granville and K.~Soundararajan, \emph{The distribution of values of
  {$L(1,\chi_d)$}}, Geom. Funct. Anal. \textbf{13} (2003), no.~5, 992--1028.
  \MR{2024414}

\bibitem[GS06]{GS06}
\bysame, \emph{{E}xtreme values of $|\zeta(1+it)|$. {T}he {R}iemann zeta
  function and related themes: papers in honour of {P}rofessor {K.}
  {R}amachandra}, Ramanujan Math. Soc. Lect. Notes Ser. \textbf{2} (2006),
  65--80. \MR{335187}

\bibitem[GS07]{GS-JAMS}
Andrew Granville and K.~Soundararajan, \emph{Large character sums: pretentious
  characters and the {P}\'{o}lya-{V}inogradov theorem}, J. Amer. Math. Soc.
  \textbf{20} (2007), no.~2, 357--384. \MR{2276774}

\bibitem[HB88]{HB88}
D.~R. Heath-Brown, \emph{The growth rate of the {D}edekind zeta-function on the
  critical line}, Acta Arith. \textbf{49} (1988), no.~4, 323--339. \MR{0937931}

\bibitem[Hin83]{Hinz}
J\"{u}rgen~G. Hinz, \emph{Character sums in algebraic number fields}, J. Number
  Theory \textbf{17} (1983), no.~1, 52--70. \MR{712968}

\bibitem[Lam10]{Lam2010}
Youness Lamzouri, \emph{Distribution of values of {L}-functions at the edge of
  the critical strip.}, Proc. Lond. Math. Soc. (3) \textbf{100} (2010), no.~3,
  835--863. \MR{2640292}

\bibitem[Lam11]{Lam2011}
\bysame, \emph{Extreme values of {${\rm arg}\,L(1,\chi)$}}, Acta Arith.
  \textbf{146} (2011), no.~4, 335--354. \MR{2747035}

\bibitem[Lam17]{Lam2017}
\bysame, \emph{Large values of {$L(1,\chi)$} for {$k$}th order characters
  {$\chi$} and applications to character sums}, Mathematika \textbf{63} (2017),
  no.~1, 53--71. \MR{3610005}

\bibitem[Lit28]{Littlewood-1928}
J.~E. Littlewood, \emph{On the {C}lass-{N}umber of the {C}orpus
  {$P({\surd}-k)$}}, Proc. London Math. Soc. (2) \textbf{27} (1928), no.~5,
  358--372. \MR{1575396}

\bibitem[LLS15]{LLS}
Youness Lamzouri, Xiannan Li, and Kannan Soundararajan, \emph{Conditional
  bounds for the least quadratic non-residue and related problems}, Math. Comp.
  \textbf{84} (2015), no.~295, 2391--2412. \MR{3356031}

\bibitem[Lum18]{Lumley-edge}
Allysa Lumley, \emph{Explicit bounds for {$L$}-functions on the edge of the
  critical strip}, J. Number Theory \textbf{188} (2018), 186--209. \MR{3778630}

\bibitem[Lum19]{lumley-1line}
\bysame, \emph{Complex moments and the distribution of values of
  {$L(1,\chi_D)$} over function fields with applications to class numbers},
  Mathematika \textbf{65} (2019), no.~2, 236--271. \MR{3884656}

\bibitem[Mon71]{Montgomery-lectures}
Hugh~L. Montgomery, \emph{Topics in multiplicative number theory}, Lecture
  Notes in Mathematics, Vol. 227, Springer-Verlag, Berlin-New York, 1971.
  \MR{0337847}

\bibitem[MV99]{MV}
H.~L. Montgomery and R.~C. Vaughan, \emph{Extreme values of {D}irichlet
  {$L$}-functions at {$1$}}, Number theory in progress, {V}ol. 2
  ({Z}akopane-{K}o\'{s}cielisko, 1997), de Gruyter, Berlin, 1999,
  pp.~1039--1052. \MR{1689558}

\bibitem[MV07]{MV-book}
Hugh~L. Montgomery and Robert~C. Vaughan, \emph{Multiplicative number theory.
  {I}. {C}lassical theory}, Cambridge Studies in Advanced Mathematics, vol.~97,
  Cambridge University Press, Cambridge, 2007. \MR{2378655}

\bibitem[RS62]{Rosser-Schoenfeld}
J.~Barkley Rosser and Lowell Schoenfeld, \emph{Approximate formulas for some
  functions of prime numbers}, Illinois J. Math. \textbf{6} (1962), 64--94.
  \MR{137689}

\end{thebibliography}
\end{document}